\documentclass[a4paper]{article}
\usepackage{lineno,hyperref}
\usepackage[english]{babel}
\usepackage{csquotes}
\usepackage{indentfirst}
\usepackage{fancyhdr}
\usepackage{graphicx,subcaption,float}
\usepackage{amssymb}
\usepackage{amsmath}
\usepackage{cases}
\usepackage{mathtools}
\usepackage{amsthm}
\usepackage{color}
\usepackage{epstopdf}
\usepackage{esint}

\usepackage[]{geometry}
\newtheorem{theorem}{Theorem}[section]
\newtheorem{definition}[theorem]{Definition}
\newtheorem{lemma}[theorem]{Lemma}
\newtheorem{corollary}[theorem]{Corollary}

\newtheorem{example}[theorem]{Example}
\newtheorem{proposition}[theorem]{Proposition}
\newtheorem{remark}[theorem]{Remark}

\newcommand{\changes}[1]{\textcolor{black}{#1}}

\usepackage{tikz}

\usetikzlibrary{positioning}
\usetikzlibrary{arrows}
\usetikzlibrary{arrows.meta}
\usetikzlibrary{calc}
\usetikzlibrary{shapes}
\usepackage{verbatim}
\makeatletter
\renewcommand*{\@fnsymbol}[1]{\ensuremath{\ifcase#1\or 1\or 2\or *\or 3\or 4\or 5\or 6\fi}}
\makeatother

\title{Non-Birkhoff periodic orbits in symmetric billiards}
\author{Casper Oelen\thanks{Maxwell Institute for Mathematical Sciences, Department of Mathematics, Heriot-Watt University, Edinburgh, EH14 4AS, UK,
{\tt c.oelen@hw.ac.uk}}\mbox{\ }$^{,}$\thanks{School of Mathematical Sciences, Loughborough University, Loughborough, LE11 3TU, UK}$^{\;\;, }$\thanks{Corresponding author, {\tt oelenc@gmail.com }}\;, Bob Rink\thanks{Department of Mathematics, Vrije Universiteit Amsterdam - Faculty of Science, De Boelelaan 1111, 1081 HV Amsterdam, The Netherlands, {\tt b.w.rink@vu.nl}}\;, Mattia Sensi\thanks{Dipartimento di Matematica, Università degli Studi di Trento, Via Sommarive 14, 38123 Povo (Trento), Italy, {\tt mattia.sensi@unitn.it}}\mbox{\ }$^{,}$\thanks{Department of Mathematical Sciences ``G. L. Lagrange'', Politecnico di Torino, Corso Duca degli Abruzzi 24, 10129 Torino, Italy}\mbox{\ }$^{,}$\thanks{MathNeuro team, Centre Inria d'Université Côte d'Azur, 2004 Rte des Lucioles, 06410 Biot, France} 
}
\date{\today}

\begin{document}

\maketitle

\begin{abstract}
We study non-Birkhoff periodic orbits in symmetric convex planar billiards. Our main result provides a quantitative criterion for the existence of such  orbits with prescribed minimal period, rotation number, and  spatiotemporal symmetry. We exploit this criterion to find sufficient conditions for a symmetric billiard to possess infinitely many non-Birkhoff periodic orbits. It follows that arbitrarily small analytical perturbations of the circular billiard have non-Birkhoff periodic orbits of any rational rotation number and with arbitrarily long periods. We also generalize a known result for elliptical billiards to other $\mathbb{D}_2$-symmetric billiards. Lastly, we  provide Matlab codes which can be used to numerically compute and visualize the non-Birkhoff periodic orbits whose existence we prove analytically. 
\end{abstract}
 
\section{Introduction}
\noindent 

The classical planar convex billiard problem \cite{tabachnikov2005geometry} considers the motion of a point through a strictly convex domain in the plane with smooth boundary. The point moves rectilinearly in the interior of the domain, changing its direction when it hits the boundary  according to the familiar billiard rule ``angle of incidence equals angle of reflection''. 
The study of this dynamical system is a well developed and active field of research \cite{DiasCarneiro_2007, rychlik1989periodic,tabachnikov2005geometry}. One of the first major results in this area is ``Poincar\'e's last geometric theorem'' about fixed points of symplectic twist maps, which was formulated by Poincar\'e \cite{poincare1912theoreme} and improved by G.D.  Birkhoff \cite{birkhoff1926extension, birkhoff1913proof}. Applied to smooth planar convex billiards, this theorem guarantees the existence of at least two distinct periodic billiard orbits of any rational rotation number $0<\frac{m}{n} < 1$. The periodic orbits obtained in this way, moreover, have the remarkable property that they preserve the orientation of the billiard boundary, that is, their impact points are \emph{well-ordered} in a sense that we make precise below. 
In honor of the discoverer of the theorem, orbits with this well-ordering property are nowadays called  \emph{Birkhoff} orbits. Billiards may have infinitely many Birkhoff orbits of any given rotation number. However, for each rational number $0<\frac{m}{n} <1$ there also exist billiards with exactly two distinct Birkhoff periodic orbits of that rotation number  \cite{pinto2012billiards}, which implies that Birkhoff's lower bound is sharp.

It can be shown that a periodic billiard trajectory of rotation number $\frac{m}{n} $ (with $m$ and $n$ co-prime) defines a Birkhoff orbit precisely when it is homeomorphic to a regular polygon of \emph{Schl\"afli symbol} $\{n/m\}$ (e.g., a triangle, square, pentagon, pentagram, hexagon, etc., see \cite{coxeter1973regular,schlafli1901theorie}.) In particular, $n$ is the minimal period of such an orbit, while $m$ is its winding number. 
{\it Non-Birkhoff} periodic orbits, that lack the aforementioned well-ordering property, do not admit such a simple geometric description. For instance, the minimal period of a non-Birkhoff periodic orbit of rotation number $\frac{m}{n}$ can be  (much) larger than $n$.  
A billiard also need not possess non-Birkhoff orbits. The circular billiard, for instance, only supports Birkhoff orbits. In  elliptical billiards, an orbit is non-Birkhoff precisely when it possesses a hyperbola as caustic, as can be inferred for example from \cite{casas2011frequency}. In particular, every non-Birkhoff periodic orbit in an elliptical billiard has rotation number $\frac{1}{2}$ automatically. 

Birkhoff orbits and non-Birkhoff orbits can be defined for general symplectic twist maps, see Definition  \ref{Birkhoff} in this paper.  For example, in the standard map and variants thereof, non-Birkhoff periodic orbits were found and investigated by Tanikawa and Yamaguchi, see for instance \cite{10.1143/PTP.108.987,10.1143/PTP.107.1117, 10.1143/PTP.117.601}.  
It was proved by Boyland and Hall \cite{boyland1987invariant}  that a symplectic twist map admits a non-Birkhoff periodic orbit of rotation number $\frac{m}{n}$, if and only if it does not admit an invariant circle of any irrational rotation number $\omega$ of which $\frac{m}{n}$ is a continued fraction convergent. For a class of twist maps of the infinite cylinder (not including billiards), this result was later improved by Cheng and Cheng \cite{chengcheng}. These authors prove that whenever such maps do not admit an invariant circle of irrational rotation number $\omega$, then they possess infinitely many non-Birkhoff periodic orbits of any rational rotation number $\frac{m}{n}$ sufficiently close to $\omega$.

As there is no simple  method to prove the non-existence of irrational invariant circles for billiards (or general twist maps), in this article we  aim instead to prove the existence of non-Birkhoff periodic orbits in billiards directly.
Specifically, we establish sufficient conditions for the existence of  these  orbits in \emph{symmetric} planar convex billiards. 
To describe our main results, we introduce some terminology. Assume that $\Gamma\subset \mathbb{R}^2$ is a $C^2$-smooth simple closed curve bounding a strictly convex domain (the billiard table). In particular, the \emph{curvature} $\kappa(z)$ of $\Gamma$  is strictly positive at a dense set of points $z\in \Gamma$. Recall that the curvature $\kappa(z)$ is defined as the reciprocal of the radius of the unique circle with a second order tangency to $\Gamma$ at $z$ (for example, if $\Gamma$ is a circle of radius $r$, then $\kappa \equiv r^{-1}$).

 A sequence $z= (\ldots, z_{-1}, z_0, z_1, \ldots)\in\Gamma^{\mathbb{Z}}$ of points in $\Gamma$ is called a {\it billiard orbit} if it is a sequence of consecutive impact points of a billiard trajectory. More precisely, $z$ is a billiard orbit when for all $i\in \mathbb{Z}$ we have that $z_{i+1}\neq z_i$  and that the vectors $z_{i-1}-z_i$ and $z_{i+1}-z_i$ make an equal angle with the tangent line to $\Gamma$ at $z_i$ (meaning that the billiard rule ``angle of incidence equals angle of reflection'' holds).

Let $G \subset O(2)$ be a finite group acting on $\mathbb{R}^2$. 
Then $\Gamma$  is called \emph{$G$-symmetric} if $g(\Gamma)=\Gamma$ for all $g\in G$. 
In this paper, we restrict $G$ to be a dihedral group containing at least one reflection and one rotation. For definiteness, we assume that $G$ is generated by the counterclockwise rotation $R$ of the plane over an angle $\frac{2\pi}{n}$, and the reflection $S$ in the horizontal axis. 
We denote  $\mathbb{D}_n:= \langle R, S\rangle \subset O(2)$.
We call a subgroup $H \subset \mathbb{D}_n$   \emph{dihedral} if it contains a reflection (we do not require that it contains a rotation).  Note that if $H\cong \mathbb{D}_N$ (for $1\leq N\leq n$), then $H$ is generated by a reflection and a rotation over $\frac{2\pi}{N}$. 

 If  $z \in \Gamma^{\mathbb{Z}}$ is a billiard orbit  for a $G$-symmetric billiard $\Gamma$, then so is 
$$g\cdot z := (\ldots, g(z_{-1}),\, g(z_0),\, g(z_1), \ldots),\ \mbox{ for all }\ g\in G .$$ 
An element $h\in G$ is called a \emph{time-preserving symmetry} of a billiard orbit $z\in \Gamma^{\mathbb{Z}}$ when there is a $k\in \mathbb{Z}$ such that $h(z_i)=z_{k+i}$ for all $i\in \mathbb{Z}$, that is, if $h\cdot z$ defines the same billiard orbit as $z$  up to a shift of time. The group of time-preserving symmetries of $z$ is denoted $H^+(z)$. An element $h\in G$ is called a \emph{time-reversing symmetry} of $z$ when there is a $k\in \mathbb{Z}$ such that $h(z_i)=z_{k-i}$ for all $i\in \mathbb{Z}$, that is, if $h\cdot z$ defines the same billiard orbit as $z$ up to a (shift and a) reversal of time. The set of time-reversing symmetries of $z$ is denoted $H^-(z)$. The union 
$$H(z):=H^+(z)\cup H^-(z)\subset G$$ 
is a group called the {\it spatiotemporal symmetry group} of $z$.  We say that $z$ is $H$-symmetric if $H\subset H(z)$.

Every $\mathbb{D}_n$-symmetric convex billiard admits precisely two (up to shifting or reversing time) $\mathbb{D}_n$-symmetric Birkhoff periodic orbits with rotation number $\frac{m}{n}$, see Lemma \ref{lem:existenceBirkhoffDn}. Let $Z\in \Gamma^{\mathbb{Z}}$ be one of these orbits.
 By symmetry, the segment length $L=\|Z_{i+1}-Z_i\|$ and the curvature $\kappa=\kappa(Z_i)$ of $\Gamma$ at $Z_i$ are constant (i.e., independent of $i\in \mathbb{Z}$). The product of these quantities is decisive for the existence of non-Birkhoff orbits, as  the main result of this paper demonstrates:

   \begin{theorem}\label{thm:sample}
    Let $m,n\in \mathbb{N}$ be co-prime, with $1\leq m \leq n-1$, and let $\Gamma$  be a $\mathbb{D}_n$-symmetric, $C^2$-smooth, strictly convex simple closed curve. 
     Let $Z\in \Gamma^{\mathbb{Z}}$ be one of its two 
      $\mathbb{D}_n$-symmetric Birkhoff periodic orbits with rotation number $\frac{m}{n}$. Denote by $L=\|Z_{i+1}-Z_i\|$ its constant  segment length and by $\kappa=\kappa(Z_i)$ the constant curvature of $\Gamma$ along $Z$. 
Let $H \subset \mathbb{D}_n$ be a dihedral subgroup of order $2N$ (for $N\geq 1$),  
let $s\in \mathbb{N}_{\geq 2}$ satisfy $\gcd(s, N)=1$, and define $p := s n$.    If      
\begin{equation}\label{eq:kappalambdaformulaintro}
\kappa L < 2 \sin\left(\frac{m \pi }{n} \right) \cos^2\left( \frac{N\pi}{p}\right) ,
\end{equation}
then $\Gamma$ admits a non-Birkhoff periodic orbit with minimal period $p$,  rotation number $\frac{m}{n}$, and spatiotemporal symmetry group $H$.   
\end{theorem}
\begin{figure}[h!]
\centering
\begin{subfigure}{.35\textwidth}
  \centering
\includegraphics[width=\linewidth]{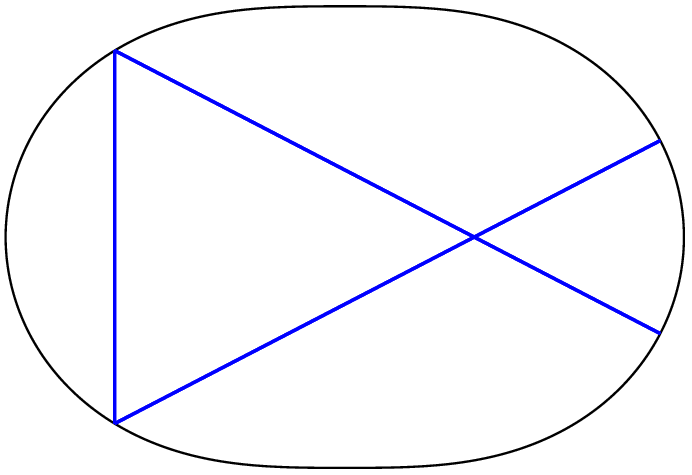}
  \caption{}
  \label{fig:intro1}
\end{subfigure}\hspace{2cm}
\begin{subfigure}{.35\textwidth}
  \centering  \includegraphics[width=\linewidth]{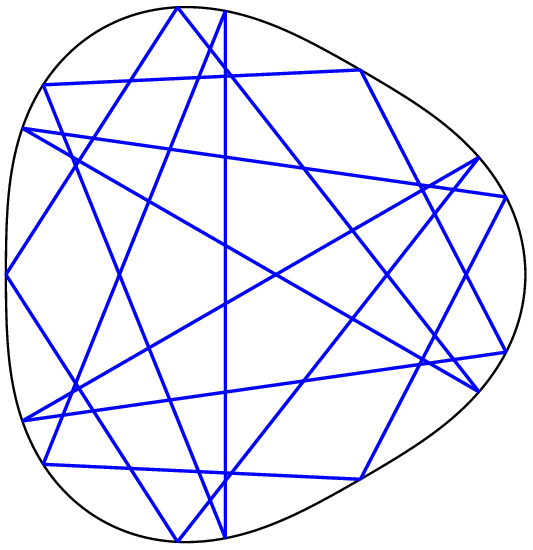}
  \caption{}
  \label{fig:intro2}
\end{subfigure}\\
\begin{subfigure}{.35\textwidth}
  \centering
\includegraphics[width=\linewidth]{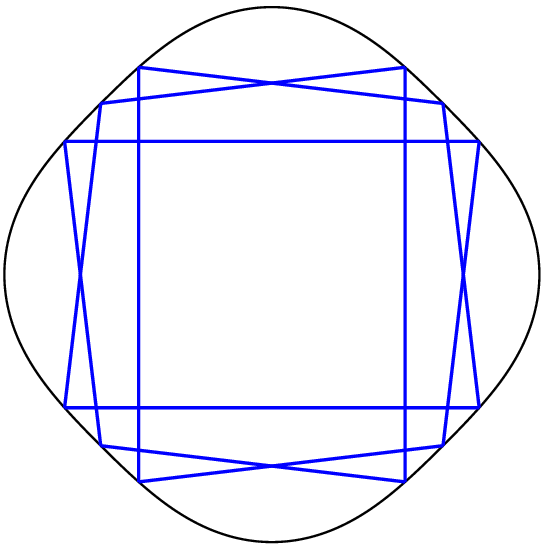}
  \caption{}
  \label{fig:intro3}
\end{subfigure}\hspace{2cm}
\begin{subfigure}{.35\textwidth}
  \centering
  \includegraphics[width=\linewidth]{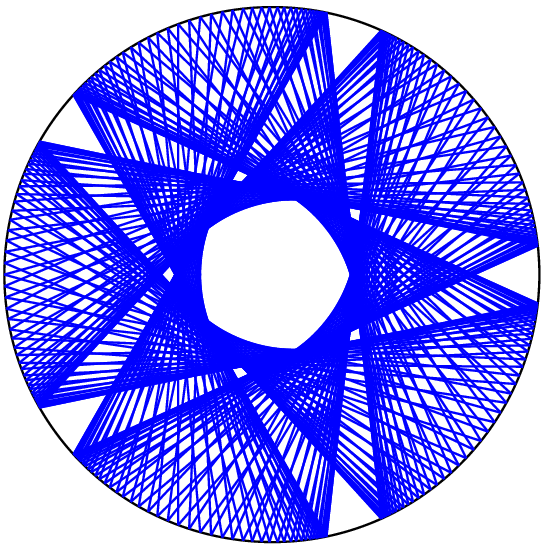}
  \caption{}
  \label{fig:intro4}
\end{subfigure}
\caption{Visualization of Theorem \ref{thm:sample} in four distinct symmetric convex billiards.  \textbf{(a)} $\mathbb{D}_1$-symmetric $(6,3)$-periodic  non-Birkhoff  orbit in a $\mathbb{D}_2$-symmetric billiard  -- note that the billiard trajectory is traversed in two directions throughout each period; \textbf{(b)} $\mathbb{D}_3$-symmetric $(15,5)$-periodic  non-Birkhoff  orbit in a $\mathbb{D}_3$-symmetric billiard; \textbf{(c)} $\mathbb{D}_4$-symmetric $(12,3)$-periodic  non-Birkhoff  orbit in a $\mathbb{D}_4$-symmetric billiard; \textbf{(d)} $\mathbb{D}_5$-symmetric $(245,98)$-periodic  non-Birkhoff  orbit in a close-to-circular $\mathbb{D}_5$-symmetric billiard. \label{fig:intro}}\end{figure}
\noindent  
Figure \ref{fig:intro} displays a small selection of symmetric non-Birkhoff periodic billiard orbits whose existence is guaranteed by this result. 
Theorem \ref{thm:sample} appears (with some more details) as Theorem \ref{thm:mainthrm} in this paper. The theorem  
gives an explicit quantitative criterion for the existence of a non-Birkhoff periodic orbit of prescribed minimal period, rotation number, and spatiotemporal symmetry.       
We prove several other theorems that follow from the verification of this criterion. We first remark that as soon as Theorem \ref{thm:sample} guarantees the existence of \emph{one} non-Birkhoff periodic orbit of minimal period $p$, then it guarantees the existence of infinitely many such orbits with periods larger than $p$ (simply because the right-hand side of \eqref{eq:kappalambdaformulaintro} increases as $p$ increases). This observation is formulated as Theorem \ref{thm:1}. Next, we show that inequality \eqref{eq:kappalambdaformulaintro} can be satisfied in arbitrarily small analytical perturbations of the circular billiard. This leads to Theorem \ref{thm:3}, which states that any open neighborhood (in the analytic topology) of the circle contains a billiard with non-Birkhoff periodic orbits of any prescribed rational rotation number.

For $\mathbb{D}_2$-symmetric billiards, we can prove more than Theorem \ref{thm:sample}. The non-Birkhoff periodic orbits guaranteed by Theorem \ref{thm:sample} have the property that any rotation in their spatiotemporal symmetry group is a time-preserving symmetry and not a time-reversing symmetry. However, $\mathbb{D}_2$-symmetric billiards may also possess periodic orbits on which a rotation acts by time-reversal. Theorems  \ref{thm:mainthrmII} and \ref{thm:mainthrmIII} are the counterparts of Theorem \ref{thm:sample} for such orbits, and give a quantitative criterion for their existence. 

A classical example of a $\mathbb{D}_2$-symmetric convex billiard is the elliptical billiard, which -- perhaps surprisingly -- still provides a source of new results \cite{akopyan2020billiards, reznik2020can, stachel2021isometric, stachel2022motion}. We mention in particular recent progress on the long-standing Birkhoff conjecture (which states that any integrable billiard must be an ellipse) in   \cite{kaloshin2024birkhoff,kaloshin2018local}.
The results in this paper imply that every (noncircular) elliptical billiard possesses infinitely many non-Birkhoff periodic orbits of rotation number $\frac{1}{2}$. Although this fact also follows from  results in \cite{casas2011frequency}, we show here that it is implied by  the $\mathbb{D}_2$-symmetry of the ellipse alone.
 
Our proof of Theorems \ref{thm:sample}, \ref{thm:mainthrm}, \ref{thm:mainthrmII} and \ref{thm:mainthrmIII} relies on various techniques from Aubry-Mather theory \cite{AubryDaeron, gole2001symplectic, MatherTopology} combined with ideas from equivariant dynamical systems \cite{field4, golschaef2}. In fact, most proofs in this paper exploit only the symmetry and the  monotone variational structure of the billiard problem. The majority of our results should therefore generalize to a large class of symmetric variational problems -- not only those that arise from mathematical billiards. Concretely, we find our non-Birkhoff periodic orbits as stationary points of the gradient flow of the length functional, restricted to appropriate spaces of symmetric periodic sequences. 
The estimate \eqref{eq:kappalambdaformulaintro} is derived in a local analysis near a symmetric Birkhoff periodic orbit. However, we stress that the non-Birkhoff periodic orbits that we find, need not lie close to this Birkhoff orbit. This distinguishes this paper from recent work on symmetric billiards, such as \cite{ferreira2023symmetric}, which studies, among other things, the effect of symmetry on the stability of periodic billiard orbits.

The remainder of this paper is organized as follows. In Section \ref{sec:basicsofbilliardssection}, we introduce some basic properties of convex billiards and rephrase the billiard problem as a monotone recurrence relation. In Section \ref{sec:Birkhoffsection}, we define Birkhoff orbits and characterize Birkhoff periodic billiard  orbits as regular polygons. In Sections \ref{sec:periodicorbitssection} and \ref{sec:symmetrygradientsection}, we investigate the spatiotemporal symmetries of billiard sequences, and we classify periodic billiard sequences with dihedral spatiotemporal symmetry groups. 
Section \ref{sec:gradientflowsection} is concerned with fundamental properties of the gradient flow of the length functional, which is a crucial technical tool used in the paper. Section  \ref{sec:symBirkhoffsection} focuses on $\mathbb{D}_n$-symmetric periodic Birkhoff orbits. Sections \ref{sec:Hessiansection} and \ref{sec:technicalsection} contain  technical results needed for the proof of our main result -- Theorem \ref{thm:mainthrm} -- in Section \ref{sec:mainthmsection}. We investigate $\mathbb{D}_2$-symmetric billiards separately in Section \ref{sec:D2section}, and in Section \ref{sec:infinitesection} we prove various corollaries of our theorems regarding the existence of infinitely many non-Birkhoff orbits. We draw some conclusions in Section \ref{sec:concl}. Appendix \ref{app:convexity} contains a proof of the convexity of a class of Lima\c{c}on-type billiards which we use to illustrate and visualize our results. Appendix \ref{app:code} briefly explains the Matlab code used to numerically produce our visualizations. This code is available in the dedicated GitHub repository \href{https://github.com/MattiaSensi/BilliardOrbitFinder}{BilliardOrbitFinder}.

 \section{A monotone recurrence relation}\label{sec:basicsofbilliardssection}
\noindent In this section and the next, we introduce some basic notions from the theory of mathematical billiards that we need in the sequel. The goal of the current section is to rephrase the billiard problem as a monotone recurrence relation on a space of real-valued sequences. The material presented here is standard. Throughout this paper, a billiard will be a strictly convex domain in $\mathbb{R}^2$ bounded by a $C^2$-smooth simple closed curve $\Gamma$. By \emph{strictly convex} we mean that any line segment connecting two points on $\Gamma$ is  contained within the interior of the domain bounded by $\Gamma$ (apart from the endpoints of the segment). Specifically, we  assume that $\Gamma$ is parameterized by a $1$-periodic and $C^2$-smooth map 
$$\gamma: \mathbb{R} \to \mathbb{R}^2\,. $$
We require that $\gamma$ descends to an embedding on $\mathbb{R}/\mathbb{Z}$. 
In particular, we assume that $\gamma(x+1)=\gamma(x)$ for all $x\in \mathbb{R}$, that $\gamma(x)=\gamma(y)$ implies that $x-y\in \mathbb{Z}$, and that $\gamma'(x) \neq (0,0)$ for all $x\in \mathbb{R}$. For definiteness, we assume that $\Gamma$ is oriented counterclockwise. The strict convexity of $\Gamma$ then implies that $\det \left( \gamma'(x), \gamma''(x) \right) \geq 0$, and that this quantity is strictly positive for a dense set of $x\in \mathbb{R}$.

We study the billiard problem by solving the recurrence relation between $z_{i-1}, z_{i}$ and $z_{i+1}$ given by the billiard rule ``angle of incidence equals angle of reflection''. 
To  write down this recurrence relation explicitly, we make some definitions. First of all, we will call a bi-infinite sequence 
$$z=(\ldots, z_{-1}, z_0, z_{1}, \ldots )\in \Gamma^{\mathbb{Z}}$$ 
of points on $\Gamma$ a {\it billiard sequence} if $z_{i}\neq z_{i+1}$ for all $i\in \mathbb{Z}$. We do not  require a billiard sequence to be a billiard orbit, i.e., a billiard sequence may or may not satisfy the billiard rule. When $z$ is  a billiard sequence, then for each $i\in \mathbb{Z}$, we may choose an $x_i\in \mathbb{R}$ with $\gamma(x_i)=z_i$. This $x_i$ is unique up to addition of an integer. The real-valued sequence 
$$x = (\ldots, x_{-1}, x_0, x_1, \ldots)\in \mathbb{R}^{\mathbb{Z}}$$ 
is then called a \emph{lift} of $z$. 
  This lift is not unique, but we shall assume without loss of generality that
$x_i < x_{i+1} < x_i + 1$ for all $i \in \mathbb{Z}$. The collection of such real-valued sequences shall be denoted by
\begin{equation}\label{eq:billiardliftsequence}
\Sigma := \{ x\in \mathbb{R}^{\mathbb{Z}} \ |\ 0 < x_{i+1} - x_i < 1 \ \mbox{for all}\ i \in \mathbb{Z}\, \}\, .
\end{equation}
When $x\in \Sigma$ is a lift of $z\in \Gamma^{\mathbb{Z}}$, then the increment $x_{i+1}-x_i\in (0,1)$ can be interpreted as    the  ``distance'' from $z_{i}$ to $z_{i+1}$ measured counterclockwise along $\Gamma$.

A billiard sequence $z$ is said to be {\it periodic} with period $p\in \mathbb{N}$ if $z_i=z_{p+i}$ for all $i\in\mathbb{Z}$. When $x\in \Sigma$ is a lift of a $p$-periodic billiard sequence $z$, then the integer $$q:= x_{p+i}-x_i  = (x_{p+i}-x_{p+i-1})+\ldots + (x_{i+1}-x_i) = (x_{p}-x_{p-1})+\ldots + (x_{1}-x_0) \in \{1,\ldots, p-1\} $$
(which is independent of $i\in \mathbb{Z}$)  is called the {\it winding number} of $z$. In other words, for every $p$-periodic billiard sequence $z$ there is a unique integer $0<q<p$ such that any lift $x \in \Sigma$ of $z$ lies in  
$$\mathbb{X}_{p,q}:= \{x\in \mathbb{R}^{\mathbb{Z}} \, |\, x_{p+i}=x_i + q \ \mbox{for all}\ i\in \mathbb{Z}\, \} \subset \mathbb{R}^{\mathbb{Z}}\, .$$
Elements of $\mathbb{X}_{p,q}$ are called {\it $(p,q)$-periodic}. We note that every $x\in\mathbb{X}_{p,q}$ satisfies $\lim_{i\to\pm \infty} \frac{x_i}{i}=\frac{q}{p}$, that is, every element of $\mathbb{X}_{p,q}$ has rational rotation number $\frac{q}{p}$. When $x\in \Sigma$ is the lift of a billiard orbit $z\in \Gamma^{\mathbb{Z}}$, this rotation number has the interpretation of the average rotation (per reflection) of the billiard ball around the billiard table.

Now we turn to describing the recurrence relation that the lift of a billiard sequence must satisfy to represent an actual billiard trajectory. For $x, X\in \mathbb{R}$, we shall denote by 
$L(x,X)$ the Euclidean distance between the two points $\gamma(x), \gamma(X)\in\Gamma \subset \mathbb{R}^2$, that is, we define 
$$L(x,X) := \|\gamma(x)-\gamma(X)\|\, .$$
This function is continuous on $\mathbb{R}^2$ and smooth at points $(x,X)\in \mathbb{R}^2$ where $x - X \notin \mathbb{Z}$. 
The following well-known lemma describes the relation between the partial derivatives of $L(x,X)$ and the angles between the billiard boundary and the vector $\gamma(X)-\gamma(x)$. In Lemma \ref{PartialsLalternativesecondorder} we will provide a similar result for the second order partial derivatives of $L$. In the following, $\partial_1 L(x,X)$ and $\partial_2 L(x,X)$ respectively indicate the derivatives of $L(x,X)$ with respect to its first and second argument.
\begin{lemma}\label{PartialsLalternativefirstorder}
Let $x,X\in \mathbb{R}$ with $x-X\notin \mathbb{Z}$, and define $$\Theta(x,X):=\angle \left(\gamma'(x), \gamma(X)-\gamma(x)\right) \in (0,\pi)$$ 
to be the angle between $\gamma'(x)$ and $\gamma(X)-\gamma(x)$, and 
$$\Phi(x,X) := \angle \left(\gamma(X)-\gamma(x),  \gamma'(X) \right)\in (0,\pi)$$ 
the angle between $\gamma(X)-\gamma(x)$ and $\gamma'(X)$.  Then
\begin{align}
& \partial_1 L(x,X) = -\|\gamma'(x)\| \cos \Theta(x,X)    \quad \mbox{and}\quad  \partial_2 L(x,X) = \|\gamma'(X)\|\cos \Phi(x,X)\, .
\end{align}
\end{lemma}
\begin{proof}
    Differentiation of $L(x,X)=\sqrt{ \langle \gamma(x)-\gamma(X),\gamma(x)-\gamma(X)\rangle }$ yields 
$$\partial_1L(x,X) =  \left\langle \frac{\gamma(x)-\gamma(X)}{L(x,X)}, \gamma'(x) \right\rangle = -\|\gamma'(x)\|\cos \Theta(x,X)\, ,$$
by definition of $\Theta(x,X)$.  Similarly, by definition of $\Phi(x,X)$,  
$$\partial_2L(x,X) =    \left\langle \frac{\gamma(X)-\gamma(x)}{L(x,X)},  \gamma'(X) \right\rangle =  \|\gamma'(X)\|\cos \Phi(x,X)\, .$$ 
\end{proof}
\noindent Note that the strict convexity of $\Gamma$ implies that $\Theta(x,X), \Phi(x,X) \neq 0, \pi$. 
When a real-valued sequence $x = (\ldots, x_{-1}, x_0, x_1, \ldots) \in \Sigma$ is the lift of a billiard sequence $z=(\ldots, z_{-1}, z_0, z_1, \ldots)$,  we will write 
$$\phi_i:=\Phi(x_{i-1}, x_i) \ \mbox{ and }\ \theta_i:=\Theta(x_{i}, x_{i+1})$$
for the \emph{angle of incidence} of the segment $z_{i}-z_{i-1}$ at $z_i$, respectively the \emph{angle of reflection} of the segment $z_{i+1}-z_i$ from $z_i$, see Figure \ref{billiard_law}.
 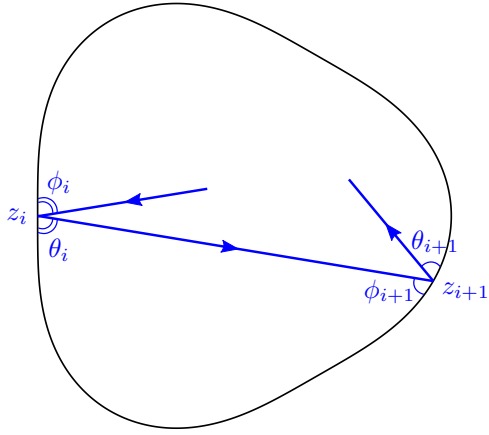
\begin{figure}
\centering
  \begin{tikzpicture}
  \node at (0,0) {\includegraphics[width=0.35\linewidth]{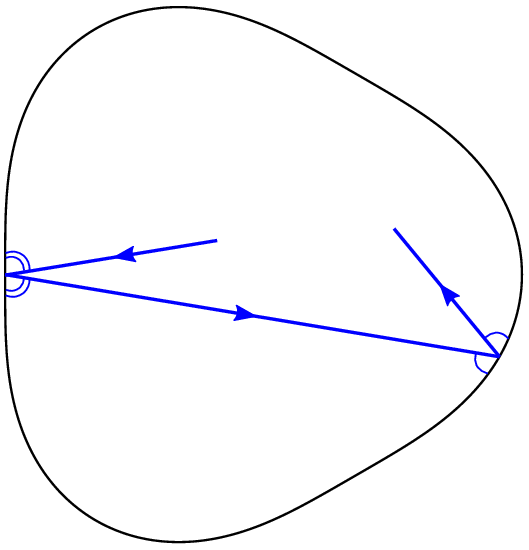}};
\node at (2.9,-1) {$\textcolor{blue}{z_{i+1}}$};
\node at (1.9,-1) {$\textcolor{blue}{\phi_{i+1}}$};
\node at (-2.45,0.45) {$\textcolor{blue}{\phi_i}$};
\node at (-2.45,-0.45) {$\textcolor{blue}{\theta_i}$};
\node at (-3,0) {$\textcolor{blue}{z_i}$};
\node at (2.5,-0.35) {$\textcolor{blue}{\theta_{i+1}}$};
  \end{tikzpicture}  
  \caption{$\mathbb{D}_3$-symmetric billiard $\Gamma$ and part of a billiard sequence $z=(\ldots, z_i, z_{i+1}, \ldots)\in \Gamma^{\mathbb{Z}}$. The angle of incidence at $z_i$ is denoted $\phi_i$, and the angle of reflection at $z_i$ is denoted $\theta_i$. We have that $z$ is a billiard orbit precisely when $\phi_i=\theta_i$ for all $i\in \mathbb{Z}$.
 \label{billiard_law}}
\end{figure}
It follows from Lemma \ref{PartialsLalternativefirstorder} that  
$$\partial_2L(x_{i-1},x_i)  = \|{\changes{\gamma'}} (x_i)\| \cos \phi_i \quad \mbox{and} \quad  \partial_1L(x_i, x_{i+1}) =  -\|\changes{\gamma'} (x_i)\|\cos \theta_i\, .$$
This shows that $z\in \Gamma^{\mathbb{Z}}$ satisfies the billiard rule if and only if its lift $x\in \Sigma$ satisfies  
\begin{align}\label{eq:recrelgeneral}
 \partial_2 L(x_{i-1}, x_i) + \partial_1 L(x_i,x_{i+1}) = 0\ \mbox{ for all }\ i\in \mathbb{Z}\, .
 \end{align}
 We will search for periodic billiard orbits by solving this  second order recurrence relation. 
 
\begin{remark}
Equation \eqref{eq:recrelgeneral} is the Euler-Lagrange equation for the discrete Lagrangian $L$. Indeed, the left-hand side of  \eqref{eq:recrelgeneral} is equal  to the  derivative
$$\left. \frac{d}{dx}\right|_{x=x_i} L(x_{i-1}, x) + L(x, x_{i+1}) \, .$$
Thus, we observe the well-known variational principle behind the billiard problem: the billiard rule is satisfied at $z_i$, precisely when $z_i$ is a stationary point (among points $z\in \Gamma$)  of the sum   $\|z_{i-1}-z\| + \|z-z_{i+1}\|$ of the lengths of the orbit segments of which it is an endpoint. 
\end{remark}

The billiard problem is autonomous and time-reversible: if $i\mapsto x_i$ is a solution to \eqref{eq:recrelgeneral} and  $k\in \mathbb{Z}$, then also $i\mapsto x_{k+i}$ and $i \mapsto x_{k-i}$ are solutions to \eqref{eq:recrelgeneral}, see Remark \ref{remk:autonomousreversible}. In other words, traversing a given solution of \eqref{eq:recrelgeneral} earlier, later or in reverse yields another solution of \eqref{eq:recrelgeneral}. 
Of course, these time-translated and time-reversed orbits define exactly the same billiard trajectory in the plane. To distinguish billiard sequences that define truly distinct planar trajectories, we make the following definition.

\begin{definition}\label{rem:geom-dist}
Let $z, Z\in\Gamma^{\mathbb{Z}}$ be two billiard sequences   with lifts $x,X\in \Sigma$. Then $z$ and $Z$ (and $x$ and $X$) are called \emph{geometrically equal} if there exists an integer $k \in \mathbb{Z}$, such that either $z_i=Z_{k+i}$ for all $i\in \mathbb{Z}$, or $z_i=Z_{k-i}$ for all $i\in \mathbb{Z}$. Otherwise, $z$ and $Z$ (and $x$ and $X$) are called \emph{geometrically distinct}. \end{definition}
   
We conclude this section with a few more remarks.
 
\begin{remark}\label{remk:autonomousreversible}
Let $x\in \Sigma$ be a solution to \eqref{eq:recrelgeneral}, let $k\in \mathbb{Z}$, and define $y_i:=x_{k+i}$. Then, clearly, 
\begin{align}\nonumber 
 \partial_2 L(y_{i-1}, y_i) + \partial_1 L(y_{i}, y_{i+1}) =    \partial_2 L(x_{(k+i)-1}, x_{k+i}) + \partial_1 L(x_{k+i}, x_{(k+i)+1}) = 0\, .
\end{align}
This shows that $y$ solves \eqref{eq:recrelgeneral} as well, and  proves our claim that the billiard problem is autonomous. Time-reversibility is a consequence of the symmetry $L(x,X)=\|\gamma(x)-\gamma(X)\| = L(X,x)$ of the discrete Lagrangian. Indeed, it follows from this symmetry that $\partial_1L(x,X)=\partial_2L(X,x)$. Therefore, if $x$ satisfies \eqref{eq:recrelgeneral}, $k\in \mathbb{Z}$, and $y_i:=x_{k-i}$, then  
    \begin{align} \nonumber
    &  \partial_2L(y_{i-1}, y_i)+ \partial_1L(y_{i}, y_{i+1}) =       \partial_1L(y_{i}, y_{i-1}) + \partial_2L(y_{i+1}, y_i) = \\ \nonumber &   \partial_1L(x_{k-i}, x_{k-(i-1)}) + \partial_2L(x_{k-(i+1)}, x_{k-i}) = 
       \partial_2L(x_{(k-i)-1}, x_{(k-i)})  + \partial_1L(x_{(k-i)}, x_{(k-i)+1}) =0 \, . 
    \end{align}
\end{remark}

\begin{remark}\label{rem:firstmonotoneremark}
We will show in Remark \ref{rem:secondmonotoneremark} and Lemma \ref{lem:Hessianlemma} below that $\partial_{1,2}L(x,X)>0$. This means that \eqref{eq:recrelgeneral} is a so-called \emph{monotone} recurrence relation. Monotonicity is crucial for proving the {\rm comparison principle} (Lemma \ref{lem:comparisonprinciple}) and the {\rm Sturmian lemma} (Lemma \ref{prop:decreaseintersection}) for the  gradient flow, which will be introduced in Section \ref{sec:gradientflowsection}.
\end{remark}

\begin{remark}
Whenever two consecutive impact points $z_{i-1}, z_i \in \Gamma$ of a billiard trajectory are known, then the next impact point $z_{i+1}\in \Gamma$ is uniquely determined by the billiard rule. Thus, any smooth and strictly convex billiard curve $\Gamma$ determines a well-defined discrete-time dynamical system on $\{(z,Z)\in \Gamma \times \Gamma\, |\, z\neq Z\}$, which assigns to each consecutive pair $(z_{i-1}, z_i)$ of impact points the next pair $(z_i, z_{i+1})$. Equivalently, one can describe the billiard dynamics in terms of impact points and angles of reflection. Indeed, given an impact point $z_i\in \Gamma$ and angle of reflection $\theta_i \in (0,\pi)$, the next impact point and angle of reflection are uniquely determined. The exact symplectic twist map $T:\Gamma\times (0,\pi) \to \Gamma \times (0,\pi)$ assigning to the  pair $(z_i,\theta_i)$  the next pair $(z_{i+1}, \theta_{i+1})$ is what is commonly known as \emph{the billiard map}. Although one may study periodic billiard trajectories by finding periodic orbits of the billiard map, in this paper we instead solve the recurrence relation \eqref{eq:recrelgeneral} directly.
\end{remark}

\section{The Birkhoff property}\label{sec:Birkhoffsection}
\noindent A Birkhoff sequence is a sequence of points on $\Gamma$ that respects the ``cyclic order'' of $\Gamma$. In this section, we will
make this statement precise. In the literature, the Birkhoff property is usually only defined for real-valued sequences. However, here we will separately define what it means for a sequence $z\in \Gamma^{\mathbb{Z}}$ to possess the Birkhoff property. In Lemma \ref{Orderintext} we then show that a billiard sequence $z\in \Gamma^{\mathbb{Z}}$ satisfies our definition of Birkhoff precisely when its lift $x\in \Sigma$ is Birkhoff in the usual sense. To the best of our knowledge Lemma \ref{Orderintext} is new (albeit rather elementary). We start by introducing some notation.
\begin{definition}
Let $z_i, z_j, z_k$ be three points on $\Gamma$. We write $$z_i\prec z_j \prec z_k$$ 
if there exist $x_i, x_j, x_k\in \mathbb{R}$ such that $z_i=\gamma(x_i), z_j=\gamma(x_j), z_k=\gamma(x_k)$, and $$x_i \leq x_j \leq x_k < x_i+1\, .$$
\end{definition}
\noindent In other words, $z_i\prec z_j \prec z_k$  if $z_j$ lies in the positively oriented interval between $z_i$ and $z_k$,  obtained by traversing $\Gamma$ counterclockwise from $z_i$ to $z_k$. We can now define when a  sequence of points on $\Gamma$ is Birkhoff, and when a real-valued sequence is Birkhoff.  

\begin{definition}\label{def:BirkhoffIntro}
 
A  sequence $z=(\ldots, z_{-1}, z_0, z_1, \ldots)\in \Gamma^{\mathbb{Z}}$ 
 is    \emph{well-ordered} or \emph{Birkhoff} if  
 $$z_i \prec z_j \prec z_k \implies z_{i+m} \prec z_{j+m} \prec z_{k+m}  \  \mbox{for all}\ i,j,k,m \in
\mathbb{Z}\, . $$ Otherwise, $z$ is called \emph{non-Birkhoff}.
\end{definition}
\begin{definition}  \label{Birkhoff}
A real-valued sequence $x=(\ldots, x_{-1}, x_0, x_1, \ldots) \in \mathbb{R}^{\mathbb{Z}}$ is   \emph{well-ordered} or \emph{Birkhoff} if $$x_i \leq x_j + l  \implies x_{i+m}\leq x_{j+m}+l\  \mbox{for all}\ i,j,l,m \in
\mathbb{Z}\, .$$ 
Otherwise, $x$ is called \emph{non-Birkhoff}.
\end{definition}
\noindent Figure \ref{fig:first_example} depicts Birkhoff and non-Birkhoff periodic orbits in two symmetric billiards.
The following result confirms that, for billiard sequences and their lifts,  Definitions \ref{def:BirkhoffIntro} and \ref{Birkhoff} coincide. 
\begin{lemma}\label{Orderintext}
Let $z\in \Gamma^{\mathbb{Z}}$ be a billiard sequence with lift $x\in \Sigma$. Then $z$ is Birkhoff according to Definition \ref{def:BirkhoffIntro} if and only if $x$ is Birkhoff according to Definition \ref{Birkhoff}.
\end{lemma} 
\begin{proof}
Let $x\in \mathbb{R}^{\mathbb{Z}}$. For all $i,j\in \mathbb{Z}$ there is a unique $l(i,j)\in \mathbb{Z}$ such that $$x_i\leq x_j +l(i,j) < x_i+ 1\, .$$
For example, for $x\in \Sigma$ we have $x_i < x_{i+1} < x_i+1 < x_{i+1}+1$ so $l(i,i+1)=0$ and $l(i+1, i) = 1$.   
The proof of the \changes{lemma} hinges on two facts concerning these $l(i,j)$. We prove these facts first.

{\it Fact 1:} $x\in\mathbb{R}^{\mathbb{Z}}$ is Birkhoff (by Definition \ref{Birkhoff}) if and only if $l(i,j)=l(i+m, j+m)$ for all $i,j,m\in \mathbb{Z}$. 

{\it Proof:} Start by assuming that 
   $x$ is Birkhoff. Then the inequality  $x_{i}\leq x_j+l(i,j)$ implies that $x_{i+m} \leq x_{j+m} +l(i,j) < x_{i+m}+1 -l(i+m, j+m)+l(i,j)$, so $l(i,j) \geq  l(i+m, j+m)$. Similarly, $x_{i+m} \leq x_{j+m} + l(i+m, j+m) $ implies $x_{i} \leq x_{j} + l(i+m, j+m) < x_i+1-l(i,j)+l(i+m,j+m)$ so $l(i,j) \leq l(i+m,j+m)$. This proves that if $x$ is Birkhoff, then $l(i,j)=l(i+m,j+m)$ for all $i,j,m\in \mathbb{Z}$. Conversely, if $l(i+m,j+m)=l(i,j)$ and $x_{i}\leq x_j + l$, then $l(i,j) \leq l$ by definition of $l(i,j)$. Therefore, $x_{i+m} \leq x_{j+m} + l(i+m, j+m) = x_{j+m} + l(i, j) \leq  x_{j+m} + l$. Thus, $x$ is Birkhoff. \hfill $\square$

{\it Fact 2:} It holds that $z_i\prec z_j \prec z_k$ if and only if $l(i,k)=l(i,j)+l(j,k)$. 

{\it Proof:} Note that $z_i\prec z_j \prec z_k$ if and only if \begin{equation}
\label{eq:orderliftequation}
x_{i} \leq x_{j}+l(i,j) \leq x_{k}+l(i,k) < x_{i}+1\, .
\end{equation}
From this we can infer that
$x_{j} \leq x_{k}+l(i,k) - l(i,j) < x_{j} + 1$. In other words, $l(j,k)=l(i,k)-l(i,j)$. Conversely, if $l(j,k)=l(i,k)-l(i,j)$, then $x_{j} \leq x_{k}+l(i,k) - l(i,j) < x_{j} + 1$, from which it follows that $x_j+l(i,j)\leq x_k+l(i,k)$, i.e., \eqref{eq:orderliftequation} holds and therefore $z_i\prec z_j\prec z_k$. \hfill  $\square$

Now we return to the proof of the lemma. Assume that $x\in \mathbb{R}^{\mathbb{Z}}$ is a lift of $z\in\Gamma^{\mathbb{Z}}$ and that $x$ is Birkhoff (according to Definition \ref{Birkhoff}). Suppose that $z_i\prec z_j \prec z_k$. This means that 
  \eqref{eq:orderliftequation}
  holds. 
  The Birkhoff property of $x$ then firstly implies that $x_{i+m} \leq x_{j+m} +l(i,j) \leq x_{k+m} + l(i,k) \leq x_{i+m}+1$, and secondly (due to Fact 1) that $x_{i+m} \leq x_{j+m} +l(i+m,j+m) \leq x_{k+m} + l(i+m,k+m) \leq x_{i+m}+1$. By definition of $l(i+m, k+m)$, the last inequality actually is strict. This proves that $z_{i+m}\prec z_{j+m} \prec z_{k+m}$. In other words, $z$ is Birkhoff (according to Definition \ref{def:BirkhoffIntro}).
    
Conversely, let $z \in \Gamma^{\mathbb{Z}}$ be Birkhoff (according to Definition \ref{def:BirkhoffIntro}) and assume that its lift $x$ lies in $\Sigma$ (i.e., $z_i\neq z_{i+1}$ for all $i\in \mathbb{Z}$). Choose $z_i, z_j, z_k$ so that $z_i\prec z_j\prec z_k$, and thus also $z_{i+m}\prec z_{j+m}\prec z_{k+m}$ for all $m\in \mathbb{Z}$. Fact 2 then implies that $l(i+m,k+m)=l(i+m,j+m)+l(j+m,k+m)$  for all $m\in \mathbb{Z}$. 
  
   Consider the points $z_0, z_1, z_2 \in \Gamma$. Either it holds that $z_0\prec z_1 \prec z_2$ or $z_0\prec z_2 \prec z_1$. In the first case, we find $l(0+m,2+m)=l(0+m,1+m)+l(1+m,2+m)=0+0=0$, which is independent of $m$. In the second case,  $0=l(0+m,1+m)=l(0+m,2+m)+l(2+m,1+m)= l(0+m,2+m)+1$, and again $l(0+m,2+m)$ is independent of $m$. Proceeding inductively (distinguishing the cases $z_0 \prec z_{j} \prec z_{j+1}$ and $z_0 \prec z_{j+1} \prec z_{j}$, we find that $l(0+m,j+m)$ is always independent of $m$. In other words $l(i+m,j+m)=l(i,j)$ for all $i,j,m\in \mathbb{Z}$. This proves that $x$ is Birkhoff (according to  Definition \ref{Birkhoff}).
\end{proof}

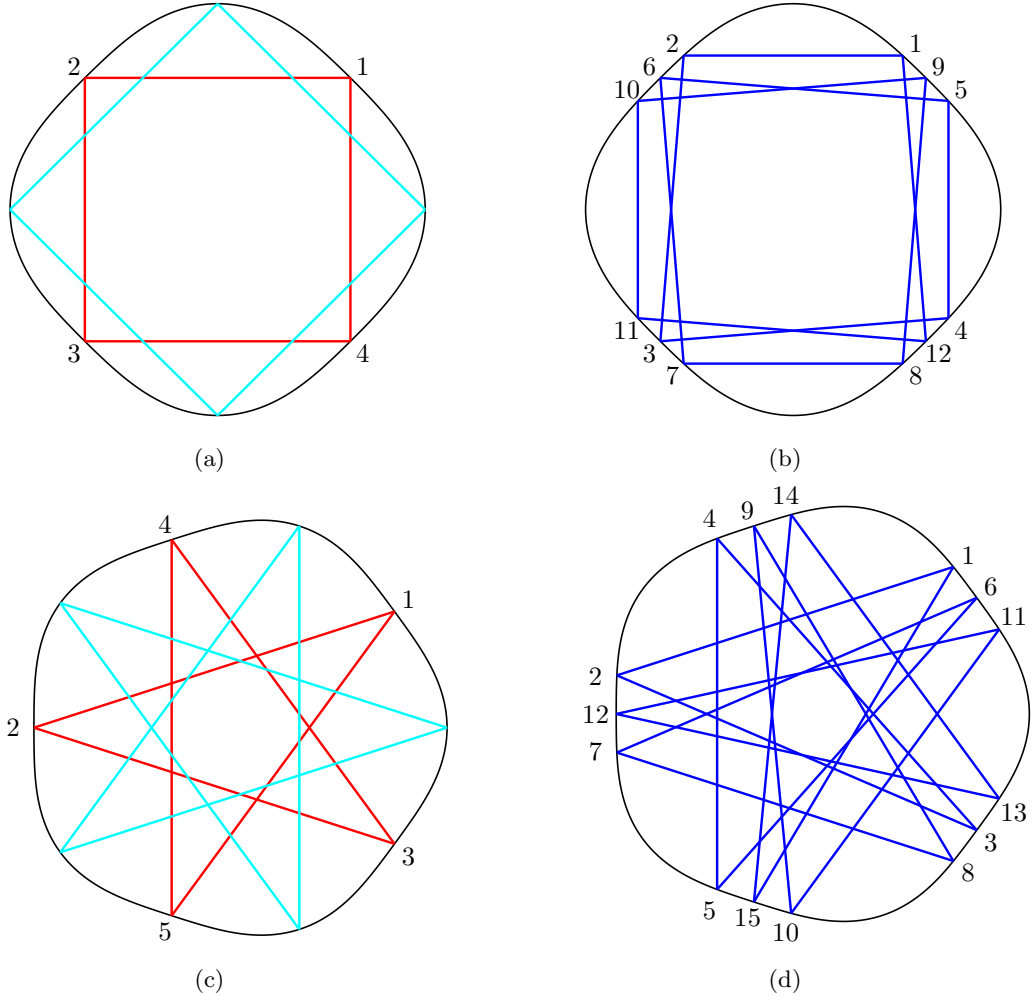
\begin{figure}[h!]
\centering
\begin{subfigure}{.35\textwidth}
  \centering
\begin{tikzpicture}
  \node at (0,0) {\includegraphics[width=\linewidth]{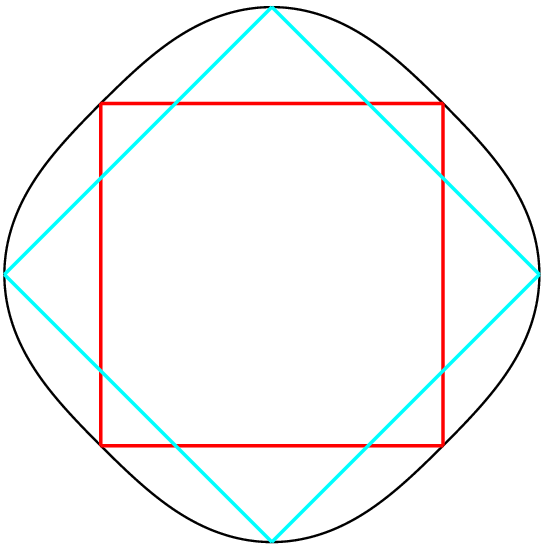}};
\node at (1.9,1.9) {$1$};
\node at (-1.9,1.9) {$2$};
\node at (-1.9,-1.9) {$3$};
\node at (1.9,-1.9) {$4$};
  \end{tikzpicture}
  \caption{}
  \label{fig:sub1}
\end{subfigure}\hspace{2cm}
\begin{subfigure}{.35\textwidth}
  \centering
  \begin{tikzpicture}
  \node at (0,0) {\includegraphics[width=\linewidth]{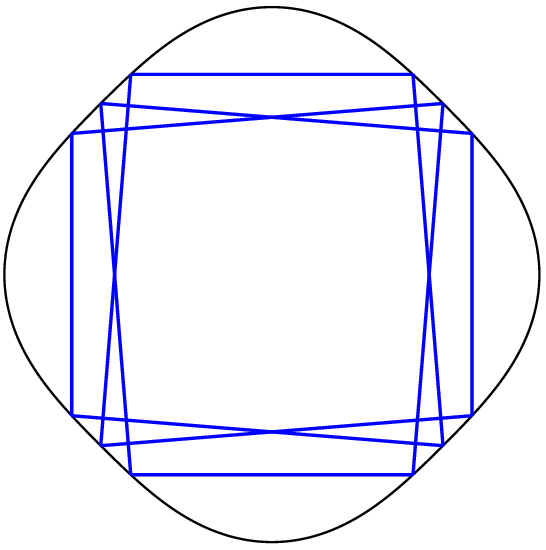}};
\node at (1.6,2.2) {$1$};
\node at (-1.6,2.2) {$2$};
\node at (-1.9,-1.9) {$3$};
\node at (2.2,-1.6) {$4$};
\node at (2.2,1.6) {$5$};
\node at (-1.9,1.9) {$6$};
\node at (-1.6,-2.2) {$7$};
\node at (1.6,-2.2) {$8$};
\node at (1.9,1.9) {$9$};
\node at (-2.2,1.6) {$10$};
\node at (-2.2,-1.6) {$11$};
\node at (1.9,-1.9) {$12$};
  \end{tikzpicture}
  \caption{}
  \label{fig:sub2}
\end{subfigure}\\
\centering \begin{subfigure}{.35\textwidth}
  \centering
  \begin{tikzpicture}
  \node at (0,0) {\includegraphics[width=\linewidth]{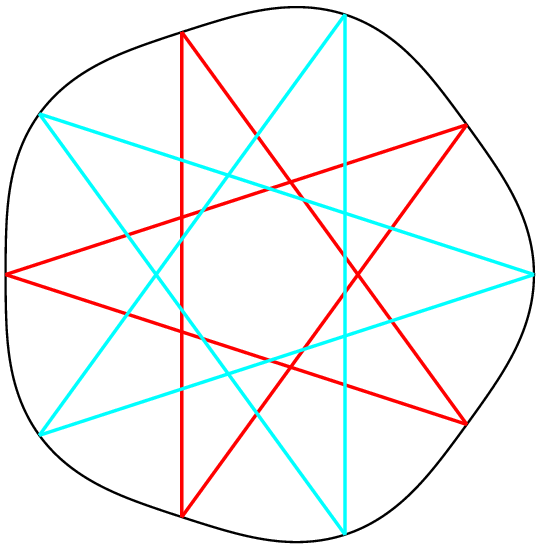}};
\node at (2.2,1.7) {$1$};
\node at (-1,2.7) {$4$};
\node at (-3,0) {$2$};
\node at (-1,-2.7) {$5$};
\node at (2.2,-1.7) {$3$};
  \end{tikzpicture}
  \caption{}
  \label{fig:sub3}
\end{subfigure}\hspace{2cm}
\begin{subfigure}{.35\textwidth}
  \centering
  \begin{tikzpicture}
  \node at (0,0) {\includegraphics[width=\linewidth]{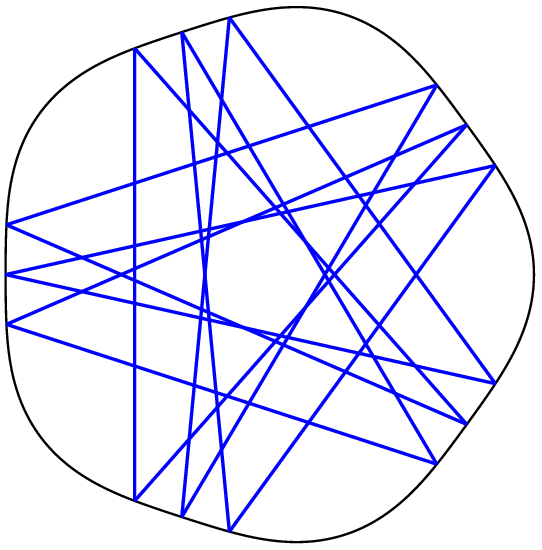}};
\node at (1.9,2.1) {$1$};
\node at (-1.5,2.6) {$4$};
\node at (-3,0.5) {$2$};
\node at (-1.5,-2.6) {$5$};
\node at (2.2,-1.7) {$3$};
\node at (2.2,1.7) {$6$};
\node at (-3,-0.5) {$7$};\node at (1.9,-2.1) {$8$};
\node at (-1,2.7) {$9$};
\node at (-0.5,-2.9) {$10$};
\node at (2.5,1.3) {$11$};
\node at (-3,0) {$12$};
\node at (2.5,-1.3) {$13$};
\node at (-0.5,2.9) {$14$};
\node at (-1,-2.7) {$15$};
  \end{tikzpicture}
  \caption{}
  \label{fig:sub4}
\end{subfigure}
\caption{Symmetric billiards with Birkhoff periodic orbits and non-Birkhoff periodic orbits: {\bf (a)} $\mathbb{D}_4$-symmetric billiard with two $(4,1)$-periodic Birkhoff orbits, the ``short'' one in red and the ``long'' one in cyan; 
{\bf (b)} $\mathbb{D}_4$-symmetric billiard with $(12,3)$-periodic non-Birkhoff orbit: note that $z_9\prec z_1 \prec z_2$ holds but $z_{10} \prec z_{2} \prec z_{3}$ does not; {\bf (c)} $\mathbb{D}_5$-symmetric billiard with two $(5,2)$-periodic Birkhoff orbits, the ``short'' one in red and the ``long'' one in cyan; 
{\bf (d)}  $\mathbb{D}_5$-symmetric billiard with $(15,6)$-periodic non-Birkhoff orbit: note that $z_6\prec z_{1} \prec z_2$ holds but $z_7 \prec z_{2} \prec z_{3}$ does not.
\label{fig:first_example}}
\end{figure}

\noindent  
For later reference, we recall the well-known fact that  periodic Birkhoff sequences always have the minimal possible period given their rotation number.
\begin{proposition}\label{prop:pqnm}
Let $x \in \mathbb{X}_{p,q}$ be  Birkhoff, where $\frac{q}{p}=\frac{m}{n}$ with $\gcd(m,n)=1$. Then $x\in \mathbb{X}_{n,m}$. 
\end{proposition}
\noindent For a proof, we refer to  \cite{mramor2012ghost} (Theorem 3.13 and the paragraph below). We finish this section with several remarks. 

\begin{remark} Let $z\in \Gamma^{\mathbb{Z}}$ be a periodic billiard orbit of minimal period $n\geq 3$ and winding number $m$. We claim that $z$  is Birkhoff precisely when it defines a trajectory that is homeomorphic to a regular polygon of \emph{Schl\"afli symbol} $\{n/m\}$ \cite{coxeter1973regular,schlafli1901theorie}. Recall that this polygon is obtained by placing $n$ distinct points on a circle and connecting  pairs of points that have exactly $m-1$ points lying strictly ``between'' them. 

To prove this claim,  define
$$M := \sharp \{ Z\in \{z_1, \ldots, z_n\}\,|\, Z\neq z_i\ \mbox{and}\  z_i \prec Z \prec z_{i+1} \}\, $$
to be the number of distinct points on the sequence lying ``between'' $z_i$ and $z_{i+1}$ (excluding $z_i$ and including $z_{i+1}$). The Birkhoff property ensures that this quantity is independent of $i\in \mathbb{Z}$, and clearly $0<M<n$. We claim that $M=m$. Indeed, let $x\in \Sigma$ be a lift of $z$, and denote $$E := \{x_j + k \ | \ j,k\in \mathbb{Z}\ \}\subset \mathbb{R}\, .$$ 
By definition of $M$, and because $0<x_{i+1}-x_i<1$, we have that $E\cap (x_i, x_{i+1}]$ has cardinality $M$ for every $i\in \mathbb{Z}$. As a consequence, $E\cap (x_0, x_{n}]$ has cardinality $Mn$. However, $x_n=x_0+m$ by definition of the winding number. Since $E\cap (x_0, x_{0}+m]$ clearly has cardinality $mn$ (this holds for the intersection of $E$ with any interval of integer length $m$), it follows that $M=m$. 

We conclude that when we connect each $z_i$ and $z_{i+1}$ by a line segment, then we obtain a homeomorphic copy of a regular polygon of Schl\"afli symbol $\{n/m\}$.  
\end{remark}

\begin{remark}
It is well known that every real-valued Birkhoff sequence $x$ has a well-defined rotation number, that is, the limit $\lim_{i\to\pm \infty} \frac{x_i}{i}$ exists \cite{gole2001symplectic}. 
\end{remark} 
\begin{remark}
A fundamental theorem from Aubry-Mather theory \cite{gole2001symplectic, mramor2012ghost} implies that if $\Gamma$ is a strictly convex and $C^2$-smooth billiard, then for every $0<\omega<1$ there exists a billiard orbit with a lift $x\in \Sigma$ of  rotation number $\omega$. This $x$ satisfies the Birkhoff property, and it is an action-maximizer. The latter means that the length of each finite segment $(x_i, x_{i+1}, \ldots, x_{i+M-1}, x_{i+M})$ of $x$ is maximal among sequences $(x_i, y_{i+1}, \ldots, y_{i+M-1}, x_{i+M})$ with the same endpoints. For $\omega=\frac{m}{n}\in \mathbb{Q}$ (with $\gcd(m,n)=1$) this action-maximizing Birkhoff orbit satisfies $x_{i+n}=x_i+m$. In other words, it represents one of the two periodic solutions guaranteed by Birkhoff's twist map theorem.
\end{remark}
\begin{remark}\label{rem:birkhoff}
For real-valued sequences, one may alternatively define the Birkhoff property as follows. For two sequences $x, y\in\mathbb{R}^{\mathbb{Z}}$ we write $x \leq y$ if $x_i \leq y_i$ for all $i \in \mathbb{Z}$. One says that $x$ and $y$ are \emph{ordered} if either $x \leq y$ or $y\leq x$. For $x\in \mathbb{R}^{\mathbb{Z}}$ and $m,l\in \mathbb{Z}$ one may define the \emph{integer translate} $\tau_{m,l}x\in \mathbb{R}^{\mathbb{Z}}$ of $x$ by $(\tau_{m,l}x)_i=x_{i+m}+l$. It is not hard to see that $x$ is Birkhoff (according to Definition \ref{Birkhoff}) if and only if any two of its integer translates are ordered. We refer to \cite{mramor2012ghost} for a proof of this claim. 
\end{remark}

\begin{remark}\label{rem:Bir_no_touch}
A sequence $x\in\mathbb{R}^{\mathbb{Z}}$ can be visualized through what is known as its \emph{Aubry diagram}. This diagram consists of all the points $(i,x_i)\in\mathbb{R}^2$ and straight line segments connecting all neighboring points $(i,x_i)$ and $(i+1, x_{i+1})$. A sequence is Birkhoff precisely when the Aubry diagrams of all its integer translates do not cross. 
Figure \ref{fig:Aubrydiagram} displays Aubry diagrams of a Birkhoff sequence and five of its translates.
\end{remark}
\vspace{-2mm}
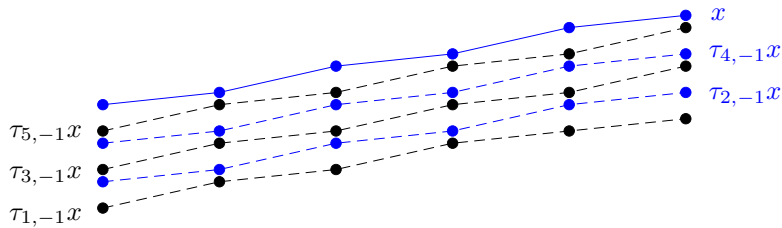
\begin{figure}[h!]
\centering
\begin{tikzpicture}
  \node at (0,0) {\includegraphics[width=0.5\linewidth]{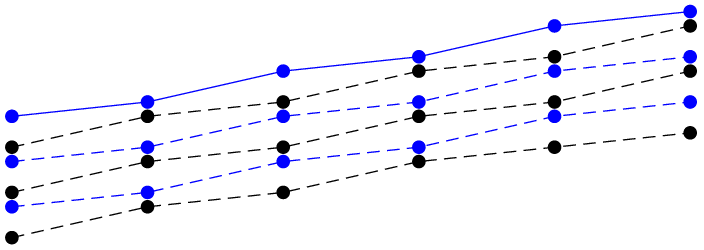}};
\node at (4.25,1.3) {\textcolor{blue}{$x$}};
\node at (4.6,0.75) {\textcolor{blue}{$\tau_{4,-1}x$}};
\node at (4.6,0.2) {\textcolor{blue}{$\tau_{2,-1}x$}};
\node at (-4.6,-0.3) {$\tau_{5,-1}x$};
\node at (-4.6,-.85) {$\tau_{3,-1}x$};
\node at (-4.6,-1.4) {$\tau_{1,-1}x$};
  \end{tikzpicture}  
  \caption{Aubry diagrams of a Birkhoff sequence (solid line) and five of its integer translates (dashed); recall Remark \ref{rem:Bir_no_touch}. \label{fig:Aubrydiagram}}
\end{figure}
\vspace{-2mm}
\section{Spatiotemporal symmetry}\label{sec:periodicorbitssection}
\noindent Recall that  a billiard $\Gamma$ is called $G$-symmetric if  $g(\Gamma) = \Gamma$ for all $g\in G$, and that throughout this paper, we assume that $G=\mathbb{D}_n = \langle R, S\rangle$. Here, $R$ is the counterclockwise rotation over an angle $\frac{2\pi}{n}$ and $S$ is the reflection in the horizontal axis. Without loss of generality, we will now also assume that the parametrization $\gamma: \mathbb{R}\to\mathbb{R}^2$ of $\Gamma$ is $\mathbb{D}_n$-equivariant, by which we mean that 
\begin{align} \label{equivarianceofgamma}
R(\gamma(x))=\gamma(x+1/n) \ \mbox{ and }\  S(\gamma(x))=\gamma(-x)\  \mbox{ for all }\ x\in \mathbb{R}\, .
\end{align}
\begin{example}\label{ex:limacon}
Consider the Lima\c{c}on-type curve which is the image of the map $\gamma_{\alpha}: \mathbb{R}\rightarrow \mathbb{R}^2$ defined by
\begin{equation}\label{eq:limac-like}
\gamma_{\alpha}(x)=r_{\alpha}(x)(\cos (2\pi x), \sin (2 \pi x)) \ \mbox{in which}\ r_{\alpha}(x) = 1+ \alpha \cos(2\pi n x)\,  .
\end{equation}
This $\gamma$ satisfies \eqref{equivarianceofgamma} and is thus $\mathbb{D}_n$-equivariant.  For $| \alpha| < 1$ it defines an embedding on $\mathbb{R}/\mathbb{Z}$. In Appendix \ref{app:convexity} we show that $\gamma_{\alpha}$ bounds a strictly convex  domain if and only if 
$$
| \alpha|\leq \alpha^*(n):=\dfrac{1}{1+n^2}\, .
$$
We used Lima\c{c}on-type billiards to generate all the figures in this paper. In Appendix \ref{app:code} we briefly explain how we numerically found non-Birkhoff periodic orbits in these billiards, and we provide a reference to a GitHub repository containing three examples of the Matlab codes used to produce our figures.

Figure \ref{fig:intro1} was produced with $\alpha=0.19<\alpha^*(2)=0.2$; Figure \ref{fig:intro2} was produced with $\alpha=0.099<\alpha^*(3)= 0.1$; Figure \ref{fig:intro3} was produced with $\alpha=0.05<\alpha^*(4)\approx 0.0588$; Figure \ref{fig:intro4} was produced with $\alpha=0.0005\ll \alpha^*(5)\approx 0.0385$. Figure \ref{billiard_law} was produced with $\alpha=0.099<\alpha^*(3)=0.1$. Figures \ref{fig:sub1} and \ref{fig:sub2} were produced with $\alpha=0.05<\alpha^*(4)\approx 0.0588$; Figures \ref{fig:sub3} and \ref{fig:sub4} were produced with $\alpha=0.035<\alpha^*(5)\approx 0.0385$.
\end{example} 
\noindent 
It follows from \eqref{equivarianceofgamma} that 
\begin{equation}\label{eq:invarianceL}
L(x+1/n, X+1/n)=L(x,X)\ \mbox{and} \  L(-x, -X)=L(x,X)\, .
\end{equation}
The reason is simply that $\|\gamma(x+1/n)-\gamma(X+1/n)\| = \|R(\gamma(x))-R(\gamma(X))\| = \|\gamma(x)-\gamma(X)\|$, and  $\|\gamma(-x)-\gamma(-X)\| = \|S(\gamma(x))-S(\gamma(X))\| = \|\gamma(x)-\gamma(X)\|$.
\noindent The following trivial proposition is an immediate consequence of \eqref{eq:invarianceL}. It states that any symmetry of a billiard is also a symmetry of the set of billiard orbits.
\begin{proposition}\label{prop:symaction}
Let $\Gamma$ be a $\mathbb{D}_n$-symmetric billiard, and let $i\mapsto z_i=\gamma(x_i)\in \Gamma$ be a billiard orbit (where $x\in \Sigma$). Then also $i\mapsto R(z_{i})=\gamma(x_{i}+1/n)$ and $i\mapsto S(z_{i})=\gamma(-x_i)$ are billiard orbits.
\end{proposition}
\begin{proof}   
Recall that a billiard sequence $z_i=\gamma(x_i)$ is a billiard orbit if and only if it satisfies the recurrence relation \eqref{eq:recrelgeneral}. The identity $L(x+1/n, X+1/n)=L(x,X)$ implies that $\partial_1 L(x+1/n, X+1/n)=\partial_1 L(x,X)$ and $\partial_2 L(x+1/n, X+1/n)=\partial_2 L(x,X)$. As a result, 
$$\partial_2 L(x_{i-1}+1/n, x_{i}+1/n) +\partial_1 L(x_{i}+1/n, x_{i+1}+1/n) = \partial_2 L(x_{i-1}, x_{i}) +\partial_1 L(x_i, x_{i+1}) \, .$$
So the sequence $i\mapsto R(z_i)=\gamma(x_i+1/n)$ is a billiard orbit if and only if $z_i = \gamma(x_i)$ is.

From $L(-x,-X)=L(x,X)$ it follows that $\partial_1 L(-x, -X)= - \partial_1 L(x,X)$ and $\partial_2 L(-x, -X)= - \partial_2 L(x,X)$, and therefore
$$\partial_2 L(-x_{i-1}, -x_{i}) +\partial_1 L(-x_{i}, -x_{i+1}) = -  \partial_2 L(x_{i-1}, x_{i}) - \partial_1 L(x_i, x_{i+1})  \, .$$
So the sequence $i\mapsto S(z_i)=\gamma(-x_i)$ is a billiard orbit if and only if $z_i = \gamma(x_i)$ is.
\end{proof}
\noindent 
We recall from Remark \ref{remk:autonomousreversible} that when $i\mapsto z_i$ is a billiard orbit, then so are the time-translated orbit $i\mapsto z_{k+i}$ and the time-reversed orbit $i\mapsto z_{k-i}$ for any $k\in \mathbb{Z}$. In the introduction, we defined a spatiotemporal symmetry of a billiard orbit $z\in \Gamma^{\mathbb{Z}}$ to be a group element $h\in \mathbb{D}_n$ with the property that the transformed orbit $i\mapsto h(z_i)$ is one of these translated or reversed versions of the orbit $z$ itself. Here we  repeat this definition, and we extend it to general billiard sequences.
  
\begin{definition}
Let $\Gamma$ be a $\mathbb{D}_n$-symmetric billiard  and $z=(\ldots, z_{-1}, z_0, z_1, \ldots)\in \Gamma^{\mathbb{Z}}$ a billiard sequence. A group element $h\in \mathbb{D}_n$ is called a \emph{time-preserving symmetry} of $z$ if there is a $k\in \mathbb{Z}$ such that $h(z_i)=z_{k+i}$ for all $i\in \mathbb{Z}$, and a \emph{time-reversing symmetry} of $z$ if there  is a $k\in \mathbb{Z}$ such that $h(z_i)=z_{k-i}$ for all $i\in \mathbb{Z}$. The group of time-preserving symmetries of $z$ is denoted $H^+(z)$, the set of time-reversing symmetries of $z$ is denoted $H^-(z)$, and the group $H(z):=H^+(z) \cup H^-(z)$ is called the \emph{spatiotemporal symmetry group} of $z$. 
\end{definition}
\noindent It may occasionally occur that the intersection $H^+(z) \cap H^-(z)$ is nonempty, that is, a symmetry $h\in \mathbb{D}_n$ may be both a time-preserving and a time-reversing symmetry of a given billiard sequence, see for example Lemma \ref{lem:D1orbitlemma} below.

The following simple but crucial lemma translates any  spatiotemporal symmetry of a billiard sequence into a set of  linear inhomogeneous equations for its lift. Figure \ref{fig:lemma_4.6} illustrates this result by depicting symmetric billiard orbits and Aubry diagrams of their lifts. 

\begin{figure}[h!]
\centering
\begin{subfigure}{.35\textwidth}
  \centering
  \begin{tikzpicture}
  \node at (0,0) {\includegraphics[width=\linewidth]{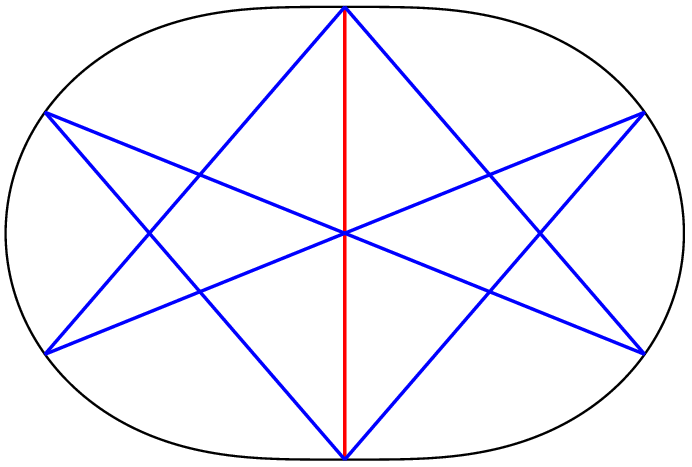}};
\node at (2.5,1.2) {$1$};
\node at (0,2.1) {$5$};
\node at (2.5,-1.2) {$4$};
\node at (-2.5,1.2) {$3$};
\node at (0,-2.1) {$2$};
\node at (-2.5,-1.2) {$6$};
  \end{tikzpicture}
  \caption{}
  \label{fig:sub3_new1}
\end{subfigure}\hspace{2cm}
\begin{subfigure}{.35\textwidth}
  \centering
    \begin{tikzpicture}
 \node at (0,0) {\includegraphics[width=\linewidth]{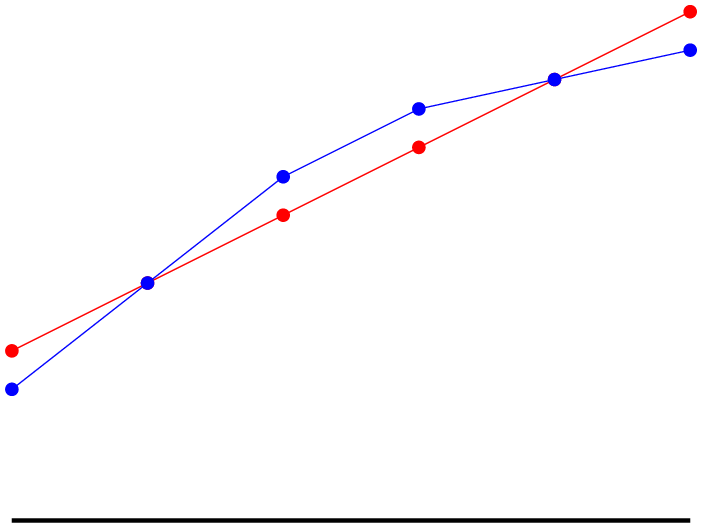}};
 \node at (-2.7,-2.3) {$1$};
 \node at (-1.7,-2.3) {$2$};
 \node at (-0.6,-2.3) {$3$};
 \node at (0.5,-2.3) {$4$};
 \node at (1.6,-2.3) {$5$};
 \node at (2.7,-2.3) {$6$};
  \end{tikzpicture}  
  \caption{}
  \label{fig:sub4_new1}
\end{subfigure}\\
\centering
\begin{subfigure}{.35\textwidth}
  \centering
  \begin{tikzpicture}
  \node at (0,0) {\includegraphics[width=\linewidth]{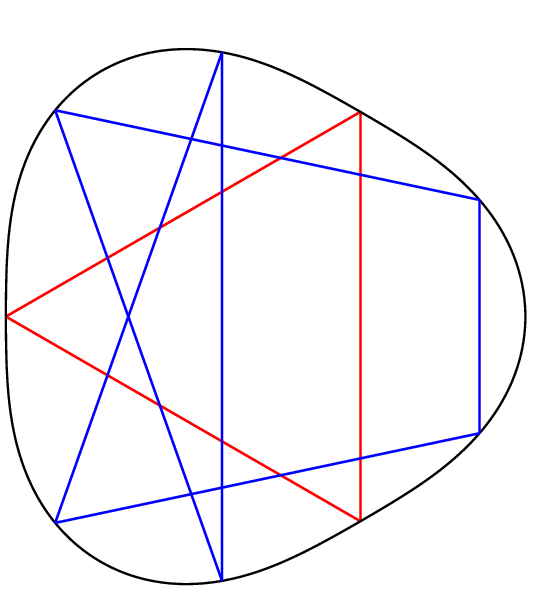}};
\node at (2.3,1.2) {$1$};
\node at (-2.4,2.1) {$2$};
\node at (-0.4,-3.2) {$3$};
\node at (-0.4,2.8) {$4$};
\node at (-2.4,-2.5) {$5$};
\node at (2.3,-1.6) {$6$};
  \end{tikzpicture}
  \caption{}
  \label{fig:sub1_new1}
\end{subfigure}\hspace{2cm}
\begin{subfigure}{.35\textwidth}
  \centering
    \begin{tikzpicture}
  \node at (0,0) {\includegraphics[width=\linewidth]{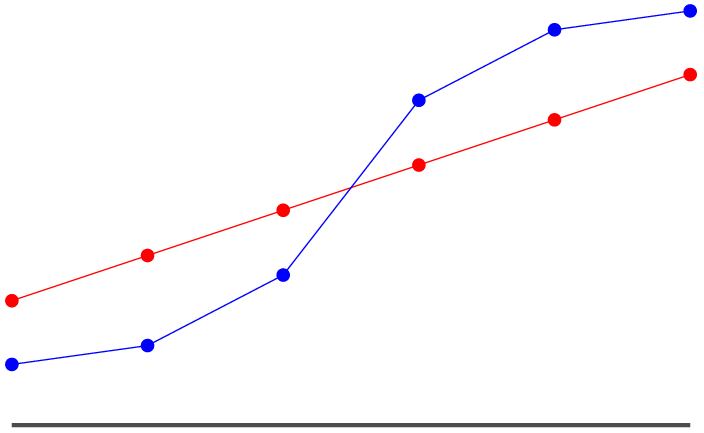}};
 \node at (-2.7,-1.9) {$1$};
 \node at (-1.7,-1.9) {$2$};
 \node at (-0.6,-1.9) {$3$};
 \node at (0.5,-1.9) {$4$};
 \node at (1.6,-1.9) {$5$};
 \node at (2.7,-1.9) {$6$};
  \end{tikzpicture}  
  \caption{}
  \label{fig:sub2_new1}  
\end{subfigure}
\caption{Symmetric billiard orbits have symmetric Aubry diagrams by Lemma \ref{liftlemma}. {\bf (a)} Convex $\mathbb{D}_2$-symmetric billiard of Lima\c con-type  (see Example \ref{ex:limacon}) with parameter $\alpha=0.1995<\alpha^*(2)= 0.2$. possessing a $(2,1)$-periodic Birkhoff orbit (red) and a $(6,3)$-periodic non-Birkhoff  orbit (blue), both satisfying $S(z_i)=z_{3+i}$ and $R(z_i)=z_{7-i}$; {\bf (b)} Aubry diagrams of lifts of these orbits satisfying $x_i+x_{3+i}=i$ and $x_{7-i}-x_i=\frac{1}{2}-i$; {\bf (c)} Convex $\mathbb{D}_3$-symmetric billiard of Lima\c con-type  (see Example \ref{ex:limacon}) with parameter $\alpha=0.09<\alpha^*(3)= 0.1$, possessing a $(3,1)$-periodic Birkhoff orbit (red) and a $(6,2)$-periodic non-Birkhoff  orbit (blue), both satisfying $S(z_i)=z_{7-i}$; {\bf (d)} Aubry diagrams of lifts of these orbits satisfying $x_i+x_{7-i}=0$.
\label{fig:lemma_4.6}}
\end{figure}
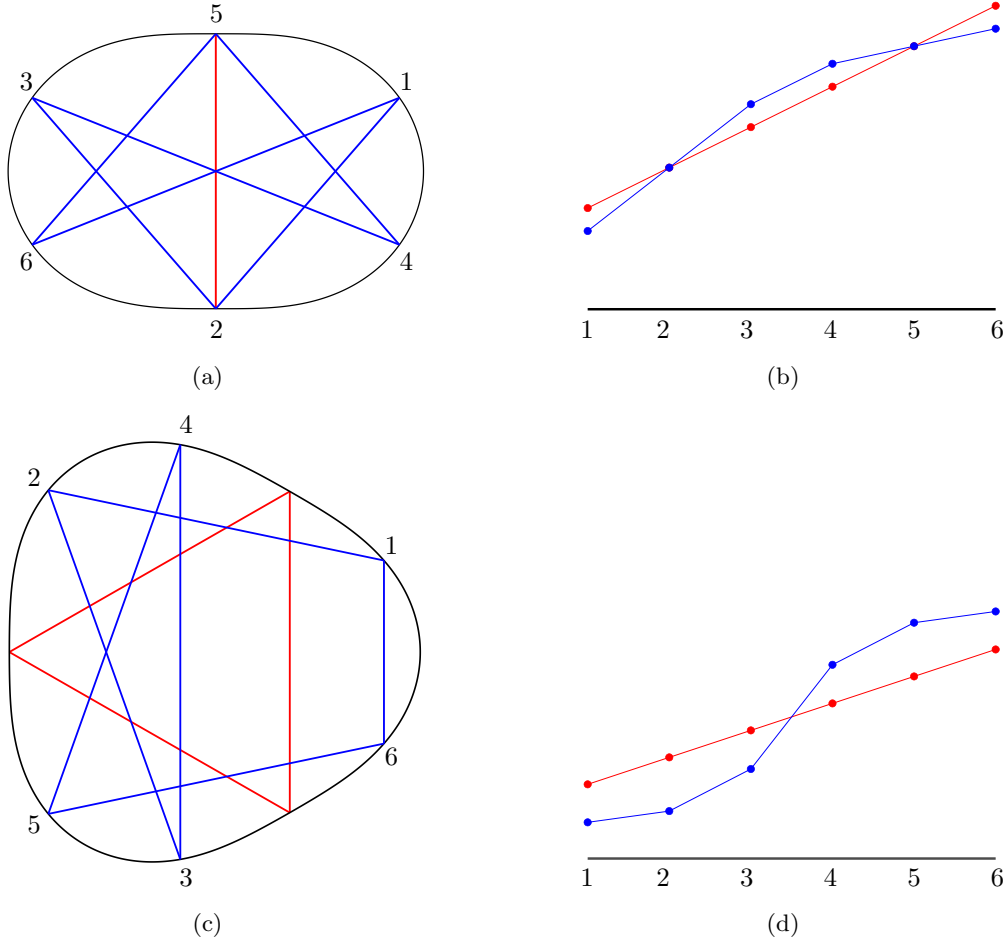
\begin{lemma} \label{liftlemma}
Let $z\in \Gamma^{\mathbb{Z}}$ be a billiard sequence with lift $x\in \Sigma$. Then
\begin{itemize}
\item[{\it i)}] $R^a(z_{i}) = z_{k+i}$ for all $i\in \mathbb{Z}$, if and only if $x_{k+i} - x_i = \frac{a}{n} + M$ for some $M\in \mathbb{Z}$ and all $i\in \mathbb{Z}$.
\item[{\it ii)}] $R^a(z_i) = z_{k-i}$ for all $i\in \mathbb{Z}$, if and only if $x_{k-i} - x_i = \frac{a}{n} + M -i$ for some $M\in \mathbb{Z}$ and all $i\in \mathbb{Z}$.
\item[{\it iii)}] $(R^bS)(z_{i}) = z_{k+i}$ for all $i\in \mathbb{Z}$, if and only if $x_i + x_{k+i} = \frac{b}{n} + M + i$ for some $M\in \mathbb{Z}$ and all $i\in \mathbb{Z}$. 
\item[{\it iv)}] $(R^bS)(z_{i}) = z_{k-i}$ for all $i\in \mathbb{Z}$, if and only if $x_i + x_{k-i} = \frac{b}{n} + M$ for some $M\in \mathbb{Z}$ and all $i\in \mathbb{Z}$.
\end{itemize}  
\end{lemma}
\begin{proof}
Our assumptions on $z$ and $x$ imply that $z_i=\gamma(x_i)$ and that $0< x_{i+1}-x_i<1$ for all $i\in\mathbb{Z}$.
\begin{itemize}
\item[{\it i)}]  Note that $R^a(z_{i}) =R^a(\gamma(x_{i})) = \gamma\left(x_i+\frac{a}{n}\right)$ and $z_{k+i}= \gamma(x_{k+i})$, so   $R^a(z_{i}) =z_{k+i}$ if and only if $x_i+\frac{a}{n}  = x_{k+i} - M(i)$ for integers $M(i)$. In particular, if $R^a(z_i) = z_{k+i}$, then $x_{i+1}-x_i =  (x_{k+i+1} - M(i+1)-\frac{a}{n})  -  (x_{k+i}-M(i) -\frac{a}{n})= x_{k+i+1} - x_{k+i} + M(i)-M(i+1)$. The only way the left- and right-hand side can both be between $0$ and $1$ is if $M(i+1)=M(i)$. Thus, $M(i)=M$ is independent of $i$, and $x_{k+i} - x_i = \frac{a}{n} + M$. Conversely, if $x_{k+i} - x_i = \frac{a}{n} + M$ for some $M\in \mathbb{Z}$ and all $i\in \mathbb{Z}$, then $x_{k+i} = x_i + \frac{a}{n} + M$ and applying $\gamma$ to both sides gives $R^a(z_{i}) = z_{k+i}$. 
\item[{\it ii)}]  Now we use that $R^a(z_{i}) =z_{k-i}$ if and only if $x_i+\frac{a}{n}  = x_{k-i} - M(i)$ for integers $M(i)$.  In particular, if $R^a(z_{i}) =z_{k-i}$, then  $x_{i+1}-x_i =  (x_{k-i-1} - M(i+1) -\frac{a}{n}) -  (x_{k-i}-M(i)-\frac{a}{n})= x_{k-i-1} - x_{k-i} + M(i)-M(i+1)$. The only way the left- and right-hand side  can both be between $0$ and $1$ is if $M(i)-M(i+1)=1$. Thus, $M(i)=M - i$, and $x_{k-i} - x_i = \frac{a}{n} + M -i$. The converse implication follows from applying $\gamma$ to $x_{k-i}=x_i+\frac{a}{n} + M - i$.
\item[{\it iii)}]  We have $(R^bS)(z_i) = z_{k+i}$ if and only if $\frac{b}{n} -x_i = x_{k+i} - M(i)$ for  integers $M(i)$.  So if $(R^bS)(z_i) = z_{k+i}$ then $x_{i+1}-x_i = - ( x_{k+i+1} - M(i+1)-\frac{b}{n}) +  (x_{k+i}-M(i) -\frac{b}{n})= x_{k+i} - x_{k+i+1} + M(i+1)-M(i)$. The only way the left- and right-hand sides can both be between $0$ and $1$ is if $M(i+1)-M(i)=1$. Thus, $M(i)=M + i$, and $x_i+x_{k+i}=\frac{b}{n}+ M+i$. The converse implication is obvious.
 \item[{\it iv)}] We have $(R^bS)(z_i) = z_{k-i}$ if and only if $\frac{b}{n}-x_i = x_{k-i} - M(i)$ for  integers $M(i)$.  So $x_{i+1}-x_i = - ( x_{k-i-1} - M(i+1) -\frac{b}{n} ) +  (x_{k-i}  -M(i)-\frac{b}{n})= x_{k-i} - x_{k-i-1} + M(i)-M(i+1)$. The only way the left-hand side and right-hand side can both be between $0$ and $1$ is if $M(i+1)=M(i)$. Thus, $M(i)=M$ is independent of $i$, and $x_i+x_{k-i}=\frac{b}{n}+M$. The converse implication is obvious.
    \end{itemize} 
    \vspace{-.3cm}
\end{proof}

\begin{remark}\label{remk:rotationnumberhalf}
Assume that $z\in \Gamma^{\mathbb{Z}}$ has a rotation number, meaning that for any lift $x\in \Sigma$ of $z$, the limit $\omega:=\lim_{i\to \pm \infty} \frac{x_i}{i}$ exists. If in addition $R^a(z_i)=z_{k-i}$ for all $i\in \mathbb{Z}$, then it follows from part {\it ii)} of Lemma \ref{liftlemma} \changes{that $x_{k-i} - x_i = \frac{a}{n} + M -i$. Dividing both sides of this equation  by $i$ and taking  the limit $i\to \infty$, we obtain that $-2\omega=-1$, as $\lim_{i\to \infty} \frac{x_{k-i}}{i}=\lim_{i\to \infty} \frac{k-i}{i} \cdot\frac{x_{k-i}}{k-i} =-\omega$.} In other words, a billiard sequence that is invariant under a rotation acting time-reversing, can only have rotation number $\frac{1}{2}$. 
    
If $(R^bS)(z_i)=z_{k+i}$ for all $i\in \mathbb{Z}$, then part {\it iii)} of Lemma \ref{liftlemma} implies that any lift $x\in \Sigma$ of $z$ must satisfy $x_i = -x_{k+i} +M+i = -(-x_{k+(k+i)}+M+(k+i))+M+i = x_{2k+i}-k$. \changes{This means that $x\in \mathbb{X}_{2k,k}$, and in particular that it has rotation number $\frac{k}{2k}=\frac{1}{2}$.} We conclude that a billiard sequence that is invariant under a reflection acting time-preserving, is necessarily periodic, and  must have rotation number $\frac{1}{2}$. 
\end{remark}

\section{Symmetric periodic billiard sequences}\label{sec:symmetrygradientsection}
\noindent In this section, we classify periodic billiard sequences with dihedral spatiotemporal symmetry groups. We recall that we call a subgroup $H\subset \mathbb{D}_n$ {\it dihedral} if it contains at least one reflection. In fact, when $H\cong \mathbb{D}_N$ has order $2N$, then $N$ is a divisor of $n$, and $H$ is generated by a rotation of order $N$ and a reflection. 
Our first result is trivial, but we formulate it  for convenience and  completeness. 
\begin{lemma}\label{lem:periodicp=2}
Let $n\geq 2$ and let $\Gamma$ be a $\mathbb{D}_n$-symmetric billiard. Let $H\subset \mathbb{D}_n$ be a dihedral subgroup of order $2N$ (for $N\geq 1$), and let $z = (\ldots, z_{-1}, z_0, z_1, \ldots) \in \Gamma^{\mathbb{Z}}$ be a  billiard sequence of minimal period $p=2$ with spatiotemporal symmetry group $H$. Then $N=1$ or $N=2$. 
\end{lemma}
\begin{proof}
Let $z$ have period $p=2$ and assume that $|H| = 2N$ with $N>1$. This implies that $H=\langle \rho, \sigma\rangle$ is generated by a nontrivial rotation $\rho$ of order $N$ and a nontrivial reflection $\sigma$. It holds that $\rho(z_i)\neq z_i$ because $\rho$ does not have any fixed points on $\Gamma$. But then necessarily $\rho^{2}(z_i) =z_i$ because $z$ has period $2$, and therefore $\rho^2={\rm Id}$, i.e., $\rho$ has order $2$. Thus $N=2$. In general, there are no restrictions on the action of $\sigma$, so either $\sigma(z_i)=z_i$, in which case $z$ consists  of points that are fixed by $\sigma$, or $\sigma(z_i)\neq z_{i}$, in which case $\sigma$ interchanges the two distinct points on the orbit.
\end{proof}
\noindent Symmetric billiard sequences of period $p\geq 3$ can be of  distinct types, as described in the following lemmas. We distinguish between the cases $N=1, N=2$ and $N\geq 3$. We treat the latter case first.
\begin{lemma}\label{lem:periodicn=2}
Let $n\geq 3$ and let $\Gamma$ be a $\mathbb{D}_n$-symmetric billiard. Let $H=\langle \rho, \sigma\rangle \subset \mathbb{D}_n$ be a dihedral subgroup of order $2N$ with $N\geq 3$, generated by a rotation $\rho$ of order $N$ and a reflection $\sigma$. 

Let $z = (\ldots, z_{-1}, z_0, z_1, \ldots) \in \Gamma^{\mathbb{Z}}$ be a billiard sequence of minimal period $p\geq 3$ with spatiotemporal symmetry group $H$. 
Then $p$ is an integer multiple of $N$, there is a unique $1\leq M\leq N-1$ with $\gcd(M,N)=1$ such that $\rho^M(z_i)=z_{\frac{p}{N}+i}$, and there is a unique  $0\leq  k < p$ such that $\sigma(z_i)=z_{k-i}$. 
\end{lemma}
\begin{proof}
Let $H$ be a dihedral subgroup of $\mathbb{D}_n = \langle R, S\rangle$ of order $2N\geq 6$. Then $N$ is a divisor of $n$, and $H$ is generated by a rotation $\rho$ of order $N$ and a reflection $\sigma$. Let $z$ have spatiotemporal symmetry group $H$. If $\rho(z_i)=z_{k-i}$ for some $k\in \mathbb{Z}$, then $\rho^2(z_i)=z_{k-(k-i)}=z_i$. This means that all $z_i$ are fixed by $\rho^2$, which is impossible because $N\geq 3$, so that $\rho^{2}\neq {\rm Id}$ is a nontrivial rotation. So it must hold that $\rho(z_i)=z_{k+i}$ for some (unique) $0< k < p$.  Note that the possibility that $k=0$ is excluded, because then $\rho(z_i)=z_i$, and all $z_i$ would be fixed by $\rho$.

From the fact that $\rho(z_i)=z_{k+i}$, it follows that  $z_i=\rho^{N}(z_i)=z_{Nk+i}$, which in turn implies that $Nk = tp$ for some $t\in \mathbb{N}$. We claim that $\gcd(t,N)=1$. If not then we could divide $t$ and $N$ by this common divisor to find $\tilde N K = \tilde t p$ for $\tilde N < N$. But then $\rho^{\tilde N}(z_i) = z_{i+\tilde N k} = z_{i+\tilde t p} = z_i$. This would mean that $\rho$ would have order $\tilde N < N$, which is nonsense. In particular, there are $s,M\in \mathbb{Z}$ so that $sN+tM=1$. It holds that $\gcd(M,N)=1$, and if we ask that $1\leq M\leq N-1$ then this makes $M$ unique. It follows that $p=sNp+tMp = sNp+MNk = N(sp+Mk)$, which shows that $p$ is an integer multiple of $N$ (the first statement of the lemma), and hence that $k=t\frac{p}{N}$.  We thus  find that $\rho^M(z_i)=z_{Mk+i}= z_{Mt\frac{p}{N}+i}=z_{\frac{p}{N}-sp+i} = z_{\frac{p}{N}+i}$. This proves the second statement of the theorem. 

We claim that it is not possible that $\sigma(z_i)=z_{k+i}$ for some $k\in \mathbb{Z}$.  Indeed, this would imply that $z_{-\frac{p}{N} + k+i} = (\rho^{-M}\sigma)(z_i) = (\sigma \rho^M)(z_i) = z_{\frac{p}{N} + k+i}$, so $-\frac{p}{N}=\frac{p}{N}\!\mod \, p$, or $\frac{2p}{N}=0\!\mod\, p$. Because $N\geq 3$,  this is  impossible. We conclude that $\sigma(z_i)=z_{k-i}$ for some (unique) $0\leq k < p$ (the third statement of the lemma).  
\end{proof}
\noindent The next lemma describes the case that $N=2$ and $p\geq 3$.
\begin{lemma}\label{lem:periodicngeq3}
Let $\Gamma$ be $\mathbb{D}_n$-symmetric and let $H=\langle \rho, \sigma\rangle \subset \mathbb{D}_n$ be a dihedral subgroup of order $4$ generated by a rotation $\rho$ of order $2$ and a reflection $\sigma$.   Let $z = (\ldots, z_{-1}, z_0, z_1, \ldots) \in \Gamma^{\mathbb{Z}}$ be a billiard sequence of minimal period $p\geq 3$  with spatiotemporal symmetry group $H$. Then $p$ is even, and either
\begin{itemize}
\item[{\rm I)}] The rotation $\rho$ acts time-preserving and the reflection $\sigma$ acts time-reversing. More precisely, $\rho(z_i)=z_{\frac{p}{2}+i}$, and $\sigma(z_i)=z_{k-i}$ for a unique $0\leq  k < p$; or
\item[{\rm II)}] The rotation $\rho$ acts time-reversing and one of the two reflections in $H$ acts time-preserving. More precisely, $\rho(z_i)=z_{k-i}$ for a unique odd $0<k<p$, and either $\sigma(z_i)=z_{\frac{p}{2}+i}$, or $(\rho \sigma)(z_i) = z_{\frac{p}{2}+i}$. 
\end{itemize}
\end{lemma}
Two examples of orbits of type {\rm I} and two of type {\rm II} are presented in Figures \ref{fig:D2typeI} and \ref{fig:D2typeII}, respectively.
\begin{proof}
Let $H$ be as in the assumptions of the lemma and let $z$ have spatiotemporal symmetry group $H$. Note that necessarily $\rho = R^{n/2}$ because $\rho^2={\rm Id}$.  We first consider the case that $\rho$ acts time-preserving, i.e., that $\rho(z_i) = z_{k+i}$ for some $0< k < p$ (note that the case $k=0$ is excluded because then $\rho(z_i)=z_i$ whereas $\rho$ is a nontrivial rotation). It follows that  $z_i=\rho^2(z_i)=z_{i+2k}$, and we conclude that $p$ is even and $k =\frac{p}{2}$. 
    
We may similarly analyze the case that $\rho$ acts time-reversing, i.e., that $\rho(z_i)=z_{k-i}$. It follows from this that $k$ must be odd, as otherwise $\rho(z_{\frac{k}{2}})= z_{\frac{k}{2}}$. In turn, this implies that $p$ must be even; otherwise, $k+p$ would be even and 
$$\rho\left(z_{\frac{k+p}{2}}\right) = z_{k-\frac{k+p}{2}} = z_{\frac{k-p}{2}} = z_{\frac{k-p}{2}+p} = z_{\frac{k+p}{2}}.$$
To summarize, we proved that $p$ is always even, and that $\rho(z_i)=z_{\frac{p}{2}+i}$, or $\rho(z_i)=z_{k-i}$ for an odd $0<k<p$. 
    
Next, we consider the case that $\sigma$ acts time-preserving, i.e., that  $\sigma(z_i)=z_{k+i}$ for some $0< k<p$ (the option that $k=0$ is excluded as it would imply that $S$ has more than two fixed points on $\Gamma$, because $p\geq 3$). Then $z_i=\sigma^2(z_i)=z_{2k+i}$, so $k=\frac{p}{2}$. So if $\sigma$ acts time-preserving, then $\sigma(z_i)=z_{\frac{p}{2}+i}$. By exactly the same argument, if $\rho\sigma$ acts time-preserving, then $(\rho\sigma)(z_i)=z_{\frac{p}{2}+i}$.

After studying the possible  actions of $\rho$ and $\sigma$ on $z$, we now claim that if $\rho$ acts time-preserving, then $\sigma$ cannot act  time-preserving.  Indeed, this would imply that $(\rho\sigma)(z_i)=z_{\frac{p}{2} + \frac{p}{2} + i} =z_i$ for all $i\in \mathbb{Z}$, i.e., that $z$ consists entirely of fixed points of the reflection $\rho\sigma$, which is only possible if $p=2$. By exactly the same argument, $\rho$ and $\rho\sigma$ cannot both act time-preserving. So if $\rho$ acts time-preserving, then $\sigma$ and $\rho\sigma$ must act time-reversing. This is case I. 
   
In case $\sigma$ or $\rho\sigma$ acts time-preserving, then clearly $\rho$ must act time-reversing. This is case II. It is also clear that cases I and II are mutually exclusive.
\end{proof}
\noindent Our final lemma describes the case that $N=1$ and $p\geq 3$.
\begin{lemma}\label{lem:D1orbitlemma}
Let $\Gamma$ be $\mathbb{D}_n$-symmetric and let $H=\langle \sigma\rangle \subset \mathbb{D}_n$ be a dihedral subgroup of order $2$ generated by a reflection $\sigma$.   Let $z = (\ldots, z_{-1}, z_0, z_1, \ldots) \in \Gamma^{\mathbb{Z}}$ be a  billiard sequence of minimal period $p\geq 3$ with spatiotemporal symmetry group $H$. 

When the rotation number of $z$ is not equal to $\frac{1}{2}$, then $\sigma(z_i)=z_{k-i}$ for a unique $0<k<p$. When the rotation number of $z$  equals $\frac{1}{2}$, then either 
\begin{itemize}
\item[{\rm III)}]  There is a unique $0<k<p$ such that $\sigma(z_i)=z_{k-i}$; or
\item[{\rm IV)}] The period $p$ is even, and $\sigma(z_i)=z_{\frac{p}{2}+i}$; or
\item[{\rm V)}] Both {\rm III} and {\rm IV} hold. In this case there is an $a\in \mathbb{Z}$ such that  for every $t\in \mathbb{Z}$ it holds that $z_{\left(a+\frac{tp}{2}\right)+i} =  z_{\left(a+\frac{tp}{2}\right)-i}$ for all $i\in \mathbb{Z}$.
\end{itemize}  
\end{lemma}
Two examples of orbits of type {\rm V} are presented in Figure \ref{fig:lemma54}.
\begin{proof}
It was already shown in the proof of Lemma \ref{lem:periodicngeq3} that if $\sigma$ acts time-reversing, then  $\sigma(z_i)=z_{k-i}$ for a unique $0<k<p$. It was also shown that, if $\sigma$ acts time-preserving, then $p$ is even and $\sigma(z_i) = z_{\frac{p}{2}+i}$. It follows from Remark \ref{remk:rotationnumberhalf} that $z$ must have rotation number $\frac{1}{2}$ then. 
    
These two cases are not necessarily exclusive: when $z$ has rotation number $\frac{1}{2}$, then it may happen that $\sigma(z_i)=z_{\frac{p}{2}+i} = z_{k-i}$. In this situation, we have that $z_i = \sigma^2(z_i) = z_{(\frac{p}{2}+k) - i}$ for all $i\in \mathbb{Z}$. We claim that $\frac{p}{2}+k$ must be even then. Indeed, if $\frac{p}{2}+k = 2b+1$ would be odd, then $z_b=z_{(\frac{p}{2}+k) - b} = z_{2b+1-b}=z_{b+1}$, so $z$ would not  be a billiard sequence (as two of its consecutive coordinates are equal). So $\frac{p}{2}+k = 2a$ is even. It follows that for any $i, t\in \mathbb{Z}$ we have  $z_{\left(a+\frac{tp}{2}\right)+i} =   z_{2a-\left(a+\frac{tp}{2}+i\right)} =  z_{a-\frac{tp}{2}-i} = z_{a-\frac{tp}{2} + tp-i} =  z_{\left(a+\frac{tp}{2}\right)-i}$.  
In other words, the sequence is ``symmetric'' around $i=a+\frac{tp}{2}$.  
\end{proof}
\begin{remark} \label{remk:typeVremark} Let $z$ be a billiard sequence of minimal period $p>2$, such that $H(z)$ contains a reflection $\sigma$ acting both time-preserving and time-reversing. We claim that then $z$ is of type ${\rm V}$, that is, $H(z)=\langle \sigma \rangle$.

Indeed, if $H(z)$ was larger than $\langle \sigma \rangle$, then it would contain at least one rotation $\rho$, and it would hold that $\rho  = \sigma \rho^{-1}\sigma$. 
Because $\sigma$ acts both time-preserving and time-reversing, it follows that if $\rho$ (and hence also $\rho^{-1}$) acts time-preserving, then $\rho$ also acts time-reversing, and vice versa. We conclude that $\rho$ must have order $2$, so $H(z)\cong \mathbb{D}_2$, and that in fact each element of $H(z)$ acts both time-preserving and time-reversing. In particular, we have that $\sigma(z_i)=z_{\frac{p}{2}+i}$ and $\rho(z_i)=z_{\frac{p}{2}+i}$. But this would mean that the rotation $\rho$ and the reflection $\sigma$ act the same on at least three distinct point on $\Gamma$, a contradiction.

\end{remark}

\section{The gradient flow of the length functional}\label{sec:gradientflowsection}
\noindent  A crucial tool used in this paper is the so-called \emph{gradient flow} or \emph{curve-lengthening flow}. It is defined as the flow of the differential equation
\begin{equation}
\label{eq:gradientflow}
\dot x_i = F_i(x) := \partial_2 L(x_{i-1}, x_i) + \partial_1 L(x_i,x_{i+1})    \ \mbox{for} \ i\in \mathbb{Z}\, .
\end{equation}
Note that the equilibrium points of \eqref{eq:gradientflow} are precisely the solutions to \eqref{eq:recrelgeneral}. A solution curve $x(t)\in \Sigma$ of \eqref{eq:gradientflow}, defined for $0\leq t < t_0$ with $t_0\in (0,\infty]$, will be called a {\it gradient flow line}. Although the right-hand side of \eqref{eq:gradientflow} is well-defined when $x\in\Sigma$, gradient flow lines may not exist for all initial condition $x(0)\in \Sigma$. 
However, we can restrict equation \eqref{eq:gradientflow} to a subset of $\Sigma$ of the form
\begin{equation}\label{def:sigmadelta}
\Sigma_{\delta} :=\{ x \in\mathbb{R}^{\mathbb{Z}} \, |\, \delta  \leq  x_{i+1} - x_i \leq 1-\delta \ \mbox{for all}\ i\in \mathbb{Z}\, \}\ \mbox{ for some }\ 0 < \delta < 1/2\, .
\end{equation}
The right-hand side of \eqref{eq:gradientflow} defines a vector field that is uniformly Lipschitz continuous on each $\Sigma_{\delta}$ with respect to the norm $\|x\| = \sup_{i\in\mathbb{Z}} |x_i|$. It follows from this that for any initial condition $x(0)\in \Sigma_{2\delta}$ there is a locally unique gradient flow line $x(t)\in \Sigma_{\delta}$ (with $0\leq t<t_0$). In other words, the initial value problem for \eqref{eq:gradientflow} possesses the local-in-time existence and uniqueness property on  $\bigcup_{\delta >0} \Sigma_{\delta}$. See \cite{mramor2012ghost} for a proof of these facts. 

In the remainder of this section, we present several fundamental properties of the gradient flow. The first is the simple observation that it restricts to a finite-dimensional ODE on each of the affine spaces of $(p,q)$-periodic sequences. 
\begin{proposition}\label{prop:Xpqpreservation}
Let $\delta>0$ and let $x(t)\in \Sigma_{\delta}$ be a gradient flow line in $\Sigma_{\delta}$, defined for $0\leq t<t_0$. 
If  $x(0) \in \mathbb{X}_{p,q}$, then $x(t)\in \mathbb{X}_{p,q}$ for all $0\leq t<t_0$.
\end{proposition}
\begin{proof}
          Because $L(x+q, X+q) = L(x,X)$ for any $q\in \mathbb{Z}$, we have that $\partial_1 L(x+q, X+q) = \partial_1 L(x,X)$  and $\partial_2 L(x+q, X+q) = \partial_2 L(x,X)$. Thus, when $x\in \mathbb{X}_{p,q}$,
    \begin{align} \nonumber 
       F_{p+i}(x) & = \partial_2 L(x_{(p+i)-1}, x_{p+i}) + \partial_1 L(x_{p+i}, x_{(p+i)-1}) \\ \nonumber 
    & = \partial_2 L(x_{i-1}+q, x_{i}+q) + \partial_1 L(x_{i}+q, x_{i-1}+q) 
   \\ \nonumber & =
      \partial_2 L(x_{i-1}, x_{i}) + \partial_1 L(x_{i}, x_{i-1}) = F_i(x) \, .  
    \end{align}
   This shows that $F$ is tangent to $\mathbb{X}_{p,q}$. By uniqueness of solutions,  $x(t)\in\mathbb{X}_{p,q}$ as long as $x(t)\in \Sigma_{\delta}$.
\end{proof}

\noindent An analogous result holds for symmetric sequences:  
\begin{lemma}\label{lem:symmetrypreservation}
Let $\delta>0$ and let $x(t)\in \Sigma_{\delta}$ be a gradient flow line in $\Sigma_{\delta}$, defined for $0\leq t<t_0$. If $x(0)$ satisfies one or more of the equalities in Lemma \ref{liftlemma}, then so does $x(t)$ for all $0\leq t<t_0$. In particular, if $z(0):=\gamma(x(0))$ is $H$-symmetric, with $H\subset \mathbb{D}_n$ a subgroup, then so is $z(t):=\gamma(x(t))$. 
\end{lemma}
\begin{proof}
Recall from Lemma \ref{liftlemma} that each spatiotemporal symmetry of a billiard sequence $z_i=\gamma(x_i)$ corresponds to a collection of inhomogeneous linear equations for the lift $x\in \Sigma$. We prove that these equations are preserved by the gradient flow. We give full details for the equality $x_i+x_{k-i}=\frac{b}{n} + M$ (case {\it iv)} of Lemma \ref{liftlemma}), omitting a complete analysis of the other equalities.
 
Assume that $x\in \Sigma_{\delta}$ satisfies $x_i+x_{k-i}=\frac{b}{n}+M$ for some $M\in \mathbb{Z}$ and all $i\in \mathbb{Z}$. Recall that $L(x,X)=L(X,x) = L(x+1, X+1) = L(x+\frac{1}{n},X+\frac{1}{n})= L(-x,-X)$. In particular, $L(\frac{b}{n}+M -x, \frac{b}{n}+M -X)=L(X,x)$ so  $\partial_1 L(\frac{b}{n}+M-x,\frac{b}{n}+M-X) = - \partial_2 L(X,x)$ and $\partial_2 L(\frac{b}{n}+M -x, \frac{b}{n}+M -X) = - \partial_1 L(X,x)$. Thus, 
\begin{align}
F_i(x) & =  \partial_2 L(x_{i-1}, x_i) + \partial_1 L(x_i,x_{i+1}) =\nonumber \\ \nonumber 
 & = \partial_2 L(b/n+M-x_{k-(i-1)}, b/n+M-x_{k-i}) + \partial_1 L(b/n+M-x_{k-i},b/n+M-x_{k-(i+1)}) 
 \nonumber \\ \nonumber 
& = - \partial_1 L(x_{k-i}, x_{k-(i-1)}) - \partial_2 L(x_{k-(i+1)}, x_{k-i}) 
 \nonumber \\ \nonumber 
& = - \partial_2 L(x_{(k-i)-1}, x_{k-i}) - \partial_1 L(x_{k-i},x_{(k-i)+1}) 
    = -F_{k-i}(x)\, .
\end{align}   
This proves that the gradient vector field $F$ is tangent to $\{x\in \Sigma_{\delta}\, |\, x_i+x_{k-i}= M \ \forall \ i\in \mathbb{Z}\}$. 
  
The rest of the equalities is treated in a similar way. Specifically, for $x_{k-i}-x_i=a/n+M-i$ (case {\it ii)} of the lemma) we use that $L(x-a/n-M-i, X-a/n-M-i)=L(X,x)$.  For $x_i+x_{k+i}=M+i$ (case {\it iii)} of Lemma \ref{liftlemma}) we use that $L(M+i-x, M+i-X)=L(x,X)$. For $x_{k+i}-x_i=a/n+M$ (case {\it i)} of the lemma) we use that $L(x-a/n-M, X-a/n-M)=L(x,X)$. 
\end{proof}

\noindent 
For $x\in \mathbb{X}_{p,q} \cap \Sigma$, the billiard sequence $z_i=\gamma(x_i)$ has period $p$. The total length (per period) of this  sequence is given by  the so-called {\it periodic action} 
\begin{equation}\label{def:Wpq}
W_{p,q}:\mathbb{X}_{p,q} \to\mathbb{R}  
\ \mbox{ defined by } \ 
    W_{p,q}(x):=\sum_{j=1}^p L(x_{j}, x_{j+1})  \, .
    \end{equation}
This function is continuous on $\mathbb{X}_{p,q}$, and smooth on $\mathbb{X}_{p,q} \cap \Sigma$. We now prove that $-W_{p,q}$ is a Lyapunov function for the restriction of the gradient flow to $\mathbb{X}_{p,q}$, meaning that $W_{p,q}$ cannot decrease under the flow of \eqref{eq:gradientflow}.

\begin{proposition}\label{prop:Lyapunovfunction} Let $\delta>0$ and let $x(t)\in \mathbb{X}_{p,q}\cap\Sigma_{\delta}$ be a gradient flow line, defined for $0\leq t<t_0$. Then $\frac{d}{dt}W_{p,q}(x(t)) \geq 0$ for all $0\leq t<t_0$. We have $\left. \frac{d}{dt}\right|_{t=0}W_{p,q}(x(t)) = 0$  if and only if $x(t)$ is independent of $t$ and defines a solution to \eqref{eq:recrelgeneral}.
\end{proposition}
     \begin{proof}
    In the proof of Proposition \ref{prop:Xpqpreservation} we already showed that $F_{1}(x)=F_{p+1}(x)$  and that $\partial_2L(x_{p}, x_{p+1})= \partial_2L(x_{0}+q, x_{1}+q) = \partial_2L(x_{0}, x_{1})$ for all $x\in \mathbb{X}_{p,q}$.  For a gradient flow line $x(t)\in \mathbb{X}_{p,q}\cap \Sigma_{\delta}$, it follows that
     \begin{align} \nonumber 
     \frac{d}{dt}W_{p,q}(x(t)) & = \sum_{j=1}^p \left( \partial_1L(x_j(t), x_{j+1}(t))\frac{dx_j(t)}{dt} + \partial_2L(x_j(t), x_{j+1}(t))\frac{dx_{j+1}(t)}{dt}\right) \\
     \nonumber 
     & = \partial_1L(x_1(t), x_{2}(t))F_1(x(t)) + \sum_{j=2}^{p} \left| F_j(x(t))\right|^2 + \partial_2L(x_{p}(t), x_{p+1}(t))F_{p+1}(x(t))\\
     \nonumber 
     & = 
         \sum_{j=1}^{p} \left| F_j(x(t))\right|^2 \geq 0\, .
     \end{align}
We see from this formula that $\left. \frac{d}{dt}\right|_{t=0}W_{p,q}(x(t)) = 0$ if and only if $F_i(x(0))=0$ for all $1\leq i\leq p$. But $F_{p+i}(x(0))=F_i(x(0))$ for all $i\in \mathbb{Z}$, so this can only happen if $F_i(x(0))=0$ for all $i\in \mathbb{Z}$, that is, if $x(0)$ solves \eqref{eq:recrelgeneral}. By uniqueness of gradient flow lines, $x(t)=x(0)$ is then necessarily independent of $t$.
 \end{proof}
\noindent The next property of the gradient flow that we exploit in this paper is its so-called \emph{comparison principle}:
\begin{lemma} \label{lem:comparisonprinciple}
Let $\delta>0$ and let $x(t), y(t)\in \Sigma_{\delta}$ be two gradient flow lines in $\Sigma_{\delta}$, defined for $0\leq t<t_0$. Assume that $x_i(0) \leq y_i(0)$ for all $i\in \mathbb{Z}$, while  $x(0)\neq y(0)$. Then $x_i(t)<y_i(t)$ for all $i\in \mathbb{Z}$ and all $0< t<t_0$.
\end{lemma}
\noindent For a proof of this result we refer to \cite[Thm. 6.2]{mramor2012ghost}. Note that Lemma \ref{lem:comparisonprinciple} does not require periodicity.

The final property of the gradient flow that we make use of in this paper, will only be formulated for periodic sequences (even though it holds in some more generality). We first make two definitions: 
\begin{definition}\label{def:transverse}
Two sequences $x,y \in \mathbb{R}^{\mathbb{Z}}$ are said to be \emph{transverse}, denoted as $x \pitchfork y$, if they have no tangencies. That is,  $x_i = y_i$ for some $i \in \mathbb{Z}$ implies  $(x_{i-1} - y_{i-1}) (x_{i+1}-y_{i+1}) < 0$. 
\end{definition} 
\noindent We note that the sequences $x$ and $y$ are trivially transverse if $x_i\neq y_i$ for all $i\in \mathbb{Z}$. 
Examples of  transverse sequences and  non-transverse sequences are shown in Figure \ref{fig:transv_seq}.
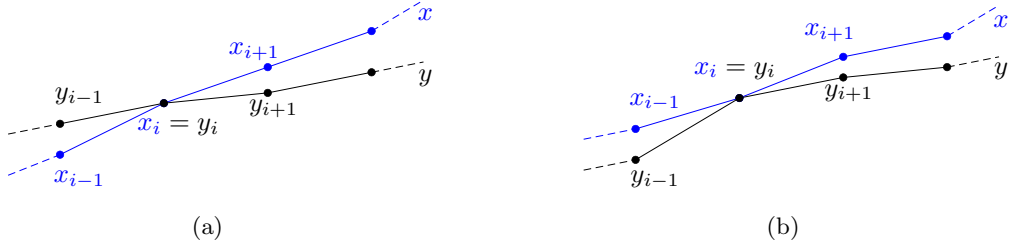
\begin{figure}[h!]
\centering
\begin{subfigure}{.35\textwidth}
  \centering
  \begin{tikzpicture}
  \node at (0,0) {\includegraphics[width=\linewidth]{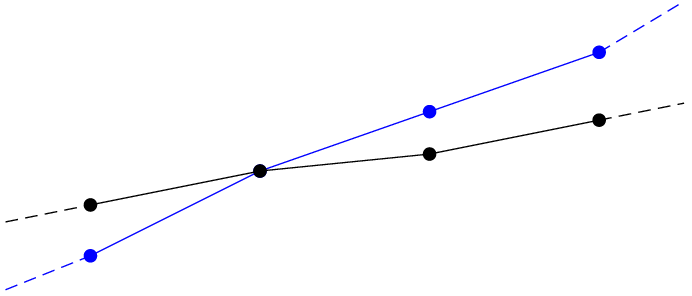}};
\node at (-1.8,-0.1) {$y_{i-1}$};
\node at (-1.8,-1.2) {$\textcolor{blue}{x_{i-1}}$};  
\node at (-0.5,-0.5) {$\textcolor{blue}{x_i}=y_i$};
\node at (0.5,0.5) {$\textcolor{blue}{x_{i+1}}$};
\node at (0.75,-0.25) {$y_{i+1}$};
\node at (2.75,1) {$\textcolor{blue}{x}$};
\node at (2.75,0.15) {$y$};
  \end{tikzpicture}
  \caption{}
  \label{fig:sub3_new}
\end{subfigure}\hspace{2cm}
\begin{subfigure}{.35\textwidth}
  \centering
    \begin{tikzpicture}
 \node at (0,0) {\includegraphics[width=\linewidth]{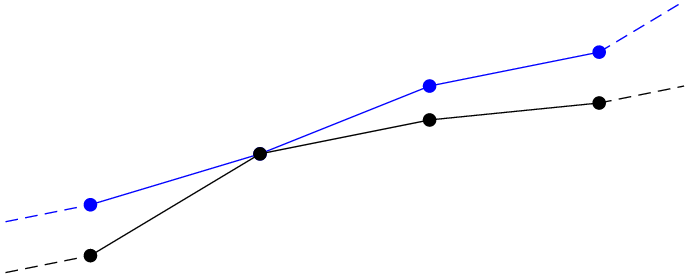}};
\node at (-1.8,-0.17) {$\textcolor{blue}{x_{i-1}}$};
\node at (-1.8,-1.2) {$y_{i-1}$}; 
\node at (-0.75,0.25) {$\textcolor{blue}{x_i}=y_i$};
\node at (0.5,0.75) {$\textcolor{blue}{x_{i+1}}$};
\node at (0.75,-0.05) {$y_{i+1}$};
\node at (2.75,0.9) {$\textcolor{blue}{x}$};
\node at (2.75,0.25) {$y$};
  \end{tikzpicture}  
  \caption{}
  \label{fig:sub4_new}
\end{subfigure}\\
  \caption{\textbf{(a)} Transverse sequences with Aubry diagrams intersecting at $(i, x_i)=(i,y_i)$; \textbf{(b)} Non-transverse sequences with a tangency at $(i, x_i)=(i,y_i)$.} \label{fig:transv_seq}
\end{figure}
\begin{definition}\label{def:intersection}
Let $x,y\in \mathbb{X}_{p,q}$. If $x \pitchfork y$, then the \emph{intersection index} $I(x,y)$ of $x$ and $y$ is defined to be largest integer $k$ for which there are $$i_0 < i_1< \ldots < i_k = i_0+p\, ,$$ such that $$(x_{i_j}-y_{i_j})(x_{i_{j+1}}-y_{i_{j+1}})<0 \ \ \mbox{holds for}\ \ j=0,1,2,\ldots,k-1\, .$$
\end{definition}
\noindent We observe that $I(x,y)$ is even because $x-y$ is periodic. 
The so-called \emph{Sturmian lemma}  states that the intersection index does not increase along gradient flow lines:
\begin{lemma} \label{prop:decreaseintersection}
Let $\delta >0$ and let $x(t), y(t) \in \mathbb{X}_{p,q}\cap \Sigma_{\delta}$ be two gradient flow lines, defined for $0\leq t < t_0$.  Assume that $x(0) \neq y(0)$. The set of $0\leq t < t_0$ for which $x(t)$ and $y(t)$ are not transverse is finite. The intersection index $I(x(t), y(t))$ (which is well-defined on the complement of this set) is a non-increasing function of $t$, which strictly decreases exactly at those $t$ at which $x(t)$
and $y(t)$ are not transverse. 
\end{lemma}
\noindent For a proof of this result we refer to \cite[Lemma 22.1]{gole2001symplectic}. 

\section{Symmetric periodic Birkhoff sequences}\label{sec:symBirkhoffsection}
\noindent In this section, we investigate $\mathbb{D}_n$-symmetric periodic Birkhoff orbits, as these play a major role in our analysis of non-Birkhoff periodic orbits. By $\mathbb{Z}_n = \langle R\rangle \subset \mathbb{D}_n$ we shall denote the subgroup consisting of all the rotations in $\mathbb{D}_n$.  Our first result characterizes $\mathbb{Z}_n$-symmetric periodic billiard sequences with period $p=n$. We recall that  $\kappa(z)$ denotes the curvature of $\Gamma$ at a point $z\in \Gamma$.
\begin{lemma}\label{lem:propertiesBirkhoffDn}
Let $m,n\in \mathbb{N}$ be co-prime with $1\leq m \leq n-1$, and let $\Gamma$ be a $\mathbb{D}_n$-symmetric billiard. Let $Z\in \Gamma^{\mathbb{Z}}$ be an $n$-periodic and $\mathbb{Z}_n$-symmetric billiard sequence with rotation number $\frac{m}{n}$ and lift $X\in \Sigma$. Then $Z$ is Birkhoff, $Z_i = R^{mi}(Z_0)$ and $X_i=X_0+\frac{m}{n}i$. The orbit segment length $L_{i,i+1}:=\|Z_{i+1}-Z_i\|$, and the billiard curvature $\kappa_i := \kappa(Z_i)$ are constant along $Z$. If $Z$ in addition satisfies the billiard law, then also the reflection angle $\theta_i:= \frac{1}{2} \angle(Z_{i+1}-Z_{i}, Z_{i}-Z_{i-1} ) = \frac{m\pi}{n}$ is constant along $Z$.
\end{lemma}

\begin{proof} Let $Z\in \Gamma^{\mathbb{Z}}$ be an $n$-periodic and $\mathbb{Z}_n$-symmetric sequence with rotation number $\frac{m}{n}$. In case $n=2$, we have $m=1$ and hence the rotation number of $Z$ is $\frac{1}{2}$. Since $p=n=2$, Lemma \ref{lem:periodicp=2} implies that $R(Z_i)=Z_{i+1}$. Any lift $X\in \Sigma$ of $Z$  therefore satisfies $X_{i+1} = X_i+\frac{1}{2}$, so that $X_i=X_0+\frac{1}{2}i$. 

For $n=p>2$, Lemma \ref{lem:periodicn=2} implies that $R^M(z_i)=z_{\frac{p}{n}+i} =z_{1+i}$ for some $M$ with $\gcd(M,n)=1$. A lift $X\in \Sigma$ of $Z$ must then satisfy $X_{i+1}=X_i +\frac{M}{n}$, so $X_{i}=X_0+\frac{M}{n}i$. But this implies that $M=m$, because $X$ has rotation number $\frac{m}{n}$. We conclude that  $X_i = X_0 + \frac{m}{n}i$.  It follows that $X\in \mathbb{X}_{n,m}$ and that $X$ is Birkhoff:  $X_i$ is an affine function of $i$, so $X$ does not intersect any of its nontrivial integer translates, cf. Remark \ref{rem:birkhoff}. We also find that  $Z_{i+1}= \gamma(X_{i+1})=\gamma(X_i + \frac{m}{n}) = R^m(\gamma(X_i))= R^m(Z_i)$, so that $Z_i=R^{mi}(Z_0)$.

Because the billiard $\Gamma$ is $\mathbb{D}_n$-invariant, it follows that $\kappa(R(z)) = \kappa(z)$ for any $z\in \Gamma$. In particular, $\kappa_i=\kappa(Z_{i})=\kappa(R^{im}(Z_0)) = \kappa(Z_0)$ is constant along the sequence $Z_i$.

It is also clear that $L_{i, i+1}=\|Z_{i+1}-Z_i\| = \|R^{(i+1)m}(Z_{0})-R^{im}(Z_0) \|= \| R^{m}(Z_{0})-Z_0\|$ is independent of $i$, i.e., the  segment length is constant along the sequence. 

Finally, if $Z$ satisfies the billiard law ``angle of incidence = angle of reflection'', then the reflection angle satisfies $\theta_i = \frac{1}{2} \angle(Z_{i+1}-Z_{i}, Z_{i}-Z_{i-1} ) = \frac{1}{2} \angle(R^m(Z_{i}-Z_{i-1}), Z_{i}-Z_{i-1} )$. Because $R^m$ is a rotation over $\frac{2 m\pi}{n}$, it follows that all reflection angles are equal to $\frac{m\pi}{n}$.
\end{proof}
\noindent The lifts $X$ of the Birkhoff sequences $Z$ described in Lemma \ref{lem:propertiesBirkhoffDn} together form the set 
$$G_{n,m}:=\left\{X\in \mathbb{R}^{Z}\, |\, X_i=X_0 + \frac{m}{n}i\ \mbox{for some} \ X_0\in \mathbb{R}\right\} \subset \mathbb{R}^{\mathbb{Z}}\, $$  
of all linear sequences of rotation number $\frac{m}{n}$. We summarize the properties of $G_{n,m}$ in the following proposition.
\begin{proposition}\label{prop:ghostcircle}
Let $m,n\in \mathbb{N}$ be co-prime with $1\leq m \leq n-1$. The set $G_{n,m}$ is nonempty, totally ordered, translation-invariant, closed under pointwise convergence, and invariant under the gradient flow.  
\end{proposition}
\begin{proof}  
The only slightly nontrivial statement is the invariance under the gradient flow. Note that $G_{n,m}$ is equal to the set of sequences satisfying $X_{i+1}=X_i+\frac{m}{n}$ for all $i\in \mathbb{Z}$. These  are equalities as in part {\it i)} of Lemma \ref{liftlemma} (choosing $k=1$, $a=m$ and $M=0$). By Lemma \ref{lem:symmetrypreservation} these equalities are preserved under the gradient flow, i.e., $G_{n,m}$ is invariant under the gradient flow.  
\end{proof}
\noindent A set with the properties described in the Proposition \ref{prop:ghostcircle} is called a {\it ghost circle} in  \cite{gole2001symplectic, mramor2012ghost} (note that $G_{n,m}$ is homeomorphic to $\mathbb{R}$, but the quotient $G_{n,m}/\mathbb{Z}$ obtained by identifying sequences $X$ and $Y$ for which $X_0-Y_0\in\mathbb{Z}$ is homeomorphic to the circle $\mathbb{R}/\mathbb{Z}$, which  explains part of the name).
 
The elements of $G_{n,m}$ are $\mathbb{Z}_n$-symmetric, but not necessarily $\mathbb{D}_n$-symmetric. Our next result states that $G_{n,m}$ contains exactly two geometrically distinct $\mathbb{D}_n$-symmetric sequences, both of which are billiard orbits.  

\begin{lemma}\label{lem:existenceBirkhoffDn}
Let $m,n\in \mathbb{N}$ be co-prime with $1\leq m \leq n-1$, and let $\Gamma$ be a $\mathbb{D}_n$-symmetric  billiard. Then $\Gamma$ admits exactly two geometrically distinct \changes{$n$-periodic} $\mathbb{D}_n$-symmetric Birkhoff orbits of rotation number $\frac{m}{n}$. 

More precisely, an $n$-periodic billiard sequence $Z\in \Gamma^{\mathbb{Z}}$ of rotation number $\frac{m}{n}$ is    $\mathbb{D}_n$-symmetric  if and only if it has a lift of the form
\begin{equation}
\label{eqn:Xaformula}
X_i = \frac{A}{2n} + \frac{m}{n} i \ \mbox{for some} \ A\in \mathbb{Z} \, .
\end{equation}
Every such sequence is the lift of a billiard orbit. Two such orbits are geometrically equal if and only if their lifts differ by an integer multiple of $\frac{1}{n}$.
\end{lemma}
\begin{proof}
Let $Z\in \Gamma^{\mathbb{Z}}$ be a $\mathbb{D}_n$-symmetric sequence of period $n$ and winding number $m$ with lift $X\in \Sigma$. Then $Z$ is in particular $\mathbb{Z}_n$-symmetric, and hence its lift is of the form $X_i=X_0+\frac{m}{n}i$ by Lemma  \ref{lem:propertiesBirkhoffDn}. However, $Z$ is also invariant under $S$, and we claim that $S(Z_i)=Z_{k-i}$ for some $k\in \mathbb{Z}$. Indeed, for $n\geq 3$ this follows directly from Lemma \ref{lem:periodicn=2}. For $n=2$, either $S(Z_i)=Z_i$ or $S(Z_i)\neq Z_i$ for all $i\in \mathbb{Z}$. Because $Z$ has period $2$, in the first case, $S(Z_i)= Z_{-i}$, while in the second case $S(Z_i)= Z_{1-i}$ for all $i\in \mathbb{Z}$. This proves our claim. 
  
Case {\it iv)} of Lemma \ref{liftlemma} states that $S(Z_i)=Z_{k-i}$ if and only if there is an integer $M$ such that $M = X_i + X_{k-i} =  X_0+\frac{m}{n}i +X_0+\frac{m}{n}(k-i) = 2X_0 + \frac{mk}{n}$ for all $i\in \mathbb{Z}$, that is, if and only if 
\begin{equation}\label{eq:X0formula}
X_0 = \frac{1}{2n}(Mn-km) \ \mbox{for some}\ k, M\in \mathbb{Z}\, .
\end{equation}
Because $\gcd(m,n)=1$, we conclude that the  $\mathbb{D}_n$-symmetric billiard sequences of period $n$ and winding number $m$ are precisely the billiard sequences with a lift of the form $X_i=\frac{A}{2n} + \frac{m}{n}i$ for some $A\in \mathbb{Z}$. 

In particular, the billiard sequences of this form constitute a  discrete subset of $\Sigma_{\frac{1}{n}}$.
By Lemma \ref{lem:symmetrypreservation}, the gradient vector field $F$ is tangent to this discrete set,  hence it vanishes. This proves that any $n$-periodic $\mathbb{D}_n$-symmetric billiard sequence is a billiard orbit. 

To see when two such orbits are geometrically distinct, let $X$ and $Y$ be two sequences with $X_i=\frac{A}{2n} + \frac{m}{n}i$ and $Y_i=\frac{B}{2n} + \frac{m}{n}i$ for some $A,B\in \mathbb{Z}$. In case $Y_i-X_i=\frac{C}{n}$ is an integer multiple of $\frac{1}{n}$, then we can choose $s, t\in \mathbb{Z}$ such that $sm+tn=1$ and define $K:=sC$. Then  $X_{K+i} = X_i + \frac{m}{n}K = X_i + \frac{m}{n}sC = X_i + \frac{C}{n}(1-tn) = X_i + \frac{C}{n} - tC =  Y_i - tC$. This shows that $Y_i-X_{K+i}\in \mathbb{Z}$ so  $i\mapsto \gamma(X_{K+i})$ and $i\mapsto \gamma(Y_i)$ are the same billiard orbit. Conversely, when $X_i=X_0+\frac{m}{n}i$ and $Y_i=Y_0+\frac{m}{n}j$ are lifts of two geometrically equal billiard orbits, then $X_i=Y_{K+i}+l$ for some $K,l\in \mathbb{Z}$. It follows that $X_i = Y_i+\frac{m}{n}K + l$ is an integer multiple of $\frac{1}{n}$. Thus, two $n$-periodic and $\mathbb{D}_n$-symmetric billiard sequences are geometrically equal if and only if their lifts differ by an integer multiple of $\frac{1}{n}$.
\end{proof}
\noindent Figures \ref{fig:sub1} and \ref{fig:sub3} both depict two (geometrically distinct) examples of $\mathbb{D}_n$-symmetric $n$-periodic Birkhoff orbits for $n=4$ and $n=5$, respectively.

\begin{remark}
As $G_{n,m}\subset \Sigma_{\frac{1}{n}}$, the restriction of $W_{n,m}$ to $G_{n,m}$ is smooth, and takes the simple form  
\begin{equation}\label{eq:WnmonGnm}
W_{n,m}(X) = \sum_{j=1}^n L(X_j, X_{j+1}) = 
\sum_{j=1}^n L\left(X_{0}+\frac{m}{n}j, X_0+\frac{m}{n}(j+1)\right) =
n L\left(X_0,X_0+\frac{m}{n}\right) .
\end{equation}
Invariance under $R$ means that this function is $1/n$-periodic in $X_0$, while invariance under $S$ means that it is even in $X_0$ (invariant under $X_0\mapsto -X_0$). These properties imply that any $X_0 $ of the form $X_0 = \frac{A}{2n}$ with $A\in \mathbb{Z}$ must be a stationary point of $W$, which confirms the conclusion of Lemma \ref{lem:existenceBirkhoffDn}.
\end{remark}

\section{The Hessian of the periodic action}\label{sec:Hessiansection}
\noindent The main result of this section is Lemma \ref{lem:Hessianlemma}, which provides an expression for the  Hessian of the periodic action $W_{p,q}$ at a $\mathbb{D}_n$-symmetric Birkhoff orbit $X\in \mathbb{X}_{n,m}\subset \mathbb{X}_{p,q}$. This expression depends, among other things, on the curvature of the billiard. We recall that  the curvature of an embedded curve $\Gamma$ parametrized by a $C^2$-smooth immersion $\gamma:\mathbb{R}\to\mathbb{R}^2$,   is given by the formula \begin{equation}\label{curvature}
\kappa(z) = \frac{  \left\|\frac{d}{dx} \frac{\gamma'(x)}{\|\gamma'(x)\|} \right\| }{ \|\gamma'(x)\|} =  \frac{|\det(\gamma'(x), \gamma''(x))|}{\|\gamma'(x)\|^3}\ , \  \mbox{with}\ z=\gamma(x)\, .
\end{equation} 
The curvature is invariant under reparametrization of $\Gamma$. A billiard may always be parametrized at constant speed, i.e., one may choose a parametrization which satisfies $\|\gamma'(x)\| \equiv c := \int_0^1\|\gamma'(s)\|ds$. Under this assumption, we have $\langle \gamma'(x), \gamma''(x)\rangle = 0$, and the formula for the curvature simplifies to $\kappa(z)=\|\gamma''(x)\|/c^2$.  Moreover, our standing assumption that $\Gamma$ is strictly convex implies that $\kappa>0$ on a dense subset of $\Gamma$. 

The first result of this section provides formulas for the second derivatives of the discrete Lagrangian $L(x,X)$. These formulas are known, see for example \cite[Lemma 2.1]{rychlik1989periodic}, but we include their derivation for completeness. \changes{Note that it is assumed in Lemma \ref{PartialsLalternativesecondorder} that the billiard is parametrized at constant speed. In Remark \ref{constantspeedrmk} below, we show that this assumption is not restrictive. However, we do not make this assumption everywhere in this paper. In particular, it does not appear in the statements of our main theorems, such as Theorem \ref{thm:sample}. Also, the parametrization \eqref{eq:limac-like} of the Lima\c con-type billiard introduced in Example \ref{ex:limacon} does not have constant speed.}

\begin{lemma} \label{PartialsLalternativesecondorder}
Assume that $\Gamma$ is parametrized at constant speed, i.e., that $\|\gamma'(x)\|   \equiv c:= \int_0^1\|\gamma'(s)\|ds$. Let $z=\gamma(x)$,  $Z=\gamma(X)\in \Gamma$ with $x-X\notin \mathbb{Z}$. Define $\Theta:=\angle \left(\gamma'(x), Z-z \right)$ and $\Phi := \angle \left(Z-z, \gamma'(X)\right)$ as in Lemma \ref{PartialsLalternativefirstorder}. Then
\begin{align}\label{2ndpartial1}
&  \partial_{1,1}L(x,X) = c^2 \left( \frac{\sin^2\Theta}{L(x,X)}- \kappa(z) \sin\Theta \right),\\ \label{2ndpartial2}
& \partial_{1,2}L(x,X)= c^2 \left( \frac{\sin \Theta\sin \Phi}{L(x,X)} \right), 
\\ \label{2ndpartial3}
&  \partial_{2,2}L(x,X) = c^2 \left( \frac{\sin^2\Phi}{L(x,X)}- \kappa(Z) \sin\Phi \right) .
\end{align}
Here, $\kappa(z)$ and $\kappa(Z)$ denote the curvatures of $\Gamma$ at $z$ and $Z$ respectively, as given in \eqref{curvature}. 
\end{lemma}
\begin{proof}
Recall that Lemma \ref{PartialsLalternativefirstorder} provides formulas for $\partial_1L(x,X)$ and $\partial_2L(x,X)$, and implies among other things that   $\partial_1L(x,X)=-c\cos \Theta$ and $\partial_2L(x,X)=c\cos \Phi$.   
For the mixed derivative we therefore obtain 
\begin{equation*}
\begin{split}
\partial_{1,2}L(x,X) & = 
\frac{\partial}{\partial x}\left\langle \frac{\gamma(X)-\gamma(x)}{L(x,X)}, \gamma'(X) \right\rangle\\
&= - \frac{\langle\gamma'(x),\gamma'(X)\rangle}{L(x,X)} - \frac{1}{L(x,X)}\left\langle \frac{\gamma(X)-\gamma(x)}{L(x,X)}, \gamma'(X) \right\rangle \cdot \partial_1 L (x,X) \\
& = - \frac{c^2\cos(\Theta+\Phi)}{L(x,X)} - \frac{1}{L(x,X)}( c \cos \Phi) (-c \cos \Theta )\\
& = c^2 \left(\frac{-\cos (\Theta+\Phi) + \cos \Phi \cos \Theta }{L(x,X)} \right)
= c^2 \left(\frac{\sin \Theta \sin \Phi }{L(x,X)} \right)  .
\end{split}
\end{equation*}
For the computation of  $\partial_{1,1}L$ we use explicitly  that $\langle \gamma'(x), \gamma''(x)\rangle =0$ and  that the curvature is given by the simplified formula $\kappa(z)=\|\gamma''(x)\|/c^2$. We find
\begin{equation*}
\begin{split}
\partial_{1,1} L (x,X) & = \frac{\partial}{\partial x}\left\langle \frac{\gamma(x)-\gamma(X)}{L(x,X)}, \gamma'(x) \right\rangle \\
& = \frac{\|\gamma'(x)\|^2}{L(x,X)} +  \left\langle \frac{\gamma(x)-\gamma(X)}{L(x,X)}, \gamma''(x)\right\rangle  - \frac{1}{L(x,X)}\left\langle  \frac{\gamma(x)-\gamma(X)}{L(x,X)}, \gamma'(x)\right\rangle \cdot \partial_1L(x,X) \\
& = \frac{c^2}{L(x,X)}- \|\gamma''(x)\| \cos(\pi/2-\Theta)  - \frac{1}{L(x,X)} (- c \cos \Theta ) (- c \cos \Theta )    \\ 
& = \frac{c^2}{L(x,X)} \sin^2 \Theta - \|\gamma''(x)\| \sin \Theta  = c^2\left( \frac{\sin^2 \Theta}{L(x,X)} - \kappa(z) \sin \Theta \right) . 
\end{split}
\end{equation*}
The formula for $\partial_{2,2}L(x,X)$ follows from an almost identical computation.
\end{proof}
\begin{remark}\label{rem:secondmonotoneremark}
The strict convexity of $\Gamma$ implies that $0<\Theta,\Phi<\pi$. It therefore follows from \eqref{2ndpartial2} that $\partial_{1,2} L >0$, as we stated in Remark \ref{rem:firstmonotoneremark}.
\end{remark} 
\begin{remark}\label{constantspeedrmk}
\changes{The assumption that $\Gamma$ is  parametrized at constant speed is important in the derivation of formulas \eqref{2ndpartial1} and \eqref{2ndpartial3}. If this assumption does not hold, then we may define the coordinate transformation $\tilde x:\mathbb{R}\to\mathbb{R}$ by $\tilde x(x):=c^{-1}\int_0^x\|\gamma'(s)\|ds$. The map $\tilde x$ is $C^2$ and satisfies $\tilde x(x+1)=\tilde x(x)+1$.  Because  $\frac{d\tilde x(x)}{dx}=c^{-1}\|\gamma'(x)\| >0$, it is  invertible; we denote the inverse by $x=x(\tilde x)$. This map is also $C^2$, and $x(\tilde x+1)=x(\tilde x)+1$.
The reparametrization of $\Gamma$ defined by $\tilde \gamma(\tilde x):= \gamma(x(\tilde x))$ is thus $C^2$ and  one-periodic. Because $\gamma'(x)  = \tilde \gamma'(\tilde x)\frac{d\tilde x}{d x}$, it follows that $\| \tilde \gamma'(\tilde x) \| = c$, that is, $\tilde \gamma$ parametrizes $\Gamma$ at constant speed. Lemma \ref{PartialsLalternativesecondorder} then holds for the reparametrized action $\tilde L(\tilde x, \tilde X):=L(x(\tilde x), x(\tilde X))$.}

\changes{It is also important to note that the reparametrization $\tilde \gamma$ inherits the symmetries of the original parametrization $\gamma$. Indeed, when $\gamma$ is $\mathbb{D}_n$-equivariant, so that \eqref{equivarianceofgamma} holds, then we have that $\|\gamma'(x+\frac{1}{n})\| = \| \gamma'(x)\| = \|\gamma'(-x)\|$. It follows from this that 
\begin{align}\nonumber 
\tilde x\left(x+\frac{1}{n}\right) & = c^{-1}\int_0^{x+\frac{1}{n}}\|\gamma'(s)\|ds =  c^{-1}\int_0^{x}\|\gamma'(s)\|ds + c^{-1}\int_{x}^{x+\frac{1}{n}}\|\gamma'(s)\|ds = \tilde x(x) +\frac{1}{n} \;\; \hbox{and} \\ \nonumber 
\tilde x(-x) &  = c^{-1}\int_0^{-x}\|\gamma'(s)\|ds  =c^{-1}\int_0^{x} - \|\gamma'(-\sigma)\|d\sigma  = -c^{-1}\int_0^{x} \|\gamma'(\sigma)\|d\sigma  =  - \tilde x(x)\, .
\end{align}
This in turn implies that $x(\tilde x+\frac{1}{n}) = x(\tilde x)+\frac{1}{n}$ and $x(-\tilde x)=-x(\tilde x)$, and therefore that $\tilde \gamma(\tilde x+\frac{1}{n}) = \gamma(x(\tilde x+\frac{1}{n})) = \gamma(x(\tilde x)+\frac{1}{n}) = R(\gamma(x(\tilde x))) =R(\tilde \gamma(\tilde x))$ and $\tilde \gamma(-\tilde x)= \gamma(x(-\tilde x)) =\gamma(-x(\tilde x)) = S(\gamma(x(\tilde x))) =  S(\tilde \gamma(\tilde x))$. This proves that the reparametrization $\tilde \gamma$ satisfies \eqref{equivarianceofgamma} and is thus  $\mathbb{D}_n$-equivariant. }
\end{remark}

\noindent 
\changes{Any $x\in \mathbb{X}_{p,q}$ is uniquely determined by its coordinates $(x_1, \ldots, x_p)\in\mathbb{R}^p$. Expressed in these coordinates, the periodic action is given by 
\begin{align}\label{wpqcoord}
W_{p,q}(x) = \sum_{j=1}^{p-1} L(x_j, x_{j+1}) + L(x_p, x_1+q)\, .
\end{align}
As a corollary of Lemma \ref{PartialsLalternativesecondorder}, we obtain the main result of this section about the Hessian of this action at a symmetric Birkhoff orbit: }
\begin{lemma}\label{lem:Hessianlemma}
Let $m,n\in \mathbb{N}$ be co-prime with $1\leq m \leq n-1$, let $\Gamma$ be a $\mathbb{D}_n$-symmetric billiard, and let $Z_i = \gamma(X_i)\in \Gamma$ be a $\mathbb{D}_n$-invariant, Birkhoff, $n$-periodic billiard orbit of $\Gamma$ with rotation number $\frac{m}{n}$ (whose existence is guaranteed by Lemma \ref{lem:existenceBirkhoffDn}). Let $(p,q)$ be a positive integer multiple of $(n,m)$, so that $X\in \mathbb{X}_{n,m}\subset \mathbb{X}_{p,q}$. Then, the Hessian of $W_{p,q}: \mathbb{X}_{p,q}\to\mathbb{R}$ at $X$, \changes{given in coordinates by \eqref{wpqcoord},} is a symmetric tridiagonal circulant matrix, i.e., it is of the form
\begin{equation}\label{Hessianmaintext}
D^2W_{p,q}(X) = \begin{pmatrix}
2 \alpha & \beta & 0 & \cdot \cdot \cdot & \beta \\[0.3 em]
\beta & 2\alpha & \beta & \cdot \cdot \cdot & 0 \\[0.3 em]
0 & \beta & 2\alpha & \ddots & \vdots \\[0.3 em]
\vdots & \vdots &\ddots &\ddots &\beta \\ 
\beta & 0 &\cdots & \beta & 2\alpha\\
\end{pmatrix}\, 
\end{equation}
in which 
$$\alpha = \partial_{2,2}L(X_{i-1}, X_i) = \partial_{1,1}L(X_{i}, X_{i+1}) \ \mbox{ and }\ \beta =  \partial_{2,1}L(X_{i}, X_{i+1}) =  \partial_{1,2}L(X_{i-1}, X_i) $$ are independent of $i\in \mathbb{Z}$.
For $N\in \mathbb{Z}$, the vectors \changes{$v, w\in \mathbb{X}_{p,0} \cong \mathbb{R}^p$}, given by 
\begin{equation} \label{eq:vwformula}
v_i = \sin\left(\frac{2\pi N i}{p}\right) , \  w_i = \cos\left(\frac{2\pi N i}{p}\right) 
\end{equation}
are eigenvectors of $D^2W_{p,q}(X)$ with eigenvalue 
$\lambda = 2\alpha + 2\beta \cos \left( \frac{2\pi N}{p} \right)$. 

When $\gamma$ parametrizes $\Gamma$ at constant speed, then
\begin{equation}\label{eq:alphabetaformula}
\alpha =  c^2 \sin (m\pi/n) \left(\frac{\sin(m\pi/n)}{L}   - \kappa \right) \; \mbox{and }\;\beta = \frac{c^2 \sin^2(m\pi/n)}{L}  \, ,
\end{equation}
in which $c=\|\gamma'(x)\|=\int_0^1\|\gamma'(s)\|ds$ is the length of $\Gamma$, $L=L(X_{i},X_{i+1}) = \|Z_{i+1}-Z_i\|$ is the constant orbit segment length along $Z$, and $\kappa=\kappa(Z_i)$ is the constant billiard curvature along $Z$. 
\end{lemma}
\begin{proof}
\changes{As remarked above,}
the periodic action is given by $$W_{p,q}(x) = \sum_{j=1}^{p-1} L(x_j, x_{j+1}) + L(x_p, x_1+q)\, ,$$
\changes{in terms of the coordinates $(x_1, \ldots, x_p)\in\mathbb{R}^p$ for $x\in \mathbb{X}_{p,q}$.}
Its Hessian at a Birkhoff periodic orbit \changes{$X\in \mathbb{X}_{n,m}$ is therefore} given by 
\begin{equation}\nonumber
D^2W_{p,q}(X) = \begin{pmatrix}
\alpha_1 +\omega_1 & \beta_1 & 0 & \cdot \cdot \cdot & \gamma_1 \\[0.3 em]
\gamma_2 & \alpha_2 +\omega_2 & \beta_2 & \cdot \cdot \cdot & 0 \\[0.3 em]
0 & \gamma_3 & \alpha_3 + \omega_3 & \ddots & \vdots \\[0.3 em]
\vdots & \vdots &\ddots &\ddots &\beta_{p-1} \\ 
\beta_p & 0 &\cdots & \gamma_p & \alpha_p + \omega_p\\
\end{pmatrix}\, 
\end{equation}
with entries 
\begin{equation}\label{eq:abcd}
\alpha_i = \partial_{2,2}L(X_{i-1}, X_i) \ , \ \omega_i= \partial_{1,1}L(X_{i}, X_{i+1})\ , \ \beta_i  = \partial_{2,1}L(X_{i}, X_{i+1}) \ \mbox{and}\ \gamma_i =   \partial_{1,2}L(X_{i-1}, X_{i})\, .
\end{equation}
To verify this, note for example that because $L(x,X)=L(x-q,X-q)$ and $X\in \mathbb{X}_{n,m}\subset \mathbb{X}_{p,q}$, we have 
\begin{align*}\partial_{1,1}W_{p,q}(X) &= \partial_{2,2}L(X_p, X_1+q) + \partial_{1,1}L(X_1, X_2) = \partial_{2,2}L(X_p-q, X_1) + \partial_{1,1}L(X_1, X_2)= \\ &= \partial_{2,2}L(X_0, X_1) + \partial_{1,1}L(X_1, X_2) = \alpha_1+\omega_1.
\end{align*}
Similar calculations show that $\partial_{2,1}W_{p,q}(X) = \gamma_1$,  $\partial_{1,2}W_{p,q}(X) = \beta_p$, and   $\partial_{p,p}W_{p,q}(X)  = \alpha_p+\omega_p$. The remaining entries of $D^2W_{p,q}(X)$ are trivially given by \eqref{eq:abcd}. 

The rotational invariance $L(x,X) = L(x+m/n, X+m/n)$ implies that $\partial_{2,2}L(x, X) = \partial_{2,2}L(x+m/n, X+m/n)$. The $\mathbb{D}_n$-invariant Birkhoff orbit $X$ satisfies $(X_i, X_{i+1})=(X_{i-1}+m/n, X_i+m/n)$ so $\alpha_{i+1}=\partial_{2,2}L(X_{i}, X_{i+1})= \partial_{2,2}L(X_{i-1}, X_{i}) = \alpha_i$. This proves that all $\alpha_i$'s are equal. Similarly, all $\omega_i$'s are equal, all $\beta_i$'s are equal and all $\gamma_i$'s are equal. Since $\beta_1=\gamma_2$ by definition, we also have that all $\gamma_i$'s are equal to all $\beta_i$'s.  Finally, the reflection invariance $L(x,X)=L(-X,-x)$ implies that $\omega_i= \partial_{1,1} L(X_i, X_{i+1}) = \partial_{2,2}L(-X_{i+1}, -X_i) = \partial_{2,2}L(X_{i-1},X_{i})= \alpha_i$. Here, the third equality follows because $X$ is $S$-invariant, so that the conclusions of Lemma \ref{liftlemma} {\it iii)} and {\it iv)} holds.  This shows that all $\omega_i$'s are equal to all $\alpha_i$'s, and finishes the proof  that $D^2W_{p,q}(X)$ is of the form \eqref{Hessianmaintext}.

A small computation (exploiting the doubling formulas for sine and cosine) confirms the statement about eigenvectors and eigenvalues of $D^2W_{p,q}(X)$. In fact, it is clear that $v$ and $w$ given by formula \eqref{eq:vwformula} are eigenvectors of any circulant matrix.

When $\|\gamma'\|\equiv c$, then by Lemma \ref{PartialsLalternativesecondorder} we have 
\begin{align} \nonumber
\begin{array}{ll} \displaystyle
\alpha_i = c^2 \left( \frac{\sin^2\phi_{i}}{L(X_{i-1},X_{i})}- \kappa(Z_i) \sin\phi_i \right)     , 
&   \displaystyle
\omega_i = c^2 \left( \frac{\sin^2\theta_i}{L(X_i,X_{i+1})}- \kappa(Z_i) \sin\theta_i \right) , \\ \\   \displaystyle
\beta_i =  c^2 \left( \frac{\sin \theta_i\sin \phi_{i+1}}{L(X_i,X_{i+1})} \right)  ,  
& \displaystyle
\gamma_i =  c^2 \left( \frac{\sin \theta_{i-1}\sin \phi_{i}}{L(X_{i-1},X_{i})} \right) ,
\end{array}
\end{align}
in which $\phi_i$ is the angle of incidence at $Z_i$ and $\theta_i$ is the angle of reflection at $Z_i$. Since $X$ is a billiard trajectory, $\theta_i=\phi_i$ for all $i\in \mathbb{Z}$, and by Lemma \ref{lem:propertiesBirkhoffDn} this angle is independent of $i\in \mathbb{Z}$ and equal to $\frac{m\pi}{n}$. The same lemma states that the orbit segment length $L=L(X_i, X_{i+1})$  and the curvature $\kappa=\kappa(Z_i)$ are independent of $i\in \mathbb{Z}$. This proves formula \eqref{eq:alphabetaformula} and concludes the proof of the lemma. 
\end{proof}
\begin{remark}\label{remk:eigenvalue}
When $\Gamma$ is parametrized at constant speed,  so that \eqref{eq:alphabetaformula} holds, the eigenvalue of $D^2W_{p,q}(X)$ for the eigenvectors $v$ and $w$ defined in \eqref{eq:vwformula} is given by
$$\lambda = 2\alpha + 2\beta \cos \left( \frac{2\pi N}{p} \right) = 2c^2\sin \left( \frac{m\pi}{n} \right) 
\left( \frac{2\sin\left( \frac{m\pi}{n} \right)  \cos^2\left( \frac{N\pi}{p} \right) }{L}  - \kappa \right) . $$
This simple formula will be crucial in the proofs of our main theorems. 
\end{remark}
 
\section{Technical preliminaries for the proof of the main theorem}\label{sec:technicalsection}
\noindent In this section, we prove some further technical results that we use in the proof of our main theorem in the next section. The main result of this section is Corollary \ref{cor:gradientlimit}, which states that gradient flow lines with appropriately chosen initial conditions are  defined for all positive time, and cannot approach a singularity (i.e., a point with $x_{i+1}-x_{i}\in \mathbb{Z}$). This result relies on two propositions that we prove first. To formulate the first proposition, we introduce some notation by defining
\begin{align} 
F^-(x, X):= & \,\, \partial_2L(x,X) = \left\langle \frac{\gamma(X)-\gamma(x)}{\| \gamma(X)-\gamma(x) \| }, \gamma'(X) \right\rangle \ \mbox{for}\ x < X < x+1\, ,  
\label{eq:F-formula} \\ \label{eq:F+formula} 
F^+(x, X):= & \,\, \partial_1L(x,X) = \left\langle  \frac{\gamma(x)-\gamma(X)}{\| \gamma(x)-\gamma(X) \| }, \gamma'(x) \right\rangle \ \mbox{for}\ x < X < x+1\, .
\end{align}
With this notation, equation \eqref{eq:gradientflow} can be written as $$\dot x_i = F_i(x)= F^-(x_{i-1}, x_{i}) + F^+(x_i, x_{i+1}) \ \mbox{ for }\ x\in \Sigma\, .$$
One can think of $F^-(x_{i-1}, x_i)$ as a ``force'' exerted on $x_i$ by $x_{i-1}$, and of $F^+(x_i, x_{i+1})$ as a ``force'' exerted on $x_i$ by $x_{i+1}$. We collect some properties of $F^-$ and $F^+$ in the following proposition. 
\begin{proposition}\label{lem:F+-lemma}
The above $F^-$ and $F^+$ extend to continuous functions  
$$F^{-}, F^+: \{ (x,X)\in \mathbb{R}^2 \, |\,  x \leq X \leq x+1 \} \to \mathbb{R}$$ 
with the following properties\changes{, which hold for all $x\leq X\leq x+1$}:
\begin{itemize}
\item[{\it i})] $F^{-}(x,x) = F^+(x,x+1)= \|\gamma'(x)\|$ and $F^{-}(x,x+1) = F^{+}(x,x) =  -\|\gamma'(x)\|$;
\item[{\it ii)}] $x\mapsto F^{-}(x, X)$ and $X\mapsto F^+(x,X)$ are strictly increasing;
\item[{\it iii)}] $F^{\pm}(x+1/n, X+1/n) = F^{\pm}(x,X)$;
\item[{\it iv)}] $F^{-}\left(x, x+ \frac{m}{n}\right) =   \| \gamma'(x)\| \cos \left( \frac{m\pi}{n} \right)$ and $F^{+}\left(x, x+ \frac{m}{n}\right) =  - \| \gamma'(x)\| \cos \left( \frac{m\pi}{n} \right)$ for $1\leq m \leq n-1$. 
\end{itemize} 
\end{proposition}
\begin{proof}
Recall that $\gamma$ is assumed to be $C^2$ and $\gamma(x)=\gamma(x+1)$. It follows from formulas \eqref{eq:F-formula} and \eqref{eq:F+formula} that $F^{-}$ and $F^+$ are continuous at any point $(x,X)$ with $x<X<x+1$. Because $\|\gamma'(x)\| > 0$ for all $x\in \mathbb{R}$, the length-one vector $\frac{\gamma(X)-\gamma(x)}{\| \gamma(X)-\gamma(x)\changes{\|}} $ converges to $\frac{\gamma'(x)}{\|\gamma'(x)\|}$ as $X\downarrow x$, and to $-\frac{\gamma'(x)}{\|\gamma'(x)\|}$  as $X\uparrow x+1$. From this it follows that $F^{\pm}$ can be extended continuously to $x\leq X\leq x+1$ with boundary values as given in {\it i)}.
    
To prove {\it ii)}, note that the angle between   $\frac{\gamma(X)-\gamma(x)}{\| \gamma(X)-\gamma(x) \|}$ and $\gamma'(X)$ strictly decreases from $\pi$ to $0$ as $x$ increases from $X-1$ to $X$, so that $x\mapsto F^-(x,X)$ strictly increases from $\|\gamma'(X)\| \cos(\pi)= - \|\gamma'(X)\| $ to $\|\gamma'(X)\| \cos(0)=\|\gamma'(X)\|$. Similarly for the monotonicity of $F^+$.
     
To prove {\it iii)}, we simply note that $\gamma(x+1/n)=R\gamma(x)$,  $\gamma(X+1/n)=R\gamma(X)$ (and the same for $\gamma'$) and that, being an isometry, $R$ preserves the inner product in the definitions of $F^{\pm}$. 
    
To prove {\it iv)}, recall that $\gamma\left(x +\frac{m}{n}\right) = R^m(\gamma(x))$, so the angle between $\frac{\gamma(x+\frac{m}{n})-\gamma(x)}{\| \gamma(x+\frac{m}{n})-\gamma(x) \|}$ and $\gamma'(x+\frac{m}{n})$ is $\frac{m\pi}{n}$. As a result, $F^-\left(x,x + \frac{m}{n}\right) =  \| \gamma'\left(x+ \frac{m}{n}\right)\| \cos \left(\frac{m\pi}{n} \right) =  \| \gamma'(x)\| \cos \left(\frac{m\pi}{n} \right)$. Similarly for $F^+\left(x,x + \frac{m}{n}\right)$.
\end{proof}
\noindent In the next proposition, $X^-, X^+ \in \mathbb{X}_{n,m}$ will denote lifts of two  $\mathbb{D}_n$-invariant Birkhoff sequences of rotation number $\frac{m}{n}$. Thus, they are both of the form 
\eqref{eqn:Xaformula}. We moreover assume that 
$$X_{i}^+ = X_i^-+\frac{1}{n}\, .$$ 
In particular, $X^-$ and $X^+$ define billiard orbits that are geometrically equal -- see Definition \ref{rem:geom-dist}
-- while there is no $X$ between $X^-$ and $X^+$ doing so too. We define the open and closed order intervals
$$(X^-,X^+):=\{ x\in\mathbb{R}^{\mathbb{Z}}\, |\, X_i^- < x_i < X_i^+\} \ \mbox{and}
 \ [X^-,X^+]:=\{ x\in\mathbb{R}^{\mathbb{Z}}\, |\, X_i^- \leq x_i \leq X_i^+\}\, .$$
Clearly, $(X^-, X^+) \subset [X^-, X^+]$. The following  result describes the relation between these order intervals and the set $\Sigma$, recall \eqref{eq:billiardliftsequence}. 
\begin{proposition} \label{prop:Y-Y+}
Let $X^-, X^+$ be as above, i.e., \eqref{eqn:Xaformula} holds and $X^+_i=X_i^-+\frac{1}{n}$. Then $(X^-, X^+) \subset \Sigma$. When $2\leq m \leq n-2$, then also $[X^-, X^+] \subset \Sigma_{\frac{1}{n}} \subset \Sigma$.
\end{proposition}
\begin{proof}
Let $x\in (X^-, X^+)$, so that $X^-_i < x_i < X^+_i = X_i^-+\frac{1}{n}$ for all $i\in \mathbb{Z}$. In particular also $X^-_{i+1} < x_{i+1} < X^+_{i+1} = X_{i+1}^-+\frac{1}{n}$. Subtracting these inequalities, and using that $X^-_{i+1}-X^-_i=\frac{m}{n}$ \changes{by Lemma \ref{lem:existenceBirkhoffDn}}, gives
\begin{equation}
\label{eq:niceestimate}
0\leq \frac{m-1}{n} < x_{i+1}-x_i < \frac{m+1}{n} \leq  1\, .
\end{equation}
This proves that $0< x_{i+1}-x_i< 1$, i.e., that $x\in \Sigma$.  When $m\leq 2 \leq n-2$, then \eqref{eq:niceestimate} implies that $\frac{1}{n} \leq x_{i+1} - x_i \leq 1-\frac{1}{n}$, i.e., $x\in \Sigma_{\frac{1}{n}}$.
\end{proof}
\noindent Proposition \ref{prop:Y-Y+} implies that the right-hand side of the gradient flow equation \eqref{eq:gradientflow} is defined on $(X^-, X^+)$. Moreover, by Lemma \ref{lem:existenceBirkhoffDn}, both $X^-$ and $X^+$ are stationary points of the gradient flow. By the comparison principle (see Lemma \ref{lem:comparisonprinciple}), the  order intervals are therefore positively invariant under the gradient flow: if $x(t)\in \Sigma_{\delta}$ (recall \eqref{def:sigmadelta}) is a gradient flow line with $x(0)\in (X^-,X^+)$ then $x(t)\in (X^-, X^+)$ for all $t>0$ for which the gradient flow line is defined (and the same statement holds for the closed order interval). Corollary \ref{cor:gradientlimit} below states among other things that the gradient flow is defined for all positive time on (a subset of) $(X^-, X^+)$. For $2\leq m \leq n-2$ this is a rather obvious consequence of Proposition \ref{prop:Y-Y+}. For  $m=1, n-1$, the result is intuitively clear, but the proof is quite technical and relies on Proposition \ref{lem:F+-lemma}. We refer to Figure \ref{fig:9point2} for a visualization of this latter situation.
 
\begin{corollary}\label{cor:gradientlimit}
Let $X^-, X^+$ be as above. There  is a $0< \delta < \frac{1}{n}$ so that $\Sigma_{\delta}\cap (X^-, X^+) \neq \emptyset$, and so that for every $\hat x\in \Sigma_{\delta}\cap (X^-, X^+)$, there is a  unique gradient flow line $x(t)\in \Sigma_{\delta}\cap (X^-, X^+)$ with $x(0)=\hat x$,  which is defined for all $t\geq 0$.
    
If $\hat x\in \mathbb{X}_{p,q}$, then this gradient flow line converges to a stationary solution $x_{\infty} = \lim_{t\to\infty} x(t)\in [X^-, X^+] \cap \Sigma_{\delta}\cap \mathbb{X}_{p,q}$ of equation \eqref{eq:gradientflow}. If $x_{\infty}\neq X^{\pm}$, then $x_{\infty}\in (X^-, X^+)$. 
\end{corollary}
\begin{proof}
When $2\leq m \leq n-2$, we choose $\delta = \frac{1}{2n}$. By Proposition \ref{prop:Y-Y+} we then have $[X^-, X^+]\subset \Sigma_{\delta}$. Because $F$ is $C^1$ on $\Sigma_{\delta}$, and hence globally Lipschitz on $\Sigma_{\delta}$, it trivially follows that $F$ is globally Lipschitz on $[X^-, X^+]$. The comparison principle (Lemma \ref{lem:comparisonprinciple}) implies that $[X^-, X^+]$ is a ``trapping region'' for the gradient flow: if  $\hat x \in [X^-, X^+]$, then $x(t) \in [X^-, X^+]$ for all $t\geq 0$.  Combining these observations, we see that the gradient flow line $x(t)\in [X^-, X^+]$ is unique and exists for all $t\geq 0$. Moreover, if $\hat x \in \mathbb{X}_{p,q}$ then $x(t)\in \mathbb{X}_{p,q}$ for all $t>0$, by Proposition \ref{prop:Xpqpreservation}. 

By Proposition \ref{prop:Lyapunovfunction}, the  function $t\mapsto W_{p,q}(x(t))$ is non-decreasing and bounded. In particular, $$\lim_{n\to\infty} \left( \left. \frac{d}{dt}  W_{p,q}(x(t))\right|_{t=t_n}  \right) = \lim_{n\to\infty} \left( \sum_{j=1}^p | F_j(x(t_n)) |^2 \right) = 0\, , $$ for some sequence of times $t_n$ converging to infinity. By compactness of $[X^-, X^+]$, we may assume that the limit $x_{\infty} = \lim_{n\to\infty} x(t_n)$ exists (by passing to a subsequence if necessary). We conclude that $F(x_{\infty})=0$ by continuity, that is, $x_{\infty}$ is a stationary point of the gradient flow. It is clear that $x_{\infty}\in\mathbb{X}_{p,q} \cap [X^-, X^+]$ as $\mathbb{X}_{p,q}$ and $[X^-, X^+]$ are closed. By the strong comparison principle, either $x_{\infty}=X^{\pm}$ or $x_{\infty}\in (X^-, X^+)$. This proves the proposition when $m\neq 1, n-1$.

The proof is more delicate when $m=1$ or $m=n-1$. We claim that also in this case there exists a $0<\delta<\frac{1}{n}$ such that the intersection $\Sigma_{\delta}\cap [X^-, X^+]$ is a trapping region for the gradient flow. The statements of the corollary then follow by exactly the same arguments that were used above for 
$2\leq m \leq n-2$.
For $m=1$ or $m=n-1$, our proof of the existence of a trapping region relies on Proposition \ref{lem:F+-lemma}.  
We start by choosing a small number $\varepsilon$ such that
\begin{align}\label{eq:epsestimate}
0< \varepsilon < \frac{1- \cos(\pi/n) }{2}   \,  \sup_{x\in \mathbb{R}} \| \gamma'(x)\| \, .
\end{align}
By the uniform continuity of $F^{-}$ and $F^+$, there exists a $0<\delta<\frac{1}{n}$ so that if $\max\{ |x-y|, |X-Y|\} \leq \delta$, then $| F^{\pm}(x,X)-F^{\pm}(y,Y)|<\varepsilon$.  We choose such a $\delta>0$. Note that $[X^-, X^+]\cap \Sigma_{\delta}\neq \emptyset$ because $\delta<\frac{1}{n}$. 
     
Next, let us consider a billiard sequence $x\in [X^-, X^+]$ for which there is an $i\in \mathbb{Z}$ such that 
$$x_{i+1}-x_i = \delta\, .$$ 
Note that, by \eqref{eq:niceestimate}, this can only occur when $m=1$. We claim that this assumption implies that 
\begin{align}
\label{eq:inequality1}
x_{i+2} \geq x_{i+1} + \frac{1}{n} - \delta \ \mbox{and}\ x_{i-1} \leq x_i - \frac{1}{n} + \delta \, .
\end{align}
{\it Proof of this claim:} 
Recall that $X_0^- + \frac{i}{n}  \leq x_i \leq X_0^- + \frac{i+1}{n} $, and hence also $X_0^- + \frac{i+2}{n}  \leq x_{i+2} \leq X_0^- + \frac{i+3}{n}$. Combining these inequalities,  we obtain $x_{i+2} \geq x_i + \frac{1}{n} = x_{i+1} - \delta + \frac{1}{n}$, the first inequality in \eqref{eq:inequality1}. The second inequality is proved in a similar fashion, using that $x_{i-1}\leq x_{i+1}-\frac{1}{n}$. We refer to Figure \ref{fig:9point2b} for an illustration of these estimates. 
\hfill $\square$\\ 
\noindent From the equality $x_{i+1}-x_{i}=\delta$ it follows that 
$$F^-(x_{i}, x_{i+1}) = F^-(x_{i+1}-\delta, x_{i+1}) \geq F^-(x_{i+1}, x_{i+1}) - \varepsilon = \|\gamma'(x_{i+1})\| - \varepsilon \, . $$
Here, the inequality follows from the uniform continuity of $F^-$, and the final equality from part {\it i)} of Proposition \ref{lem:F+-lemma}. On the other hand, the inequality  $x_{i+2} \geq x_{i+1}+\frac{1}{n}-\delta$ implies that 
$$F^{+}(x_{i+1}, x_{i+2})  \geq F^{+}(x_{i+1}, x_{i+1} + \frac{1}{n} - \delta)  \geq F^{+}(x_{i+1}, x_{i+1} + \frac{1}{n})  - \varepsilon = - \| \gamma'(x_{i})\| \cos(\pi/n)- \varepsilon  \, .$$
Here, in the first inequality, we used that $F^+$ is increasing in its second argument, see part {\it ii)} of Proposition \ref{lem:F+-lemma}. The second inequality follows from the uniform continuity of $F^+$, and the last equality from part {\it iv)} of Proposition \ref{lem:F+-lemma}. Using \eqref{eq:epsestimate}, we conclude that 
$$F_{i+1}(x) = F^-(x_{i}, x_{i+1}) + F^+(x_{i+1}, x_{i+2}) \geq  \|\gamma'(x_{i+1})\| - \varepsilon - \|\gamma'(x_{i+1})\| \cos(\pi/n) - \varepsilon > 0\, . $$
One could say that the ``force'' of $x_{i}$ on $x_{i+1}$ dominates the force of $x_{i+2}$ on $x_{i+1}$. Exploiting the second estimate in \eqref{eq:inequality1} and the fact that $F^-$ is increasing in its first argument, we can similarly show that $F_i(x) = F^-(x_{i-1},x_i)+F^+(x_i,x_{i+1}) < 0$ if $x_{i+1}-x_i=\delta$. Analogously, one can prove that  $F_{i+1}(x)<0$ and $F_i(x)>0$ if $x_{i+1}-x_i=1-\delta$ (which can only occur if $m=n-1$, by \eqref{eq:niceestimate}.)  

To summarize, if $0<\delta<\frac{1}{n}$ is chosen as above, then for any $x\in[X^-, X^+]$  we have the implications
\begin{align}\label{eq:Fimplications}
\begin{array}{rl}  \ x_{i+1}-x_i=\delta & \implies F_{i+1}(x)-F_i(x) > 0 \, ,\\ 
 x_{i+1}-x_i=1-\delta & \implies F_{i+1}(x)-F_i(x) < 0\, .
\end{array}
\end{align}
\noindent To conclude the proof, let $x(t)\in \Sigma_{\delta/2}\cap [X^-, X^+]$ be a gradient flow line defined for $0\leq t < t_0$ such that $\hat x:= x(0)\in \Sigma_{\delta}$. Assume, for the sake of contradiction, that there is a  $0<t_1<t_0$ for which $x(t_1)\notin \Sigma_{\delta}$. Then, by continuity, there must  exist a $0\leq t_2 \leq t_1$ and an $i\in \mathbb{Z}$ for which either $x_{i+1}(t_2)-x_i(t_2)=\delta$ and $\left. \frac{d}{dt}\right|_{t=t_2} x_{i+1}(t)-x_i(t) \leq 0$,  or $x_{i+1}(t_2)-x_i(t_2)=1-\delta$ and $\left. \frac{d}{dt}\right|_{t=t_2} x_{i+1}(t)-x_i(t) \geq 0$.  However, $\left. \frac{d}{dt}\right|_{t=t_2} x_{i+1}(t)-x_i(t) = F_{i+1}(x(t_2))-F_i(x(t_2))$ because  $x(t)$ is a gradient flow line. Thus, the assumption that $x(t_1)\notin \Sigma_{\delta}$ contradicts \eqref{eq:Fimplications}.  
This proves that $[X^-, X^+]\cap \Sigma_{\delta} \cap \mathbb{X}_{p,p}$ is a trapping region and finishes the proof of the corollary.
\end{proof}
\begin{figure}[h!]
\centering
\begin{subfigure}{.35\textwidth}
  \centering
  \begin{tikzpicture}
  \node at (0,0) {\includegraphics[width=\linewidth]{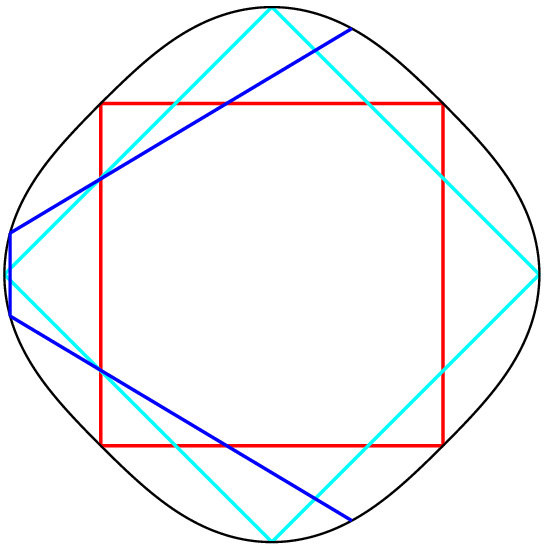}};
\node at (1.2,2.7) {$z_{i-1}$};
\node at (-2.9,0.6) {$z_i$};
\node at (-3,-0.6) {$z_{i+1}$};
\node at (1.2,-2.7) {$z_{i+2}$};
  \end{tikzpicture}
  \caption{}
  \label{fig:9point2a}
\end{subfigure}\hspace{2cm}
\begin{subfigure}{.35\textwidth}
  \centering
    \begin{tikzpicture}
 \node at (0,0) {\includegraphics[width=\linewidth]{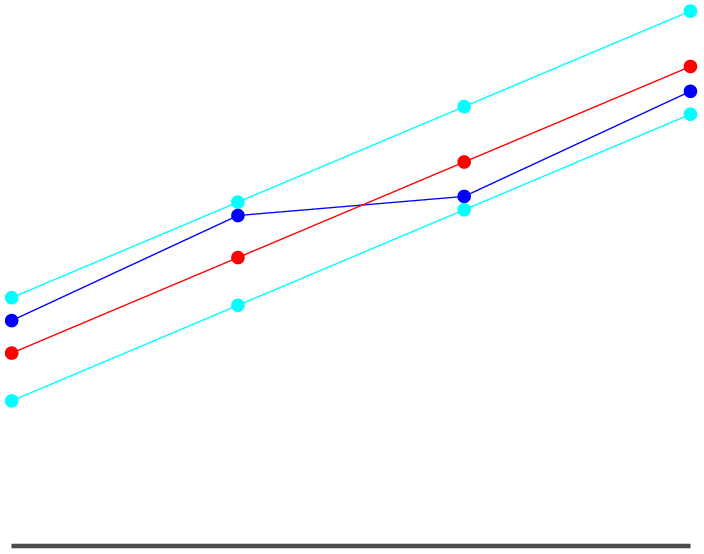}};
 \node at (-2.7,-2.3) {$i-1$};
 \node at (-0.9,-2.3) {$i$};
 \node at (0.8,-2.3) {$i+1$};
 \node at (2.6,-2.3) {$i+2$};
  \end{tikzpicture}  
  \caption{}
  \label{fig:9point2b}
\end{subfigure}
\caption{Visualization of the proof of Corollary \ref{cor:gradientlimit} for the case $(m,n)=(1,4)$ in a $\mathbb{D}_4$-symmetric billiard. {\bf (a)} ``Long'' and ``short'' symmetric Birkhoff orbits (in cyan and red, respectively) and part of a billiard sequence $z$ (in blue) for which $\|z_{i+1}-z_i\|$ is small. The ordering of $z$ with respect to the cyan Birkhoff orbit implies that $\|z_{i}-z_{i-1}\|$ and $\|z_{i+2}-z_{i+1}\|$ are not small; {\bf (b)} the corresponding Aubry diagrams illustrate the  estimates in \eqref{eq:inequality1}: the segment from $(i, x_i)$ to $(i+1, x_{i+1})$ is nearly horizontal. The ordering of the blue Aubry diagram with respect to the cyan Aubry diagrams implies that the neighboring blue segments have a slope close to or larger than $\frac{1}{4}$. \label{fig:9point2}}
\end{figure}
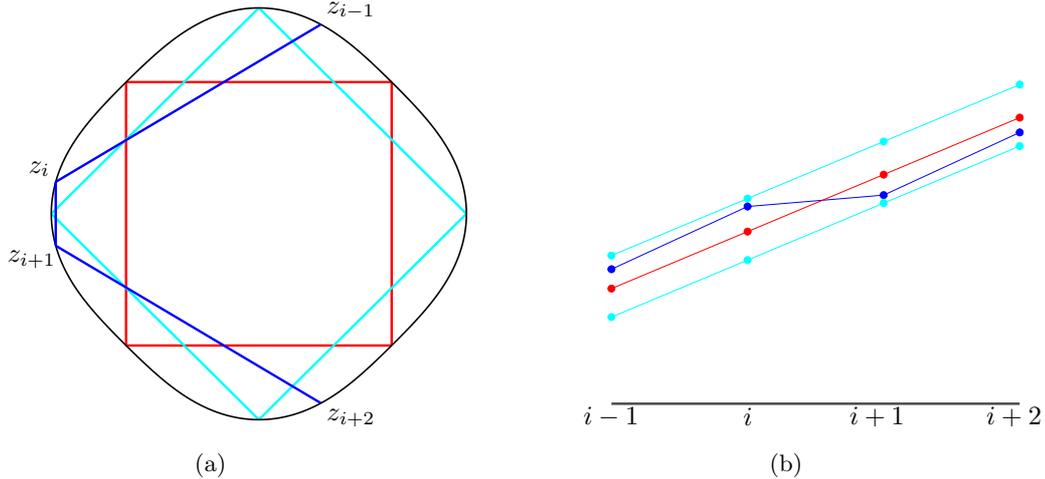
\noindent We conclude this section with one final technical result that is important in the next section. In the proposition below, the sequence $X$ plays the role of a $\mathbb{D}_n$-symmetric Birkhoff sequence that is geometrically distinct from $X^{\pm}$, while $x$ is a candidate non-Birkhoff orbit.  
\begin{proposition}\label{prop:minperiod}
Let $X^-, X^+$ be as above, and $x, X\in (X^-, X^+)$. Assume that $X$ is the lift of a $\mathbb{D}_n$-symmetric Birkhoff periodic orbit, i.e., $X_i=X_i^-+\frac{1}{2n} = X_i^+-\frac{1}{2n}$, and that $x-X\in \mathbb{X}_{K,0}$ has minimal period $K$. Then the minimal period of the sequence $z\in \Gamma^{\mathbb{Z}}$  defined by $z_i=\gamma(x_i)$ is the least common multiple of $n$ and $K$.
\end{proposition}
\begin{proof}
Our assumptions imply that $X_i^-< x_i < X_i^-+\frac{1}{n}$ and $X_i=X_i^-+\frac{1}{2n}$. Subtracting the latter equality from the former inequality produces $-\frac{1}{2n} <x_i-X_i<\frac{1}{2n}$ for all $i\in \mathbb{Z}$. Subsequently subtracting this from $ -\frac{1}{2n} < x_{M+i}-X_{M+i}< \frac{1}{2n}$, we obtain
\begin{equation}\label{eq:niceniceestimate}
-\frac{1}{n} < (x-X)_{M+i}-(x-X)_{i} < \frac{1}{n} \ \mbox{for all}\ i\in\mathbb{Z} \ \mbox{and all}\ M\in \mathbb{Z}\, .
\end{equation} 
Next, define $z_i:=\gamma(x_i)$. It holds that  $z_{M+i}=z_i$ if and only if $x_{M+i}-x_i \in \mathbb{Z}$. We write 
$$x_{M+i}-x_{i} = X_{M+i}-X_i + (x-X)_{M+i}-(x-X)_i = M\frac{m}{n} + (x-X)_{M+i}-(x-X)_i\, .$$ 
In view of formula \eqref{eq:niceniceestimate}, this expression is integer precisely when $M$ is a multiple of $n$ (because $\gcd(m,n)=1$), and $(x-X)_{M+i}=(x-X)_i$. In other words, $z_i=z_{M+i}$ if and only if $M$ is a multiple of both $n$ and $K$. So the minimal period of $z$ is ${\rm lcm}(n,K)$.  
\end{proof}   

\section{Non-Birkhoff orbits with dihedral symmetry}\label{sec:mainthmsection}
\noindent 
We now prove the main result of this paper, which states that, under a condition on the curvature of a billiard at a $\mathbb{D}_n$-symmetric periodic Birkhoff orbit, this billiard must possess symmetric periodic non-Birkhoff orbits. 
This result applies to billiards with $\mathbb{D}_n$-symmetry for any $n\in \mathbb{N}_{\geq 2}$.
\begin{theorem}\label{thm:mainthrm}
Let $m,n\in \mathbb{N}$ be co-prime, with $1\leq m \leq n-1$, and let $\Gamma$ be a $C^2$-smooth $\mathbb{D}_n$-symmetric strictly convex billiard.  Denote by  $Z_i = \gamma(X_i)$ one of its $\mathbb{D}_n$-symmetric $n$-periodic Birkhoff orbits of rotation number $\frac{m}{n}$. In particular, it has constant orbit segment length $L$ and curvature $\kappa$. 
    
Let $N\geq 1$ and let $H$ be a dihedral subgroup of $\mathbb{D}_n$ of order $2N$. Let $s\in \mathbb{N}_{\geq 2}$ satisfy $\gcd(s, N)=1$, and define $p := s n$ and $q :=s m$.  When 
\begin{equation}\label{eq:ineq-kappaL}
\kappa L < 2 \sin\left(\frac{m \pi }{n} \right) \cos^2\left( \frac{N\pi}{p}\right) ,
\end{equation}
then $\Gamma$ admits a non-Birkhoff periodic orbit $z_i=\gamma(x_i)$  with minimal period $p$, winding number $q$, and spatiotemporal symmetry group $H$. The lifts $x$ (of $z$) and $X$ (of $Z$) cross exactly $2N$ times per period $p$. When $n$ is odd and $p$ is even, then there are two geometrically distinct such orbits.
\end{theorem}
\noindent  Figures \ref{fig:main_thm} and \ref{fig:main_thm_pt2} illustrate Theorem \ref{thm:mainthrm} by displaying symmetric non-Birkhoff periodic orbits in billiards  with four distinct symmetry groups ($\mathbb{D}_3, \mathbb{D}_4, \mathbb{D}_5$ and $\mathbb{D}_7$). Figure \ref{fig:10.1D_6} shows six non-Birkhoff periodic orbits  in $\mathbb{D}_6$-symmetric billiards,  with mutually non-conjugate spatiotemporal symmetry groups.
\begin{figure}[h!]
\centering
\begin{subfigure}{.35\textwidth}
  \centering
  \begin{tikzpicture}
  \node at (0,0) {\includegraphics[width=\linewidth]{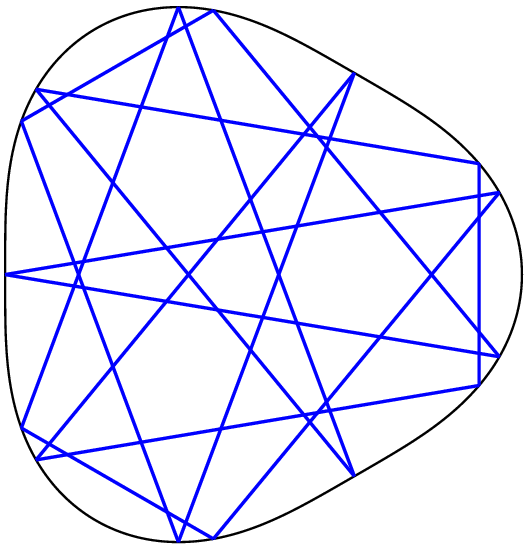}};
\node at (1,2.3) {$1$};
\node at (-2.6,-2.1) {$2$};
\node at (2.4,-1.3) {$3$};
\node at (2.4,1.3) {$4$};
\node at (-2.6,2.1) {$5$};
\node at (1,-2.3) {$6$};
\node at (-1,3) {$7$};
\node at (-2.7,-1.7) {$8$};
\node at (-0.5,-3) {$9$};
\node at (2.7,0.9) {$10$};
\node at (-3,0) {$11$};
\node at (2.7,-0.9) {$12$};
\node at (-0.5,3) {$13$};
\node at (-2.8,1.7) {$14$};
\node at (-1,-3) {$15$};
  \end{tikzpicture}
  \caption{}
  \label{fig:subb1}
\end{subfigure}\hspace{2cm}
\begin{subfigure}{.35\textwidth}
  \centering
  \begin{tikzpicture}
  \node at (0,0) {\includegraphics[width=\linewidth]{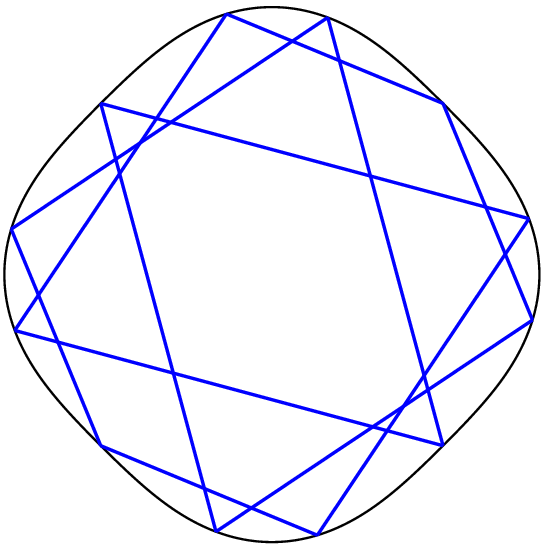}};
\node at (0.7,2.8) {$1$};
\node at (-2.9,0.5) {$2$};
\node at (-1.9,-1.9) {$3$};
\node at (0.5,-2.9) {$4$};
\node at (2.8,0.7) {$5$};
\node at (-1.9,1.9) {$6$};
\node at (-0.7,-2.9) {$7$};
\node at (2.8,-0.5) {$8$};
\node at (1.9,1.9) {$9$};
\node at (-0.5,2.9) {$10$};
\node at (-2.8,-0.7) {$11$};
\node at (1.9,-1.9) {$12$};
  \end{tikzpicture}
  \caption{}
  \label{fig:subb2}
\end{subfigure}\\
\begin{subfigure}{.35\textwidth}
  \centering
  \begin{tikzpicture}
  \node at (0,0) {\includegraphics[width=\linewidth]{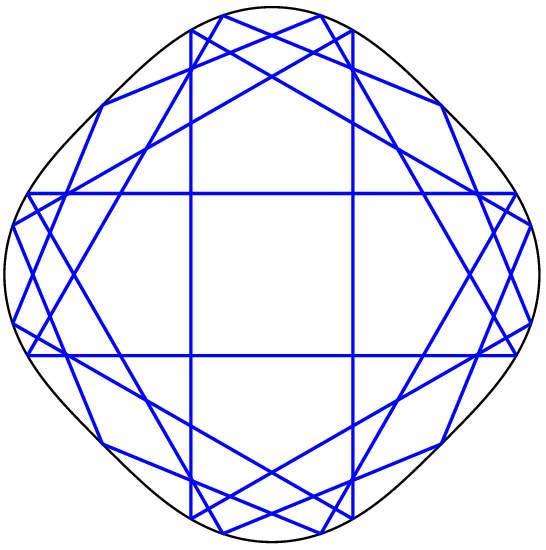}};
\node at (0.55,2.85) {$1$};
\node at (-1.9,1.9) {$2$};
\node at (-2.85,-0.55) {$3$};
\node at (0.9,-2.7) {$4$};
\node at (0.9,2.7) {$5$};
\node at (-2.85,0.55) {$6$};
\node at (-1.9,-1.9) {$7$};
\node at (0.55,-2.85) {$8$};
\node at (2.7,0.9) {$9$};
\node at (-2.75,0.9) {$10$};
\node at (-0.6,-2.85) {$11$};
\node at (1.9,-1.9) {$12$};
\node at (2.85,0.55) {$13$};
\node at (-1,2.7) {$14$};
\node at (-1,-2.7) {$15$};
\node at (2.85,-0.55) {$16$};
\node at (1.9,1.9) {$17$};
\node at (-0.6,2.85) {$18$};
\node at (-2.75,-0.9) {$19$};
\node at (2.75,-0.9) {$20$};
  \end{tikzpicture}
  \caption{}
  \label{fig:subb3}
\end{subfigure}\hspace{2cm}
\begin{subfigure}{.35\textwidth}
  \centering
  \begin{tikzpicture}
  \node at (0,0) {\includegraphics[width=\linewidth]{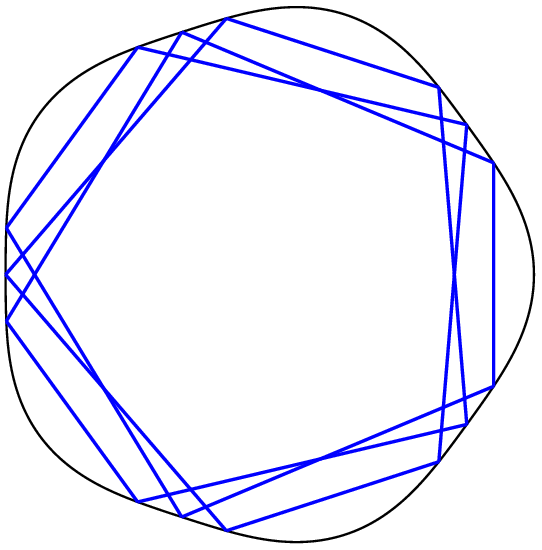}};
\node at (1.8,2.1) {$1$};
\node at (-1,2.7) {$12$};
\node at (-2.9,0.6) {$8$};
\node at (-0.5,-2.85) {$4$};
\node at (2.2,-1.7) {$15$};
\node at (2.5,1.25) {$11$};
\node at (-1.5,2.5) {$7$};
\node at (-2.9,0) {$3$};
\node at (-1.55,-2.55) {$14$};
\node at (2.5,-1.2) {$10$};
\node at (2.2,1.7) {$6$};
\node at (-0.5,2.85) {$2$};
\node at (-3,-0.6) {$13$};
\node at (-1,-2.7) {$9$};
\node at (1.8,-2.1) {$5$};
  \end{tikzpicture}
  \caption{}
  \label{fig:subb4}
\end{subfigure}
\caption{Visualization of Theorem \ref{thm:mainthrm} in  billiards of Lima\c con-type (see Example \ref{ex:limacon}). {\bf (a)} $\mathbb{D}_3$-symmetric $(15,5)$-periodic non-Birkhoff orbit in a $\mathbb{D}_3$-symmetric Lima\c con-type billiard with parameter $\alpha=0.099<\alpha^*(3)=0.1$; {\bf (b)} $\mathbb{D}_2$-symmetric $(12,3)$-periodic non-Birkhoff orbit in a $\mathbb{D}_4$-symmetric Lima\c con-type billiard with parameter $\alpha=0.05<\alpha^*(4)\approx 0.0588$; {\bf (c)} $\mathbb{D}_4$-symmetric $(20,5)$-periodic non-Birkhoff orbit in a $\mathbb{D}_4$-symmetric Lima\c con-type billiard with parameter $\alpha=0.055<\alpha^*(4)\approx 0.0588$; {\bf (d)} $\mathbb{D}_5$-symmetric $(15,3)$-periodic non-Birkhoff orbit in a $\mathbb{D}_5$-symmetric Lima\c con-type billiard with parameter $\alpha=0.035<\alpha^*(5)\approx 0.0385$. 
\label{fig:main_thm}}
\end{figure}
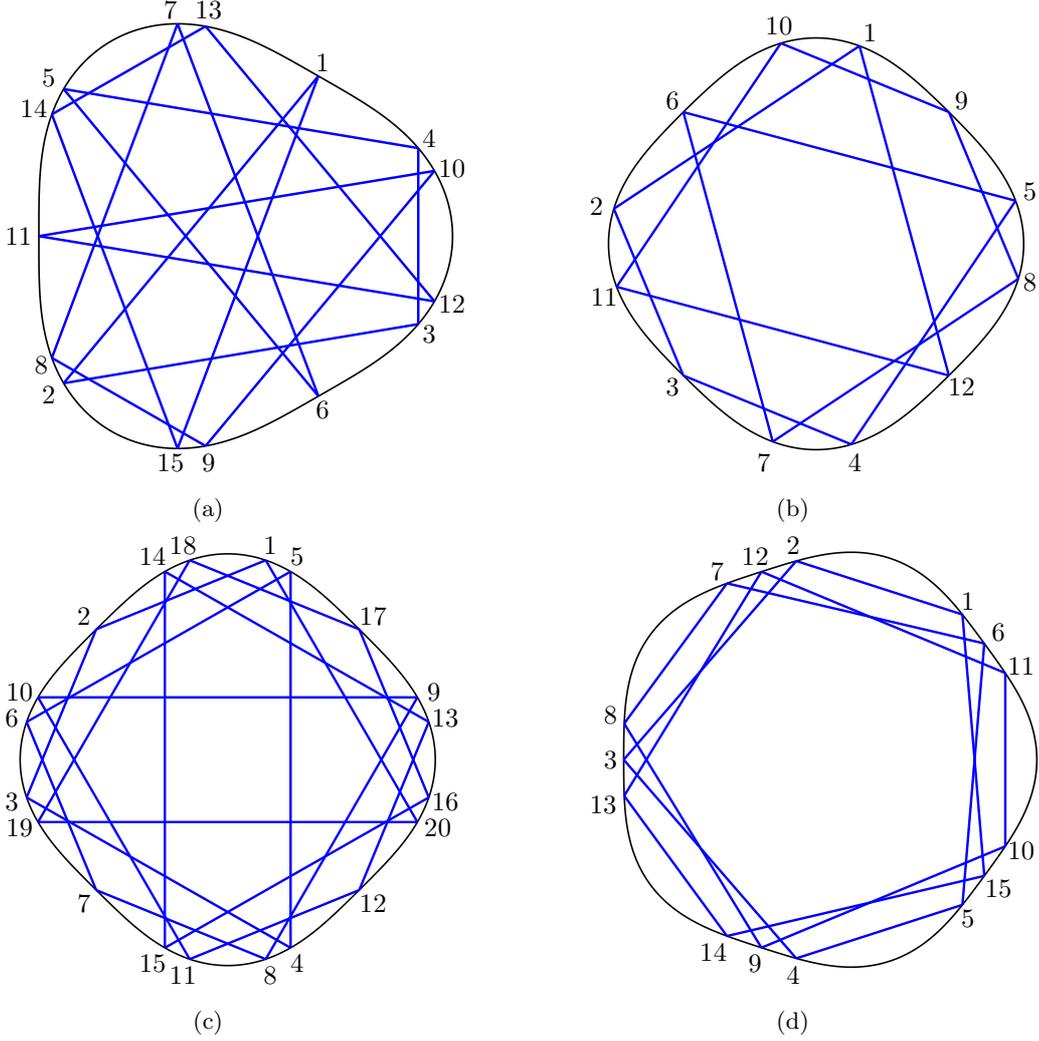
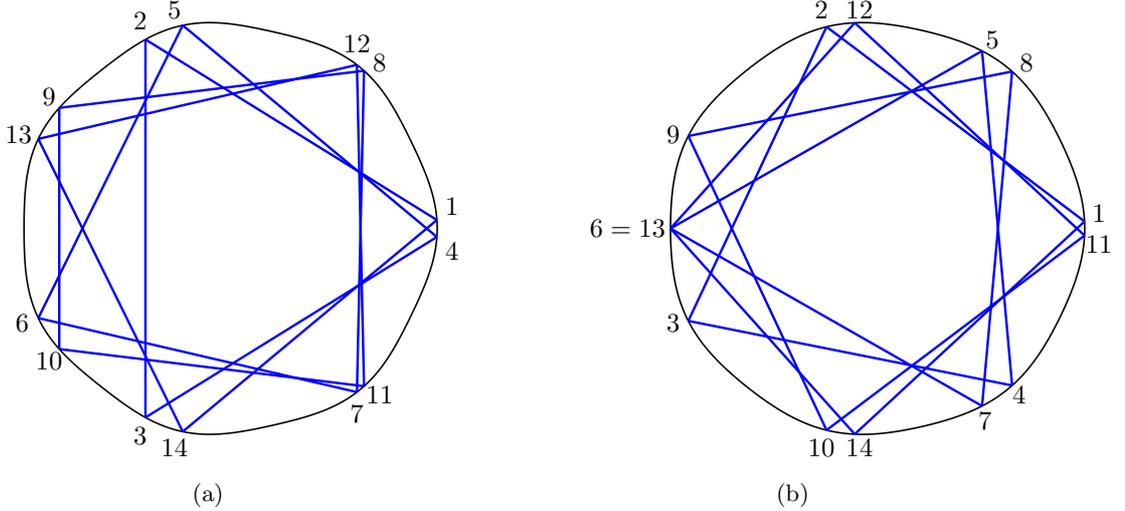
\begin{figure}[h!]
\centering
\begin{subfigure}{.35\textwidth}
  \centering
  \begin{tikzpicture}
  \node at (0,0) {\includegraphics[width=\linewidth]{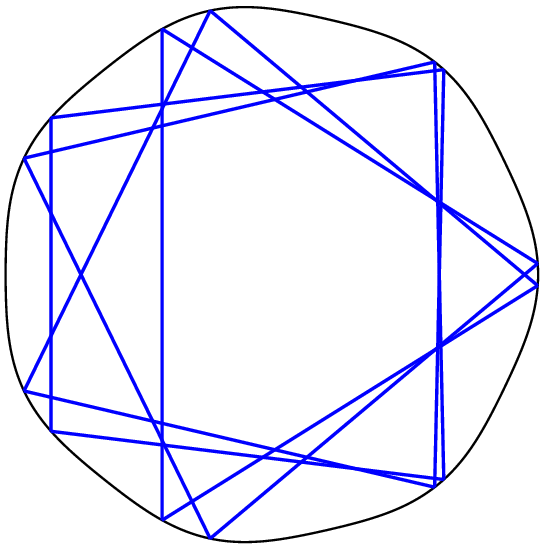}};
\node at (2.9,0.3) {$1$};
\node at (-1.2,2.75) {$2$};
\node at (-1.2,-2.75) {$3$};
\node at (2.9,-0.3) {$4$};
\node at (-0.75,2.9) {$5$};
\node at (-2.75,-1.25) {$6$};
\node at (1.65,-2.45) {$7$};
\node at (1.95,2.2) {$8$};
\node at (-2.4,1.75) {$9$};
\node at (-2.4,-1.75) {$10$};
\node at (1.95,-2.2) {$11$};
\node at (1.65,2.45) {$12$};
\node at (-2.8,1.25) {$13$};
\node at (-0.75,-2.9) {$14$};
  \end{tikzpicture}
  \caption{}
  \label{fig:subb_7a}
\end{subfigure}\hspace{2cm}
\begin{subfigure}{.35\textwidth}
  \centering
  \begin{tikzpicture}
  \node at (0,0) {\includegraphics[width=\linewidth]{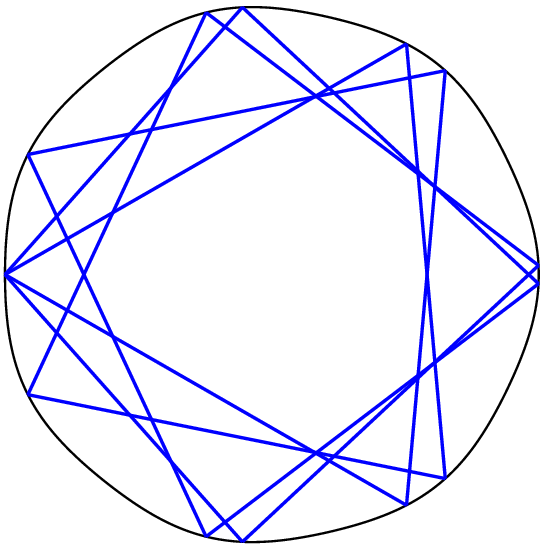}};
\node at (2.9,0.2) {$1$};
\node at (-0.75,2.9) {$2$};
\node at (-2.7,-1.25) {$3$};
\node at (1.85,-2.2) {$4$};
\node at (1.5,2.5) {$5$};
\node at (-3.3,0) {$6=13$};
\node at (1.4,-2.55) {$7$};
\node at (1.95,2.2) {$8$};
\node at (-2.7,1.25) {$9$};
\node at (-0.75,-2.9) {$10$};
\node at (2.9,-0.2)  {$11$};
\node at (-0.25,2.9) {$12$};
\node at (-0.25,-2.9) {$14$};
  \end{tikzpicture}
  \caption{}
  \label{fig:subb7_b}
\end{subfigure}
\caption{Visualization of Theorem \ref{thm:mainthrm} in two $\mathbb{D}_7$-symmetric billiards of Lima\c con-type (see Example \ref{ex:limacon}). Because $n=7$ is odd and $p=14$ is even, the theorem  predicts two geometrically distinct $(14,4)$-periodic non-Birkhoff orbits with spatiotemporal symmetry group $H = \langle S\rangle \cong \mathbb{D}_1$.  
{\bf (a)} $\mathbb{D}_1$-symmetric $(14,4)$-periodic non-Birkhoff orbit without points on the symmetry axis, in a $\mathbb{D}_7$-symmetric Lima\c con-type billiard with parameter $\alpha=0.015<\alpha^*(7)=0.02$; 
{\bf (b)} $\mathbb{D}_1$-symmetric $(14,4)$-periodic non-Birkhoff orbit with points on the symmetry axis, in a $\mathbb{D}_7$-symmetric Lima\c con-type billiard with parameter $\alpha=0.01<\alpha^*(7)=0.02$.  
\label{fig:main_thm_pt2}}
\end{figure}
\begin{figure}[h!]
\centering
\begin{subfigure}{.35\textwidth}
  \centering
\includegraphics[width=\linewidth]{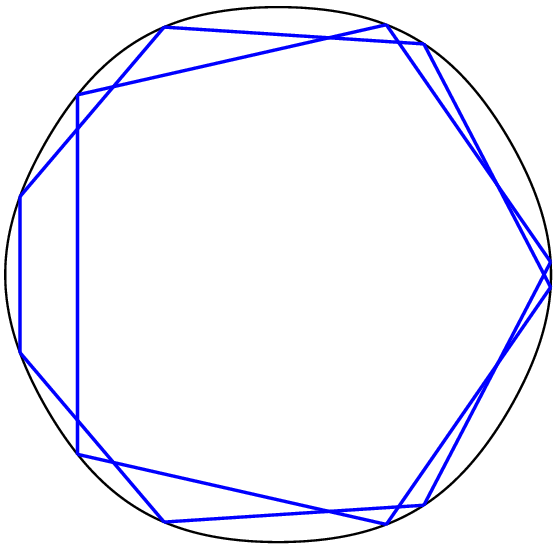}
  \caption{}
  \label{fig:10.1D_6a}
\end{subfigure}\hspace{2cm}
\begin{subfigure}{.35\textwidth}
  \centering
  \includegraphics[width=\linewidth]{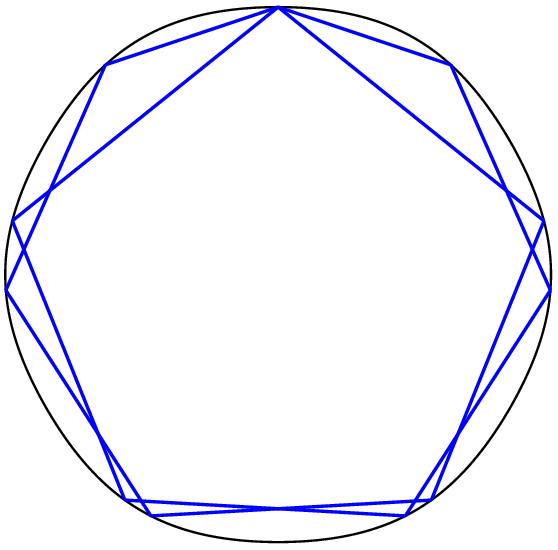}
  \caption{}
  \label{fig:10.1D_6b}
\end{subfigure}\\
\begin{subfigure}{.35\textwidth}
  \centering
\includegraphics[width=\linewidth]{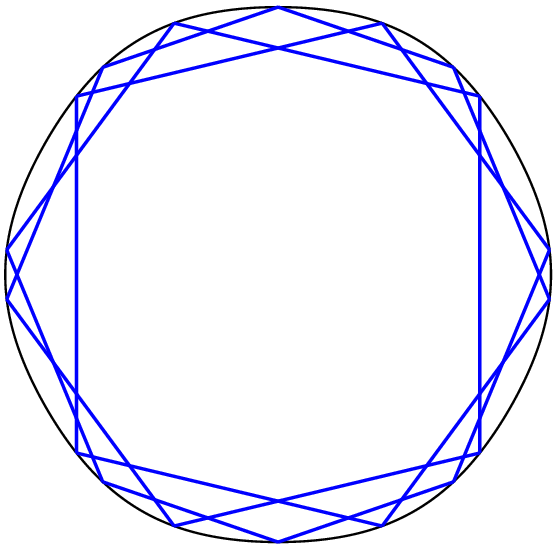}
  \caption{}
  \label{fig:10.1D_6c}
\end{subfigure}\hspace{2cm}
\begin{subfigure}{.35\textwidth}
  \centering
  \includegraphics[width=\linewidth]{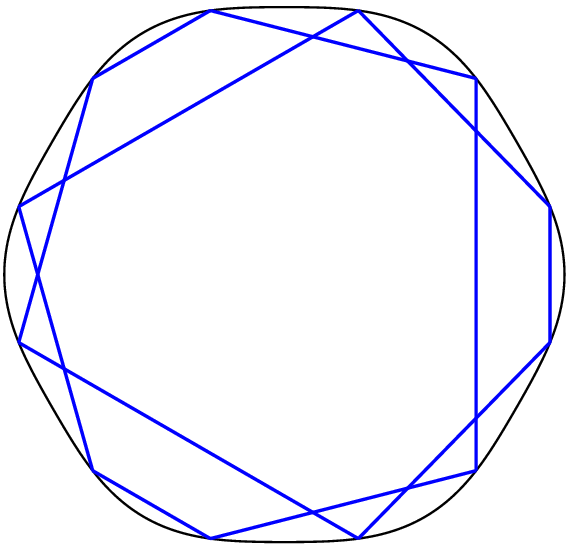}
  \caption{}
  \label{fig:10.1D_6d}
\end{subfigure}\\
\begin{subfigure}{.35\textwidth}
  \centering
\includegraphics[width=\linewidth]{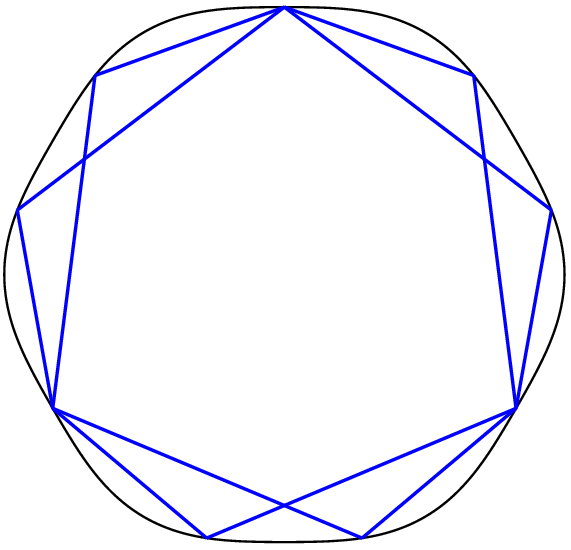}
  \caption{}
  \label{fig:10.1D_6e}
\end{subfigure}\hspace{2cm}
\begin{subfigure}{.35\textwidth}
  \centering
  \includegraphics[width=\linewidth]{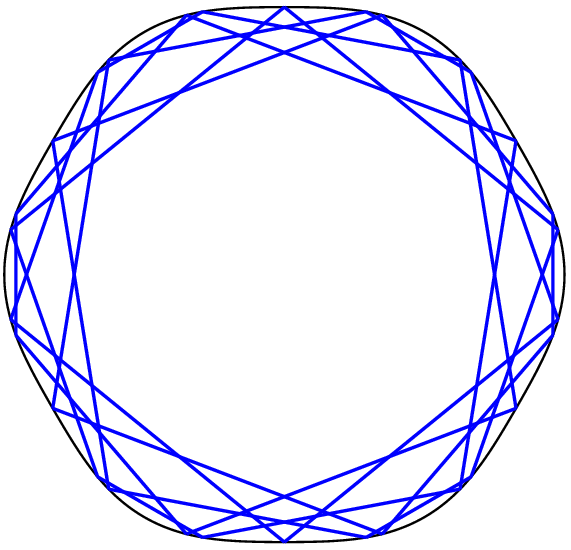}
  \caption{}
  \label{fig:10.1D_6f}
\end{subfigure}
\caption{Visualization of Theorem \ref{thm:sample} in  $\mathbb{D}_6$-symmetric billiards of Lima\c con-type (see Example \ref{ex:limacon}). As not all reflections in $\mathbb{D}_6$ are conjugate, $\mathbb{D}_6$ has six distinct non-conjugate dihedral subgroups $H$. For each of these, a non-Birkhoff periodic orbit with spatiotemporal symmetry group $H$  is displayed. Parameters were chosen so that the orbits exist, and their geometric properties are clearly visible. Namely, $\alpha=0.01<\alpha^*(6)\approx 0.027$ in subfigures \textbf{(a)}, \textbf{(b)} and \textbf{(c)}, while   $\alpha=0.023<\alpha^*(6)\approx 0.027$ in subfigures \textbf{(d)}, \textbf{(e)} and \textbf{(f)}.  \textbf{(a)} $\mathbb{D}_1$-symmetric $(12,2)$-periodic  non-Birkhoff orbit, reflection symmetric in the horizontal axis; \textbf{(b)} $\mathbb{D}_1$-symmetric $(12,2)$-periodic  non-Birkhoff orbit, reflection symmetric in the vertical axis; \textbf{(c)} $\mathbb{D}_2$-symmetric $(18,3)$-periodic non-Birkhoff orbit; \textbf{(d)} $\mathbb{D}_3$-symmetric $(12,2)$-periodic non-Birkhoff orbit, reflection symmetric the horizontal axis; \textbf{(e)} $\mathbb{D}_3$-symmetric $(12,2)$-periodic non-Birkhoff  orbit, reflection symmetric in the vertical axis; 
\textbf{(f)} $\mathbb{D}_6$-symmetric $(30,5)$-periodic non-Birkhoff orbit.
\label{fig:10.1D_6}
}
\end{figure}

\begin{proof} 
Let $Z$ be a $\mathbb{D}_n$-invariant Birkhoff orbit of period $n$ and rotation number $\frac{m}{n}$ as in the statement of the theorem. Its existence is guaranteed by Lemma  \ref{lem:existenceBirkhoffDn}. The lemma also states that its lift $X\in \mathbb{X}_{n,m}$ is given by $X_i=\frac{A}{2n} + \frac{m}{n}i$, for some $A\in \mathbb{Z}$, see \eqref{eqn:Xaformula}.  It follows that
\begin{equation}\label{eq:RmZj}
R^m(Z_i)=R^m(\gamma(X_i)) = \gamma(X_i+m/n) = \gamma(X_{i+1}) = Z_{i+1}\, .
\end{equation}
Moreover, choosing $B,C\in \mathbb{Z}$ such that $Bm+Cn=1$, we find that  $X_{i}+X_{-AB-i}= AC \in \mathbb{Z}$, so that 
\begin{equation}\label{eq:SZj}
S(Z_i)= S(\gamma(X_i)) = \gamma(-X_i)= \gamma(X_{-AB- i} -AC) = \gamma(X_{-AB- i}) =Z_{-AB-i}\, .  
\end{equation}
Let $H\subset \mathbb{D}_n$ be a dihedral subgroup of order $2N$. Then $N$ divides $n$ and $H=\langle \rho, \sigma\rangle$  is  generated by the rotation $\rho=R^{n/N}$ and a reflection $\sigma=R^bS$ for some $b\in \mathbb{Z}$. 
From \eqref{eq:RmZj} it follows that 
$$\rho^{m}(Z_i) = (R^{n/N})^{m}(Z_i)  = (R^{m})^{n/N}(Z_i) =Z_{(n/N)+i},$$ 
while \eqref{eq:SZj} gives that 
$$\sigma(Z_i)=(R^bS)(Z_i) = R^{b(Bm+Cn)}S(Z_i)  = (R^m)^{bB}S(Z_i) = (R^m)^{bB}(Z_{-AB-i}) = Z_{(bB-AB)-i}.$$ 
For any $s\in \mathbb{N}_{\geq 2}$, we thus have
\begin{equation}\label{eq:Zequalities}
\rho^{sm}(Z_i) =Z_{K+i} \ \mbox{and} \  \sigma(Z_i)=Z_{k-i} \ \mbox{where}\ K:=s(n/N) \ \mbox{and}\ k := B(b-A). 
\end{equation}
For later reference, we remark that the integers $k$ and $K$ in \eqref{eq:Zequalities} are not unique: because $Z$ is $n$-periodic, we may add an arbitrary integer multiple of $n$ to both $k$ and $K$, and \eqref{eq:Zequalities} will remain true.  

Next, let us fix an $s\in \mathbb{N}_{\geq 2}$ with $\gcd(s,N)=1$, and define $p:=s n, q:=sm$ as in the assumptions of the theorem. 
Note that $Z$ does not only have period $n$, but also (nonminimal) period $p=sn$. Similarly, $X$ is not only an element of $\mathbb{X}_{n,m}$ but also of $\mathbb{X}_{p,q} =\mathbb{X}_{sn,sm}$. 
We will now study general $p$-periodic sequences $z$, with a lift $x\in \mathbb{X}_{p,q}$, satisfying the same equalities as \eqref{eq:Zequalities}, that is,
\begin{equation}\label{eq:zequalities}
\rho^{sm}(z_i)=  z_{K+i} \ \mbox{and}\  \sigma(z_i)= z_{k-i}\, .
\end{equation}
By Lemma \ref{liftlemma}, a billiard sequence $z$ with lift $x\in \mathbb{X}_{p,q}$ satisfies \eqref{eq:zequalities} if and only if the lift satisfies 
\begin{equation}\label{eq:veqn}
(x-X)_{K+i}-(x-X)_i = 0\ \mbox{and}\  (x-X)_{k-i} +(x-X)_i = 0 \ \mbox{for all}\ i\in \mathbb{Z} \, .
\end{equation}
\changes{This can be seen by noting that $x$ and $X$ both satisfy the equations in Lemma \ref{liftlemma}, parts {\it i)} and {\it iv)}, so that their difference $x-X$ satisfies \eqref{eq:veqn}.}
An example of such a sequence is $\hat x= X+ \varepsilon v \in \mathbb{X}_{p,q}$, where $0<\varepsilon <\frac{1}{2n}$, and $v\in \mathbb{X}_{K,0}$ is given by 
$$v_i := \sin\left(\frac{2\pi i}{K} -\frac{\pi k}{K}\right)  
\changes{=\cos\left(\frac{\pi Nk}{p}\right)\sin\left(\frac{2\pi Ni}{p}\right)-\sin\left(\frac{\pi Nk}{p}\right)\cos\left(\frac{2\pi Ni}{p}\right)} ,$$
\changes{where the equality follows from the angle difference formula for the sine and the fact that $K=p/N$.}
Indeed, it is easy to check that $v_{i+K}-v_i=0$ and $v_{k-i}+v_i=0$, so $\hat x$ satisfies \eqref{eq:veqn}. 
To ensure that $v$ is not the zero vector, we shall assume from here on that $K \geq 3$. But note that $K\geq 2$ by assumption and that when $K=p/N=2$, then $\cos^2 \left( \frac{N\pi}{p}\right) = \cos^2 \left( \frac{\pi}{2}\right) = 0$ and the theorem is silent anyway.  

By Lemma \ref{lem:Hessianlemma}, $v$ is  an eigenvector of the Hessian $D^2W_{p,q}(X)$ (when the billiard is parametrized at constant speed $c=\int_{0}^1\|\gamma'(x)\|dx$, which we may assume without loss of generality) with eigenvalue  
$$\lambda_v = 2c^2\sin \left( \frac{m\pi}{n} \right) \left( \frac{2\sin\left( \frac{m\pi}{n} \right)  \cos^2\left( \frac{N\pi}{p} \right) }{L}  - \kappa \right).
$$
See also Remark \ref{remk:eigenvalue}. This eigenvalue is positive precisely when $$\kappa L <   2 \sin\left(\frac{m \pi }{n} \right) \cos^2\left( \frac{N\pi}{p}\right)  , $$ which is exactly the assumption of the theorem. For any small enough $0<\varepsilon<\frac{1}{2n}$ we therefore have
$$W_{p,q}(\hat x) > W_{p,q}(X)\, .$$ 
Our second  observation about $v$ is that it has exactly $2$ sign changes per period $K$, and hence $2N$ sign changes per period $p=NK$. This number of sign changes is minimal among nonzero vectors $v$ satisfying $v_{k-i}+v_i=0$ and $v_{K+i}-v_i=0$, because any such $v$ with less sign changes is necessarily zero.  
It also follows from the number of sign changes of $v$ that the minimal period of $v$ is $K$.     

The sequence $\hat x = X+\varepsilon v$ will  serve as initial condition for the gradient flow. We shall denote by $x(t)$ the gradient flow line with initial condition $x(0)=\hat x$. We will prove that this gradient flow line exists for all $t\geq 0$ and converges (along a subsequence) to a non-Birkhoff periodic orbit as $t\to \infty$. 

To show this, let us denote by $X^-, X^+$ the  $\mathbb{D}_n$-invariant Birkhoff sequences given by $X^-_i =X_i - \frac{1}{2n}$ and $X^+_i=X_i+\frac{1}{2n}$. By Lemma \ref{lem:existenceBirkhoffDn} these are lifts of billiard orbits, i.e., they are stationary points of the gradient flow. The order interval $[X^-, X^+]$ is thus positively invariant under the gradient flow, and we have that $\hat x, X\in (X^-, X^+)$ because $0<\varepsilon<\frac{1}{2n}$. By choosing $\varepsilon$ smaller if necessary, we can moreover arrange that $W_{p,q}(\hat x) > W_{p,q}(X)$, and that $\hat x \in \Sigma_{\delta}$, where $\delta>0$ is as in the conclusion of Corollary \ref{cor:gradientlimit}. By Corollary \ref{cor:gradientlimit}, the flow line $x(t)$ is then defined for all $t\geq 0$, lies entirely in $[X^-, X^+]\cap \mathbb{X}_{p,q}\cap \Sigma_{\delta}$ and converges (along a subsequence) to a stationary point $x_{\infty} \in [X^-, X^+]\cap \mathbb{X}_{p,q}\cap \Sigma_{\delta}$. The sequence $z_{\infty} \in  \Gamma^{\mathbb{Z}}$ defined by  $z_{\infty,i} :=\gamma(x_{\infty,i})$ is therefore a billiard trajectory. We claim that $z_{\infty}$ is non-Birkhoff and satisfies all the conclusions of the theorem. 

To prove this claim, we start by recalling that $(\hat x-X)_{K+i}-(\hat x-X)_{i}=0$ and $(\hat x -X)_{k-i}+(\hat x-X)_{i}=0$. Because these equalities are preserved by the gradient flow according to Lemma \ref{lem:symmetrypreservation}, and the space of sequences satisfying these equalities is closed, it follows that $(x_{\infty}-X)_{K+i}-(x_{\infty}-X)_{i}=0$ and $(x_{\infty}-X)_{k-i}+(x_{\infty}-X)_{i}=0$. Thus, $$\rho^{sm}(z_{\infty,i}) = z_{\infty,K+i} \quad \text{and} \quad \sigma(z_{\infty,i}) = z_{\infty,k-i}.$$ This proves that $z_{\infty}$ is $H$-symmetric, i.e., that $H\subset H(z_{\infty})$. We  prove below that $H = H(z_{\infty})$.
    
Next, we show that $x_{\infty} \neq X^-, X, X^+$. Recall that $x_{\infty}\in [X^-, X^+]$. Because $(X^{\pm} -X)_{k-i} +(X^{\pm}-x)_i = \pm \frac{1}{n} \neq 0$, we see that $X^{\pm}$ do not satisfy one of the  equalities that $x_{\infty}$ satisfies, so $x_{\infty}\neq X^{\pm}$. By the strong comparison principle we therefore have $x_{\infty} \in (X^-, X^+)$. It is also clear that $x_{\infty} \neq X$ because $W_{p,q}$ can only increase  along gradient flow lines, so that $$W_{p,q}(x_{\infty}) \geq W_{p,q}(\hat x) > W_{p,q}(X).$$ 
Now we show that \changes{$z_{\infty}$} has minimal period $p$ and is not Birkhoff. Recall that $\hat x-X\in \mathbb{X}_{K,0}$ has $2$ sign changes per period $K$. It follows from this that $x_{\infty}-X\in \mathbb{X}_{K,0}$ has at most $2$ sign changes per period $K$, as this number can only decrease along gradient flow lines. However, if $x_{\infty}-X$ had zero sign changes, then $x_{\infty}-X=0$, but we already proved that $x_{\infty}\neq X$. Hence, $x_{\infty}-X$ has exactly two sign changes per period $K$, \changes{i.e., $x_{\infty}$ and $X$  have $2N$ crossings per period $p$.} In particular, the minimal period of $x_{\infty}-X$ equals $K$. Thus,  Proposition \ref{prop:minperiod} implies that the minimal period of $z_{\infty}$ is $\mbox{lcm}(n,K)= \mbox{lcm}(n, sn/N) = sn = p$, because $\gcd(s, N)=1$.  \changes{We conclude that $x_{\infty}\in \mathbb{X}_{p,q}$, while $x_{\infty}\not\in \mathbb{X}_{n,m}$. It therefore follows from Proposition \ref{prop:pqnm} that $x_{\infty}$ and $z_{\infty}$ are not Birkhoff. }

At this point, we can also prove that $H=H(z_{\infty})$. Indeed, if $H\subsetneq H(z_{\infty})$  then $H(z_{\infty})$ would contain a rotation $\tilde \rho$ of order $\tilde N>N$ where $N$ divides $\tilde N$. This would imply that  $x_{\infty}-X$ has at least $2\tilde N$ sign changes (because $x_{\infty}\neq X$), which is more than the number $2N$ of sign changes of $\hat x-X$. This is impossible as the number of intersections can only decrease along gradient flow lines. Thus, $H=H(z_{\infty})$.
 
This finishes the proof that there is at least one non-Birkhoff periodic orbit of minimal period $p$, winding number $q$, and spatiotemporal symmetry group $H$. In the special case that $n$ is odd and $p$ is even, we now prove that there are two such orbits. Recall that the integer $k$ in the above proof is not unique: equations \eqref{eq:Zequalities} hold for $k=B(b-A)$ but also for any other $k$ of the form $k=B(b-Z)+tn$. In particular, when $n$ is odd, we can choose both an odd $k$, say $k^{\rm odd}$, and an even $k$, say $k^{\rm even}$, for which \eqref{eq:Zequalities} holds (we do not vary $K$). Our proof shows that there  are $(p,q)$-periodic non-Birkhoff orbits $z_{\infty}^{\rm odd}$ and  $z_{\infty}^{\rm even}$ satisfying \eqref{eq:zequalities}, respectively for $k=k^{\rm odd}$ and $k=k^{\rm even}$.
     
We claim that these orbits are geometrically distinct. If not, then there would be an $l\in \mathbb{Z}$ such that $(z_{\infty}^{\rm even})_{i}=(z_{\infty}^{\rm odd})_{l+i}$ or $(z_{\infty}^{\rm even})_{i}=(z_{\infty}^{\rm odd})_{l-i}$. In the first case, it would follow that  $$(z_{\infty}^{\rm even})_{i} = \sigma((z_{\infty}^{\rm even})_{k^{\rm even}-i})= \sigma((z_{\infty}^{\rm odd})_{l+ k^{\rm even}-i}) = (z_{\infty}^{\rm odd})_{k^{\rm odd} - (l+ k^{\rm even}-i)}  = (z_{\infty}^{\rm even})_{k^{\rm odd} - k^{\rm even} - 2l +i}\, .$$ Because $k^{\rm odd} - k^{\rm even} - 2l$ is obviously nonzero and $z_{\infty}^{\rm even}$ has period $p$,  this would imply that $z$ also has a period lower than $p$. This is a contradiction because $z_{\infty}^{\rm even}$ has minimal period $p$. In the second case, where $(z_{\infty}^{\rm even})_{i}=(z_{\infty}^{\rm odd})_{l-i}$,  we would find that $$(z_{\infty}^{\rm even})_{i} = \sigma((z_{\infty}^{\rm even})_{k^{\rm even}-i})= \sigma((z_{\infty}^{\rm odd})_{l- (k^{\rm even}-i)}) = (z_{\infty}^{\rm odd})_{k^{\rm odd} - l+ k^{\rm even}-i}  = (z_{\infty}^{\rm even})_{2l - k^{\rm odd} - k^{\rm even}  +i}\, ,$$ once more contradicting that $z_{\infty}^{\rm even}$ has minimal period $p$. This proves that $z_{\infty}^{\rm odd}$ and $z_{\infty}^{\rm even}$ are geometrically distinct orbits.
\end{proof}
     
\begin{remark}
The quantity $\kappa L$ is invariant under rescaling the billiard. Indeed, if the billiard is rescaled by a factor $C$, then $L$ is multiplied by a factor $C$ and $\kappa$ is multiplied by a factor $C^{-1}$.
\end{remark}

\begin{remark}
Theorem \ref{thm:sample} in the introduction follows immediately from Theorem \ref{thm:mainthrm}. 
\end{remark}

\section{Non-Birkhoff orbits in billiards with $\mathbb{D}_2$-symmetry} 
\label{sec:D2section}
\noindent Theorem \ref{thm:mainthrm} proves the existence of $\mathbb{D}_N$-symmetric non-Birkhoff orbits in billiards with $\mathbb{D}_n$-symmetry. For $n=N=2$, these orbits are of type I, as described in Lemma \ref{lem:periodicngeq3}. In this section, we give  conditions under which $\mathbb{D}_2$-symmetric billiards admit non-Birkhoff periodic orbits type II (see again Lemma \ref{lem:periodicngeq3}) and type {\rm V} (see Lemma \ref{lem:D1orbitlemma}). We first repeat Theorem \ref{thm:mainthrm} for the case $n=N=2$.

\begin{theorem}[{\bf Non-Birkhoff orbits of type I}] \label{thm:mainthrmI}
Let $\Gamma$ be a $\mathbb{D}_2$-symmetric billiard, and denote by $Z_i = \gamma(X_i)$ one of its $\mathbb{D}_2$-symmetric $2$-periodic Birkhoff orbits. In particular, it has constant orbit segment length $L$ and curvature $\kappa$. 
Let $s\geq 3$ be odd and define $p=2s$. 
When  
$$\kappa L < 2 \cos^2\left( \frac{2\pi}{p}\right),$$ then $\Gamma$ admits a $\mathbb{D}_2$-symmetric non-Birkhoff periodic orbit $z_i=\gamma(x_i)$ of type {\rm I}, of minimal period $p$ and rotation number $\frac{1}{2}$. The lifts $x$ (of $z$) and $X$ (of $Z$) cross exactly $4$ times per period $p$. 
 \end{theorem}

\noindent We refer to Figure \ref{fig:D2typeI} for a visualization of two orbits of type I.

\begin{proof}
This follows directly from  Theorem \ref{thm:mainthrm} above, applied to $n=N=2, m=1$, and $s$ odd. We note that the billiard orbits guaranteed by  Theorem \ref{thm:mainthrm} are of type I by construction, as they satisfy equations \eqref{eq:zequalities} for certain $k,K\in \mathbb{Z}$. 
\end{proof}

\begin{figure}[h!]
\centering
\begin{subfigure}{.35\textwidth}
  \centering
 \begin{tikzpicture}
  \node at (0,0) {\includegraphics[angle=90,origin=c,width=.75\linewidth]{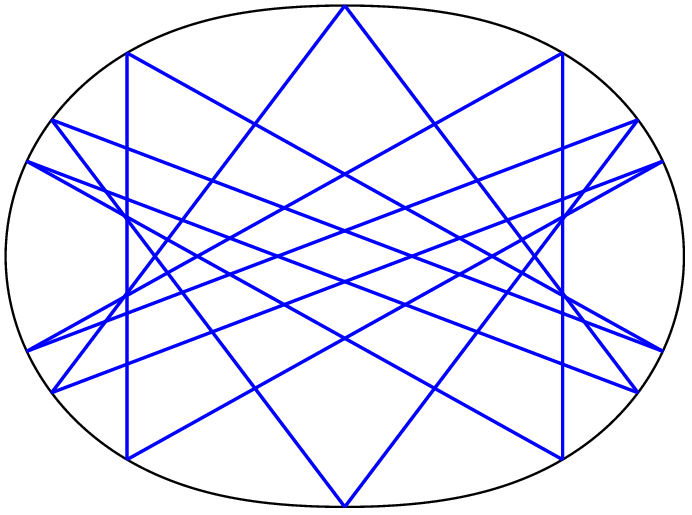}};
\node at (0.8,2.8) {$1$};
\node at (-1.8,-1.9) {$2$};
\node at (1.8,-1.9) {$3$};
\node at (-0.8,2.8) {$4$};
\node at (1.2,-2.6) {$5$};
\node at (-2.2,0) {$6$};
\node at (1.2,2.6) {$7$};
\node at (-0.8,-2.8) {$8$};
\node at (1.8,1.9) {$9$};
\node at (-1.9,1.9) {$10$};
\node at (0.8,-2.8) {$11$};
\node at (-1.2,2.6) {$12$};
\node at (2.3,0) {$13$};
\node at (-1.2,-2.6) {$14$};
  \end{tikzpicture}
  \caption{}
  \label{fig:fig:D2typeIa}
\end{subfigure}\hspace{2cm}
\begin{subfigure}{.35\textwidth}
  \centering
 \begin{tikzpicture}
  \node at (0,0) {\includegraphics[width=\linewidth]{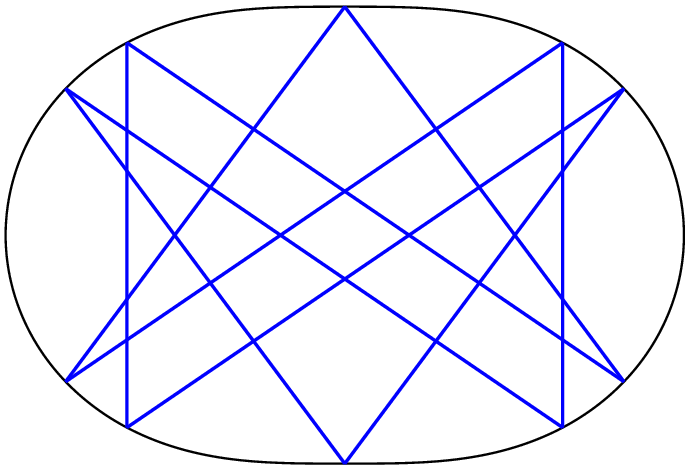}};
\node at (1.9,1.8) {$1$};
\node at (-2.4,-1.3) {$2$};
\node at (0,2.1) {$3$};
\node at (2.4,-1.3) {$4$};
\node at (-1.9,1.8) {$5$};
\node at (-1.9,-1.8) {$6$};
\node at (2.4,1.3) {$7$};
\node at (0,-2.1) {$8$};
\node at (-2.4,1.3) {$9$};
\node at (1.9,-1.8) {$10$};
  \end{tikzpicture}
  \caption{}
   \label{fig:fig:D2typeIb}
\end{subfigure}
\caption{Visualization of Theorem \ref{thm:mainthrmI} in a $\mathbb{D}_2$-symmetric billiard of Lima\c con-type (see Example \ref{ex:limacon}). {\bf (a)} $(14,7)$-periodic non-Birkhoff orbit of type I satisfying $R(z_i)=z_{7+i}$ and $S(z_i)=z_{12-i}$. Billiard parameter given by  $\alpha=0.15<\alpha^*(2)=0.2$; {\bf (b)} $(10,5)$-periodic  non-Birkhoff orbit of type I satisfying  $R(z_i)=z_{5+i}$ and $S(z_i)=z_{11-i}$. Billiard parameter given by $\alpha=0.195<\alpha^*(2)=0.2$. }
 \label{fig:D2typeI}
\end{figure}
\noindent Orbits of type II are not guaranteed by Theorem \ref{thm:mainthrm}. The next theorem thus requires a separate proof.
\begin{theorem}[{\bf Non-Birkhoff orbits of type  II}] \label{thm:mainthrmII}
Let $\Gamma$ be a $\mathbb{D}_2$-symmetric billiard, and denote by $Z_i = \gamma(X_i)$ one of its $\mathbb{D}_2$-symmetric $2$-periodic Birkhoff orbits. In particular, it has constant orbit segment length $L$ and curvature $\kappa$. 
Let $s\geq 2$ and define $p=2s$. When  
$$\kappa L < 2 \cos^2\left( \frac{\pi}{p}\right) ,$$
then $\Gamma$ admits a $\mathbb{D}_2$-symmetric non-Birkhoff periodic orbit $z_i=\gamma(x_i)$ of type {\rm II}, of minimal period $p$ and rotation number $\frac{1}{2}$. The lifts $x$ (of $z$) and $X$ (of $Z$) cross exactly $2$ times per period $p$.
 \end{theorem}
\noindent We refer to Figure \ref{fig:D2typeII} for a visualization of two orbits of type II.
 \begin{figure}[h!]
\centering
\begin{subfigure}{.35\textwidth}
  \centering
 \begin{tikzpicture}
  \node at (0,0) {\includegraphics[angle=90,origin=c,width=.7\linewidth]{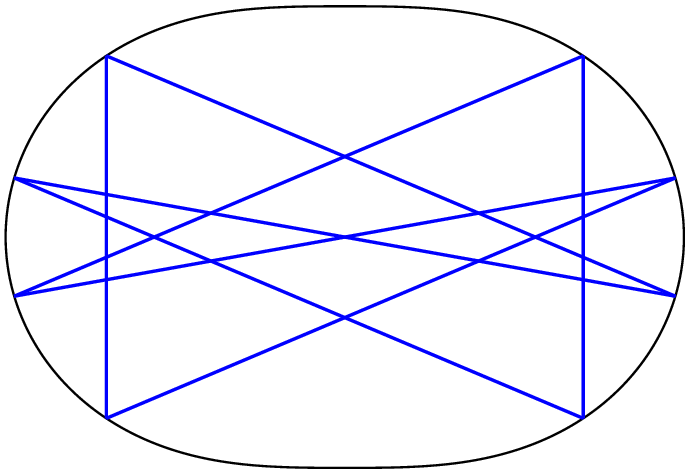}};
\node at (0.5,2.9) {$1$};
\node at (-1.7,-2) {$2$};
\node at (1.7,-2) {$3$};
\node at (-0.5,2.9) {$4$};
\node at (0.5,-2.9) {$5$};
\node at (-1.7,2) {$6$};
\node at (1.7,2) {$7$};
\node at (-0.5,-2.9) {$8$};
  \end{tikzpicture}
  \caption{}
  \label{fig:fig:D2typeIIa}
\end{subfigure}\hspace{2cm}
\begin{subfigure}{.35\textwidth}
  \centering
 \begin{tikzpicture}
  \node at (0,0) {\includegraphics[width=\linewidth]{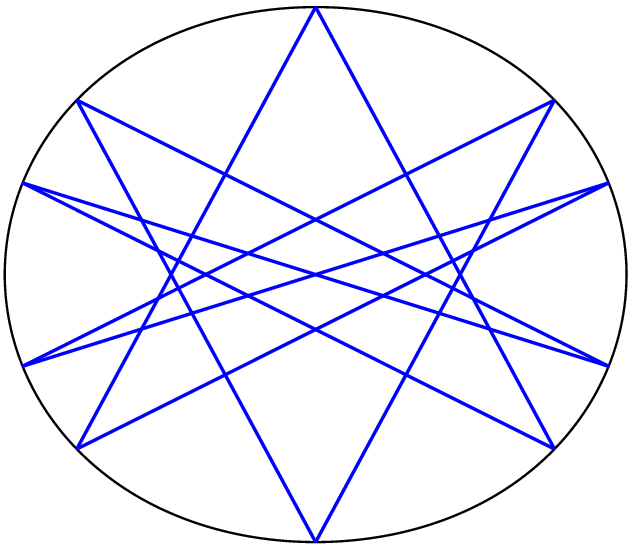}};
\node at (2.2,1.7) {$1$};
\node at (2.7,0.9) {$9$};
\node at (0,-2.6) {$2$};
\node at (-2.2,1.7) {$3$};
\node at (-2.8,0.9) {$5$};
\node at (2.2,-1.7) {$6$};
\node at (2.7,-0.9) {$4$};
\node at (0,2.6) {$7$};
\node at (-2.2,-1.7) {$8$};
\node at (-2.8,-0.9) {$10$};
  \end{tikzpicture}
  \caption{}
   \label{fig:fig:D2typeIIb}
\end{subfigure}
 \caption{  
 Visualization of Theorem \ref{thm:mainthrmII} in a $\mathbb{D}_2$-symmetric billiard of Lima\c con-type (see Example \ref{ex:limacon}). {\bf (a)} $(8,4)$-periodic non-Birkhoff orbit of type II  satisfying $R(z_i)=z_{9-i}$ and $S(z_i)=z_{4+i}$. Billiard parameter given by $\alpha=0.19<\alpha^*(2)=0.2$; {\bf (b)} $(10,5)$-periodic non-Birkhoff orbit of type II satisfying  $R(z_i)=z_{9-i}$ and $S(z_i)=z_{5+i}$. Billiard parameter given by $\alpha=0.075<\alpha^*(2)=0.2$. \label{fig:D2typeII}}
\end{figure} 
\begin{proof}
Let $Z$ be a Birkhoff orbit of  period $2$ as in the statement of the theorem. Note that its lift $X\in \mathbb{X}_{2,1}$ is then given by $X_i=X_0+\frac{1}{2}i$, where either $X_0\in \frac{1}{2}\mathbb{Z}$ (when $S(Z_i)=Z_i$) or $X_0\in \frac{1}{4} + \frac{1}{2}\mathbb{Z}$ (when $S(Z_i)\neq Z_i$).  
Let $s \geq 2$ be an integer as in the assumption of the theorem, and define $p:=2s$ and $q:=s$. Clearly, $Z$ does not only have period $2$, but also (nonminimal) period $p=2s$. Similarly, it holds that $X\in \mathbb{X}_{2s,s} = \mathbb{X}_{p,q}$. 

Because $Z$ has period $2$ and $R(Z_i)\neq Z_i$, we have that $R(Z_i)=Z_{K-i}$ for any odd integer $K\in \mathbb{Z}$. 
Moreover, we claim that we can always choose a reflection $\sigma\in \{S, RS\}$ such that $\sigma(Z_i)=Z_{s+i}$. To see this, note that $S(Z_i)=Z_i$ if and only if $(RS)(Z_i)\neq Z_i$ and vice versa. Thus, if $S(Z_i)\neq Z_{s+i}$ then necessarily $(RS)(Z_i)=Z_{s+i}$. This proves our claim.

We now  consider general sequences $z$ with a lift $x\in \mathbb{X}_{2s,s} = \mathbb{X}_{p,q}$. 
In view of Lemma \ref{liftlemma}, such a sequence satisfies $R(z_i)= z_{K-i}$ and $\sigma(z_i)= z_{s+i}$ if and only if
\begin{equation}\label{eq:veqnII}
   (x-X)_{K-i}-(x-X)_i = 0\ \mbox{and}\  (x-X)_{s+i} +(x-X)_i = 0  \, .
    \end{equation}
    An example of a sequence satisfying \eqref{eq:veqnII} is $\hat x= X+\varepsilon v\in (X^-, X^+)$, where $0<\varepsilon <\frac{1}{4}$, and $v$ is given by
      $$v_i=  \cos\left(\frac{\pi i}{s} - \frac{\pi K}{2s} \right) = \cos\left(\frac{2\pi i}{p} - \frac{\pi K}{p} \right).  $$
Note that $v\in \mathbb{X}_{2s,0}=\mathbb{X}_{p,0}$, which  implies that $\hat x \in \mathbb{X}_{2s,s}=\mathbb{X}_{p, q}$. Moreover, $v$ has $2$ sign changes per period $p=2s$, and this number of sign changes is minimal among nonzero vectors $v$ satisfying $v_{s+i}+v_i=0$, as any such $v$ without sign changes is necessarily trivial. 
It is also clear that $v$ is  an eigenvector of the Hessian $D^2W_{p,q}(X)$, its eigenvalue being 
     $$\lambda =  
    2   
    \left( \frac{2 \cos^2\left( \frac{\pi}{p} \right) }{L}  - \kappa \right)  .
    $$
This eigenvalue is positive exactly under the condition of the theorem. 

We omit the remainder of the proof, as it is more or less identical to that of Theorem \ref{thm:mainthrmI}.
\end{proof}
\noindent 
The final result in this section proves the existence of non-Birkhoff periodic orbits that are only reflection-symmetric, meaning that  their spatiotemporal symmetry group is generated by a single reflection. For such orbits, we can only prove the existence of orbits of type {\rm V} (see Lemma \ref{lem:D1orbitlemma}). We recall that the reflection symmetry acts on  orbits of type {\rm V} both time-preserving and time-reversing. In particular, such orbits  traverse the same path both forward and backward during each period.

\begin{theorem}[{\bf Non-Birkhoff orbits of type {\rm V}}] \label{thm:mainthrmIII}
Let $\Gamma$ be a $\mathbb{D}_2$-symmetric billiard, and denote by $Z_i = \gamma(X_i)$ one of its $\mathbb{D}_2$-symmetric $2$-periodic Birkhoff orbits. In particular, it has constant orbit segment length $L$ and curvature $\kappa$. 
Let $s\geq 2$ and define $p=2s$. When  
$$\kappa L < 2 \cos^2\left( \frac{\pi}{p}\right) ,$$
then $\Gamma$ admits a non-Birkhoff periodic orbit $z_i=\gamma(x_i)$ of type {\rm V}, of minimal period $p$ and rotation number $\frac{1}{2}$. Its  spatiotemporal symmetry group is generated by a single reflection. The lifts $x$ (of $z$) and $X$ (of $Z$) cross exactly $2$ times per period $p$.
 \end{theorem}
\noindent We refer to Figure \ref{fig:lemma54} for a visualization of two orbits of type {\rm V}.
\begin{figure}[h!]
\centering
\hspace{-2cm}\begin{subfigure}{.35\textwidth}
  \centering
  \begin{tikzpicture}
  \node at (0,0) {\includegraphics[width=\linewidth]{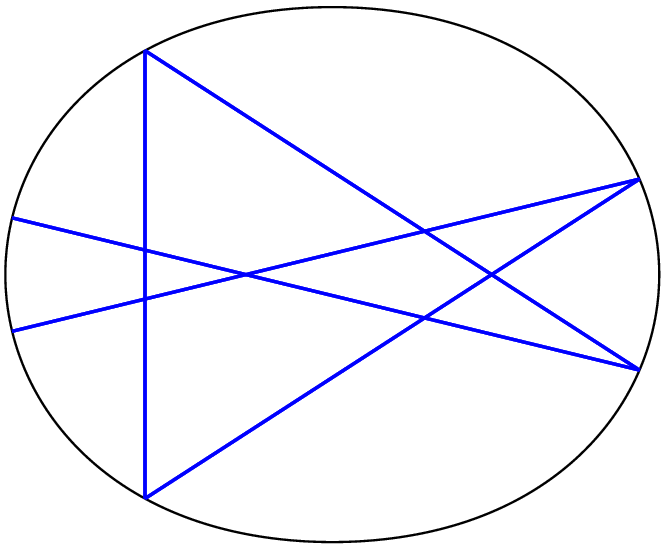}};
\node at (-2.9,0.5) {$1$};
\node at (3.15,-0.9) {$2=10$};
\node at (-1.9,2.1) {$3=9$};
\node at (-1.9,-2.1) {$4=8$};
\node at (3.05,0.9) {$5=7$};
\node at (-2.9,-0.5) {$6$};
  \end{tikzpicture}
  \caption{}
  \label{fig:lemma541}
\end{subfigure}\hspace{2cm}
\begin{subfigure}{.35\textwidth}
  \centering
  \begin{tikzpicture}
  \node at (0,0) {\includegraphics[width=\linewidth]{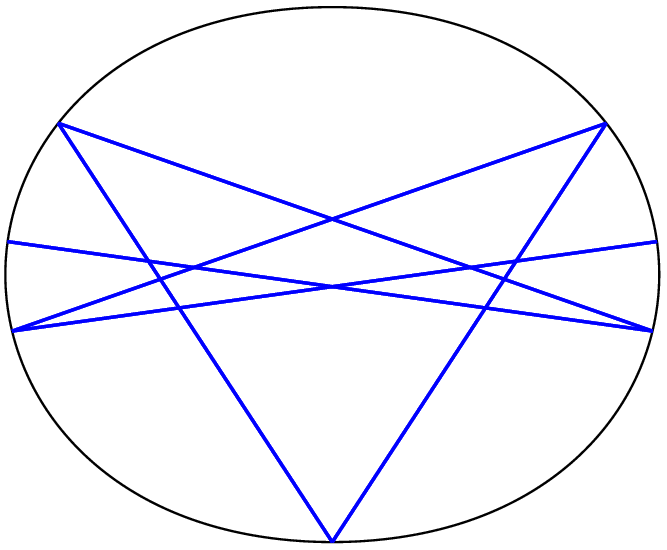}};
\node at (-2.9,0.35) {$1$};
\node at (3.2,-0.6) {$2=12$};
\node at (-2.85,1.4) {$3=11$};
\node at (0,-2.4) {$4=10$};
\node at (2.75,1.4) {$5=9$};
\node at (-3.15,-0.6) {$6=8$};
\node at (2.9,0.35) {$7$};
  \end{tikzpicture}
  \caption{}
  \label{fig:lemma542}
\end{subfigure}
\caption{Visualization of Theorem \ref{thm:mainthrmIII} in a $\mathbb{D}_2$-symmetric billiard of Lima\c con-type  (see Example \ref{ex:limacon}). Billiard parameter given by $\alpha=0.1<\alpha^*(2)= 0.2$.
\textbf{(a)} $\mathbb{D}_1$-symmetric $(10,5)$-periodic non-Birkhoff orbit of type {\rm V}, symmetric with respect to the horizontal axis. Note that $z_{1+i}=z_{1-i}$ and $z_{6+i}=z_{6-i}$; \textbf{(b)} $\mathbb{D}_1$-symmetric $(12,6)$-periodic non-Birkhoff orbit of type {\rm V}, symmetric with respect to the vertical axis. Note that $z_{1+i}=z_{1-i}$ and $z_{7+i}=z_{7-i}$.
\label{fig:lemma54}}
\end{figure}
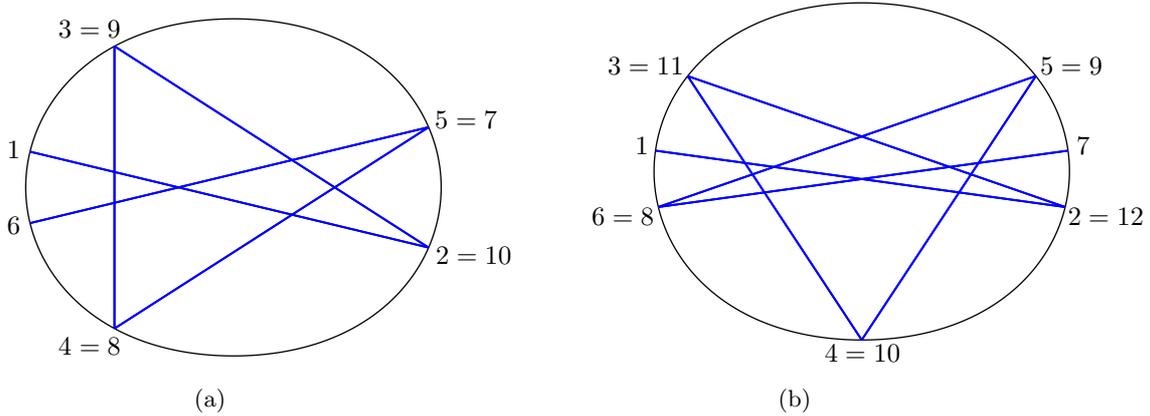
 \begin{proof}      
Let $Z$ be a $\mathbb{D}_2$-symmetric Birkhoff orbit of period $2$ as in the statement of the theorem, choose $s\geq 2$, and define $p=2s, q=s$. As explained in the proof of Theorem \ref{thm:mainthrmII}, there is exactly one reflection $\sigma\in\{S, RS\}$ with the property that $\sigma(Z_i)=Z_{s+i}$. We define $H:=\langle \sigma\rangle$. 
Because $Z$ has period $2$, it holds that $\sigma(Z_i)=Z_{k-i}$ for any $k$ with the same signature as $s$. Let us choose such a $k$.

Theorem \ref{thm:mainthrm}  (applied to $n=2, N=1$ and $m=1$) states that there is a $p$-periodic non-Birkhoff orbit $z$ with spatiotemporal symmetry group $H$ and rotation number $\frac{1}{2}$, with the property that the Aubry diagrams of the lifts $x$ (of $z$) and $X$ (of $Z$) cross twice per period $p$. 

It remains to prove that this orbit is of type {\rm V}.
To do so, recall from the proof of Theorem \ref{thm:mainthrm} that the lift $x\in \mathbb{X}_{p,q}$ of $z$ is the limit (along a subsequence) of a gradient flow line $x(t)$ starting from an initial condition of the form 
$\hat x = X + \varepsilon v\in \mathbb{X}_{p,q}$, where the sequence $v\in \mathbb{X}_{p,0}$ is given explicitly by
\begin{align}\label{eq:vtypeiii}
    v_i =   \sin \left(\frac{2\pi i}{p} -\frac{\pi k}{p} \right)   .
\end{align}
It is clear from \eqref{eq:vtypeiii} that $v_{k-i}+v_i=0$, which implies by case {\it iv)} of Lemma \ref{liftlemma}  that $\sigma(\hat z_i) = \hat z_{k-i}$, where $\hat z\in \Gamma^{\mathbb{Z}}$ is the sequence given by $\hat z_i = \gamma(\hat x_i)$. Because this symmetry is preserved under the gradient flow by Lemma \ref{lem:symmetrypreservation}, we also have that $\sigma(z_i) = z_{k-i}$ for the billiard orbit $z$. In other words, $\sigma$ acts on $z$ time-reversing, as was also the conclusion of Theorem \ref{thm:mainthrm}.
    
However, we also see from \eqref{eq:vtypeiii} that $v_{s+i}+v_i=0$ (because $p=2s$), and by case {\it iii)} of Lemma \ref{liftlemma} this implies that $\hat z_{s+i}=\sigma(\hat z_i)$. This symmetry is also preserved under the gradient flow, so that $z_{s+i}=\sigma(z_i)$. This proves that $\sigma$ acts on $z$ time-preserving as well. This finishes the proof that $z$ is of type {\rm V}. 
\end{proof}
\begin{remark}
Theorems \ref{thm:mainthrmI}, \ref{thm:mainthrmII} and \ref{thm:mainthrmIII} partly generalize a result derived in \cite{casas2011frequency} about the existence of non-Birkhoff periodic orbits in elliptical billiards to arbitrary $\mathbb{D}_2$-symmetric billiards. We recall that the non-Birkhoff orbits inside an elliptical billiard are precisely those orbits that have  a confocal hyperbola as caustic, and consequently these non-Birkhoff orbits all have rotation number $\frac{1}{2}$. 

Consider the ellipse 
$$\mathcal{E}_{a,b}=\left\{ \dfrac{x^2}{a^2}+ \dfrac{y^2}{b^2}
=1\right\} \ \mbox{ with }\ 0<b<a\ .$$ 
According to \cite{casas2011frequency}, this ellipse admits a  non-Birkhoff periodic billiard orbit of minimal period $p$, that crosses the minor axis of the ellipse $M$ times per period $p$, if 
\begin{equation}\label{ineq}
    b^2<a^2\sin^2\left(\frac{\pi M}{2p}\right)\,.
\end{equation}
Our Theorems \ref{thm:mainthrmI},  \ref{thm:mainthrmII} and \ref{thm:mainthrmIII} reproduce this result for the choices $M=p-2$ and $M=p-4$. To see this, note that the minor axis of $\mathcal{E}_{a,b}$ defines a Birkhoff periodic orbit $Z$ of period $2$, and that $M=p-I$, where $I$ is the intersection index between the lift $x$ of the non-Birkhoff periodic orbit $z$ and the lift of the minor axis. This intersection index equals $4$ in Theorem   \ref{thm:mainthrmI} and it equals $2$ in Theorems \ref{thm:mainthrmII} and \ref{thm:mainthrmIII}. Substitution of $M=p-I$ into \eqref{ineq} gives 
\begin{equation}\label{ineqrewritten}
\frac{2b^2}{a^2} < 2\sin^2\left(\frac{\pi M}{2p}\right)=2\sin^2\left(\frac{\pi}{2}-\frac{\pi I}{2p}\right)  = 2\cos^2\left(\frac{\pi I}{2p}\right)\, .
\end{equation}
Now note that the minor axis of $\mathcal{E}_{a,b}$ has length $L=2b$, and that the curvature at its endpoints is $\kappa = \frac{b}{a^2}$. This means that the left hand side of \eqref{ineqrewritten} is precisely the value of the quantity $\kappa L$ for the minor axis.  
We thus recover from \eqref{ineq} the exact conditions of Theorems \ref{thm:mainthrmI}, \ref{thm:mainthrmII} and \ref{thm:mainthrmIII}.

Finally, we remark that symmetric periodic orbits inside ellipses have been classified in \cite{casas2012classification}. The orbits in Figures \ref{fig:D2typeI}, \ref{fig:D2typeII} and  \ref{fig:lemma54} of this paper are similar to those in Table {\rm V} in \cite{casas2012classification}.
\end{remark}

\section{Billiards with infinitely many non-Birkhoff periodic orbits}\label{sec:infinitesection}
\noindent As soon as our results guarantee the existence of \emph{one} non-Birkhoff periodic orbit, they automatically guarantee the existence of infinitely many such orbits. Our final three theorems exploit this observation.
\begin{theorem}\label{thm:1}
Let $\Gamma$ be a $\mathbb{D}_n$-symmetric billiard, and denote by $Z_i = \gamma(X_i)$  one of its $\mathbb{D}_n$-symmetric $n$-periodic Birkhoff orbits of rotation number $\frac{m}{n}$, which has constant curvature $\kappa$ and segment length $L$, as in Theorem \ref{thm:mainthrm}.
    When $\kappa L < 2 \sin \left(\frac{m\pi}{n}\right)$, then $\Gamma$ admits infinitely many non-Birkhoff periodic orbits of rotation number $\frac{m}{n}$, with arbitrarily long minimal periods.
\end{theorem}
\begin{figure}[h!]
\centering
\begin{subfigure}{.45\textwidth}
  \centering
  \includegraphics[width=\linewidth]{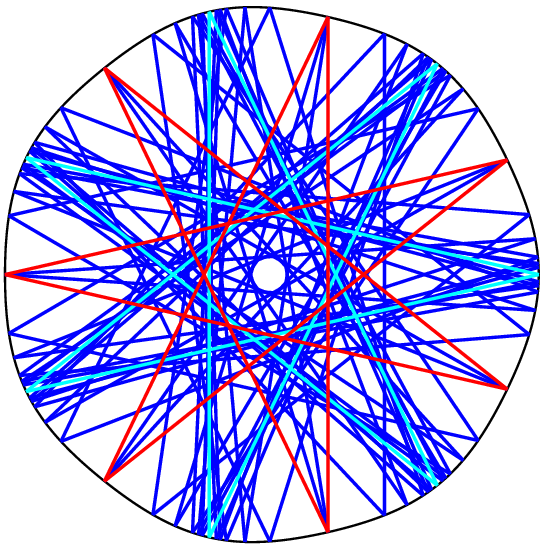}
  \caption{}
  \label{fig:infff1}
\end{subfigure}\hspace{.75cm}
\begin{subfigure}{.45\textwidth}
  \centering
  \includegraphics[width=\linewidth]{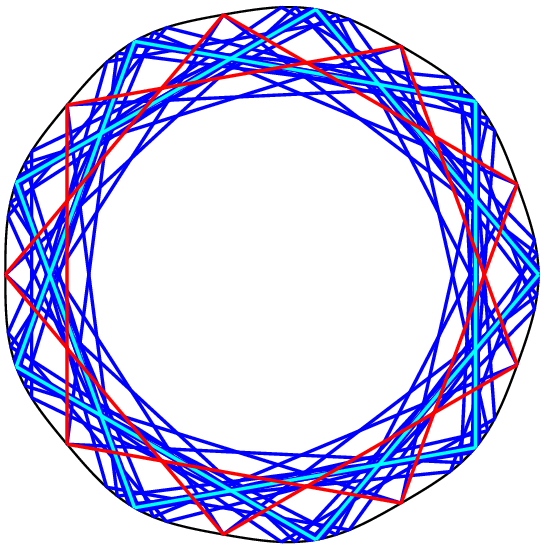}
  \caption{}
  \label{fig:infff2}
\end{subfigure}
\caption{Visualization of Theorem \ref{thm:1}. (a) $(77,33)$-periodic non-Birkhoff orbit in a convex $\mathbb{D}_7$-symmetric billiard of Lima\c con-type  (see Example \ref{ex:limacon}) with parameter $\alpha=0.01<\alpha^*(7)= 0.02$.  The corresponding $(7,3)$-periodic ``short'' Birkhoff orbit is depicted in red and the ``long'' one in cyan; (b) $(99,22)$-periodic non-Birkhoff orbit in a convex $\mathbb{D}_9$-symmetric billiard of Lima\c con-type with parameter $\alpha=0.01<\alpha^*(9) \approx 0.012$.  The corresponding $(9,2)$-periodic ``short'' Birkhoff orbit is depicted in red and the ``long'' one in cyan.
\label{fig:infff}}
\end{figure}
\begin{proof}
Let $s_1, s_2, s_3, \ldots \in \mathbb{N}_{\geq 2}$ be any  sequence of integers with the property that $\gcd(s_i, n)=1$ for all $i\in \mathbb{N}$, and $\lim_{i\to\infty} s_i= \infty$ . For instance, the $s_i$ can be a growing sequence of  prime numbers not equal to $n$. It obviously holds that $\lim_{i\to\infty} \cos^2\left( \frac{\pi}{s_i}\right)=1$, and therefore our assumption on $\kappa L$ implies that there is an $i_0\in \mathbb{N}$ such that 
$$\kappa L< 2 \sin \left(\frac{m\pi}{n}\right) \cos^2\left( \frac{\pi}{s_i}\right)$$ 
for all $i>i_0$. For these $i$, Theorem \ref{thm:mainthrm} (applied to the case $N=n$ and hence $K_i=s_i$) guarantees the existence of a $\mathbb{D}_n$-symmetric non-Birkhoff periodic orbit with rotation number $\frac{m}{n}$ and of minimal period $p_i=s_i n$. This proves the theorem.  We refer to Figure \ref{fig:infff} for a visualization.
\end{proof}

\begin{theorem}\label{thm:2}
Let $\Gamma$ be a $\mathbb{D}_2$-symmetric billiard, and denote by $Z_i = \gamma(X_i)$  one of its $\mathbb{D}_2$-symmetric $2$-periodic Birkhoff orbits of rotation number $\frac{1}{2}$, which has constant curvature $\kappa$ and segment length $L$, as in Theorems \ref{thm:mainthrmI} and \ref{thm:mainthrmII}. When $\kappa L < 2 $, then $\Gamma$ admits infinitely many non-Birkhoff periodic orbits of rotation number $\frac{1}{2}$ and with arbitrarily long minimal periods, of type {\rm I}, of type {\rm II}, and of type {\rm V}.
\end{theorem}
\begin{figure}[h!]
\centering
\begin{subfigure}{.35\textwidth}
  \centering
\includegraphics[angle=90,origin=c,width=\linewidth]{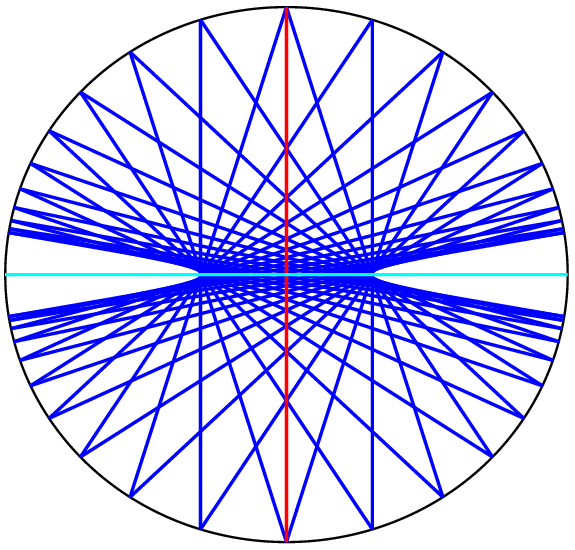}
  \caption{}
  \label{fig:infD211}
\end{subfigure}
\hspace{2cm}
\begin{subfigure}{.35\textwidth}
 \centering  \includegraphics[angle=90,origin=c,width=\linewidth]{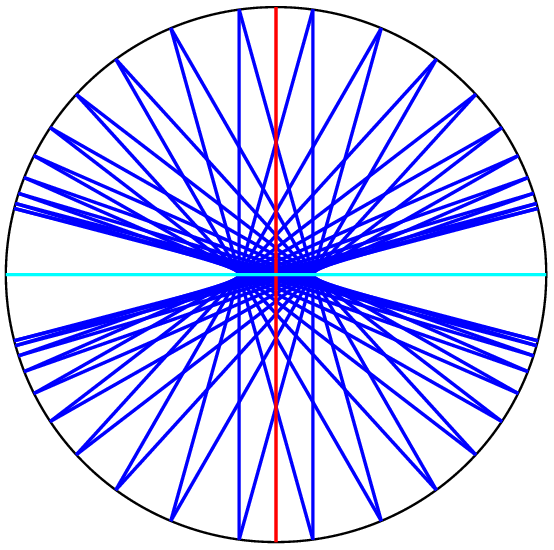}
  \caption{}
 \label{fig:infD22}
\end{subfigure}  \\ 
\centering
\begin{subfigure}{.35\textwidth}
  \centering
  \includegraphics[width=\linewidth]{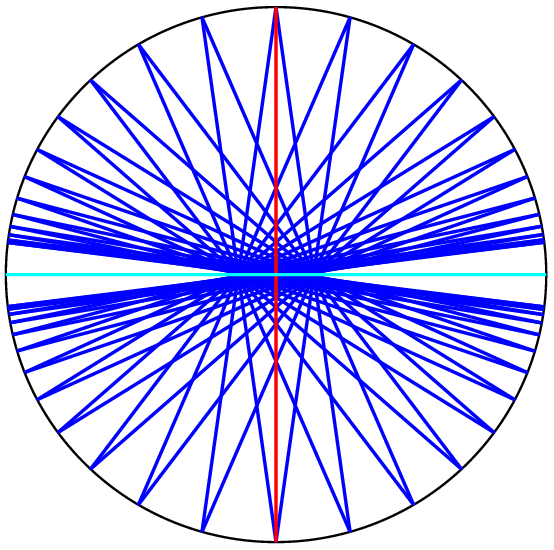}
  \caption{}
  \label{fig:infD23}
\end{subfigure}
\hspace{2cm} 
\begin{subfigure}{.35\textwidth}
  \centering
\includegraphics[width=\linewidth]{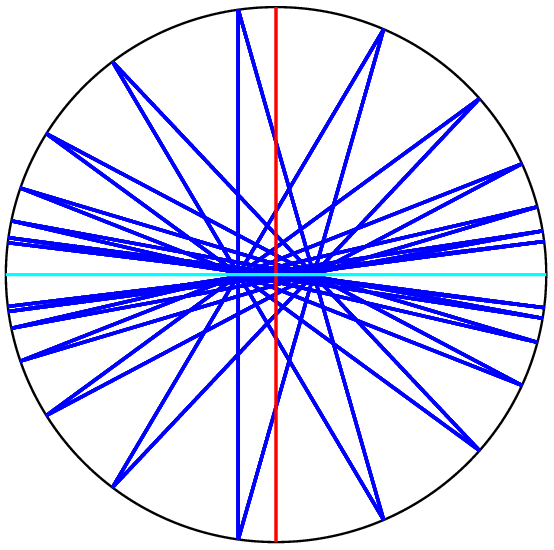}
  \caption{}
  \label{fig:infD25}
\end{subfigure}
\caption{Visualization of Theorem \ref{thm:2} in  convex $\mathbb{D}_2$-symmetric billiards of Lima\c con-type (see Example \ref{ex:limacon}). All billiards are chosen close to circular to improve readability of the figures, but the orbits also exist in ``flatter'' billiards. For clarity, the $(2,1)$-periodic ``short'' Birkhoff orbits are depicted in red and the ``long'' ones in cyan. 
{\bf (a)} $(42,21)$-periodic non-Birkhoff orbit of type I (high-period version of Figure \ref{fig:fig:D2typeIa}). Billiard parameter given by $\alpha=0.025<\alpha^*(2)= 0.2$; 
{\bf (b)} $(40,20)$-periodic non-Birkhoff orbit of type II (high-period version of Figure \ref{fig:fig:D2typeIIa}). Billiard parameter given by $\alpha=0.005$; 
{\bf (c)} $(50,25)$-periodic non-Birkhoff orbit of type {\rm II} (high-period version of Figure \ref{fig:fig:D2typeIIb}). Billiard parameter given by $\alpha=0.005$;  {\bf (d)} $(30,15)$-periodic non-Birkhoff orbit of type {\rm V} (high-period version of Figure \ref{fig:lemma541}). Note that the orbit admits only one reflection symmetry.  Billiard parameter given by $\alpha=0.005$.\label{fig:infD2}}
\end{figure}
\begin{proof}
    Obvious from Theorems \ref{thm:mainthrmI},  \ref{thm:mainthrmII} and \ref{thm:mainthrmIII}. We refer to Figure \ref{fig:infD2} for a visualization.
\end{proof}
\noindent The circular billiard clearly does not admit any non-Birkhoff orbits. However, our final theorem shows that arbitrarily small perturbations of the circular billiard possess infinitely many non-Birkhoff periodic orbits of arbitrary rational rotation number.
\begin{theorem}\label{thm:3}
    Let $0< \frac{m}{n}<1$ be a rational number. In the analytic topology, any open neighborhood of the circular billiard contains a billiard with infinitely many non-Birkhoff periodic orbits of rotation number $\frac{m}{n}$.
\end{theorem}
\begin{figure}[h!]
\centering
\begin{subfigure}{.45\textwidth}
  \centering
\includegraphics[width=\linewidth]{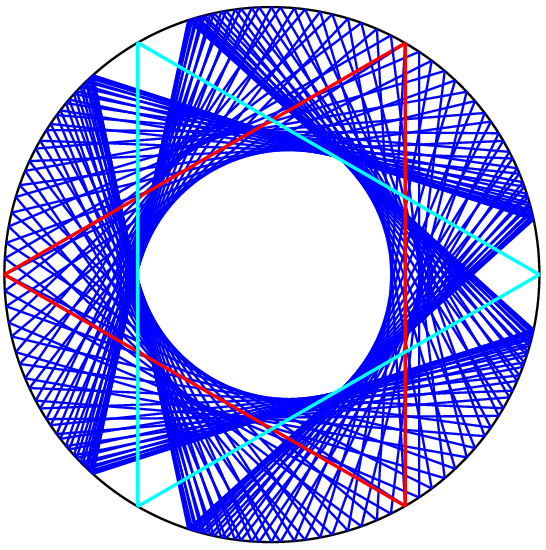}
  \caption{}
  \label{fig:infinite1}
\end{subfigure}\hspace{.75cm}
\begin{subfigure}{.45\textwidth}
  \centering
  \includegraphics[width=\linewidth]{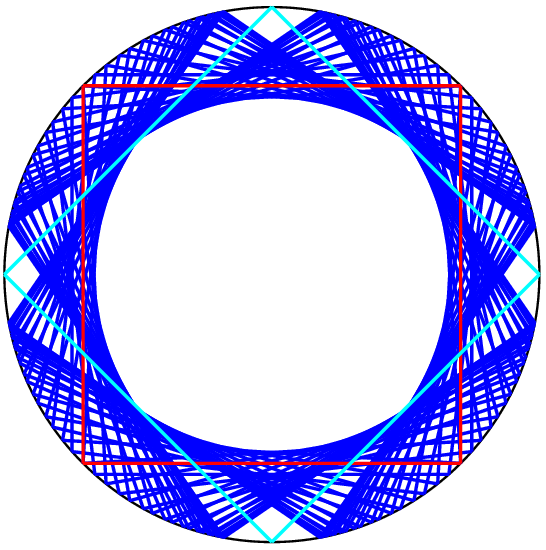}
  \caption{}
  \label{fig:infinite2}
\end{subfigure}
  \caption{Visualizations of Theorem \ref{thm:3}: \textbf{(a)} $(165,55)$-periodic non-Birkhoff orbit in a convex $\mathbb{D}_3$-symmetric billiard of Lima\c con-type (see Example \ref{ex:limacon}). Billiard parameter given by $\alpha=0.001<\alpha^*(3)= 0.1$. For clarity, the $(3,1)$-periodic ``short'' Birkhoff orbit is depicted in red and the ``long'' one in cyan; \textbf{(b)} $(156,39)$-periodic non-Birkhoff orbit in a convex $\mathbb{D}_4$-symmetric billiard of Lima\c con-type. Billiard parameter given by $\alpha=0.001<\alpha^*(4)\approx 0.0588$. For clarity, the $(4,1)$-periodic ``short'' Birkhoff orbit is depicted in red and the ``long'' one in cyan; Recall Figure \ref{fig:intro4} for a third visualization of the theorem.\label{fig:infinite}}
\end{figure}
\begin{proof}
Let $n, m$ be co-prime with $1\leq m \leq n-1$, and let $\alpha > 0$. We consider the Lima\c con-type billiard curve $\Gamma_{\alpha}$ given by the embedding 
   $\gamma_{\alpha}(x)= r_{\alpha}(x) \left( \cos(2\pi x), \sin(2\pi x) \right) $ where 
    $$r_{\alpha}(x) = 1 + \alpha \cos(2\pi n x)\, .$$
Because $r_{\alpha}(-x)=r_{\alpha}(x)$ and  $r_{\alpha}(x+\frac{1}{n})=r_{\alpha}(x)$, we have that $\Gamma_{\alpha}$ is $\mathbb{D}_n$-symmetric. It follows from Proposition \ref{prop:convexityprop}  that 
$\Gamma_{\alpha}$ is convex for sufficiently small values of $\alpha$, namely when $0\leq \alpha \leq \alpha^*(n)=\frac{1}{1+n^2}$. 

Next,  we observe that the analytic norm $$\|\gamma_{\alpha} - \gamma_0\|_{\rho} := \sup_{x\in \mathbb{R} + i [-\rho, \rho] }\| \gamma_{\alpha}(x) - \gamma_0(x)\| =  \sup_{x\in \mathbb{R} + i [-\rho, \rho]} | \alpha \cos(2\pi n x) | \leq \alpha   e^{2\pi n \rho}  $$ can be made arbitrarily small by choosing $\alpha$ small (depending on $n\in \mathbb{N}$ and $\rho>0$). This shows that $\Gamma_{\alpha}$ can be chosen arbitrarily close to the circle $\Gamma_0$ in any analytic norm $\| \cdot \|_{\rho}$. 
    
The final thing to note is that $\Gamma_{\alpha}$ is tangent to the circle of radius $1-\alpha$ at any of the points $\gamma_{\alpha}\left(\frac{j}{n} + \frac{1}{2n}\right)=  (1-\alpha)\left(\cos\left(\frac{2 \pi j}{n} + \frac{\pi}{n}\right)\!,\, \sin\left(\frac{2 \pi j}{n} + \frac{\pi}{n}\right)\right)$. This implies that $\Gamma_{\alpha}$ admits a Birkhoff periodic orbit of rotation number $\frac{m}{n}$ reflecting at these points, just like the circular billiard of radius $1-\alpha$ does. This orbit is in fact given by $$Z_i = \gamma_{\alpha}\left(\frac{m}{n}i + \frac{1}{2n}\right) .$$ 
The length of all the orbit segments of this Birkhoff orbit is equal to 
$$L = \left\|Z_{i+1} - Z_i \right\|
= 2  (1-\alpha)\sin\left(\frac{m \pi}{n} \right) .$$ 
It is also not hard to see that, for $0<\alpha < \frac{1}{1+n^2}$, the curvature of $\Gamma_{\alpha}$ at any of the points $Z_i$ is strictly less than the curvature $\frac{1}{1-\alpha}$ of the circle of radius $1-\alpha$. Thus, $\kappa L < 2\sin(m\pi/n)$. The result now follows from Theorem \ref{thm:1}. We refer to Figure \ref{fig:infinite} for a visualization.
\end{proof}
\section{Conclusion and discussion}\label{sec:concl}
\noindent In this article, we established sufficient conditions for the existence of non-Birkhoff periodic orbits in symmetric planar convex billiards. We also classified the spatiotemporal symmetries of periodic orbits that arise in this context. Our results show that for any $\mathbb{D}_n$-symmetric billiard  $\Gamma$ and any subgroup $H\subset \mathbb{D}_n$ containing a reflection, $\Gamma$ admits $H$-symmetric non-Birkhoff orbits of certain minimal periods, granted that the curvature at the points of a $\mathbb{D}_n$-symmetric Birkhoff orbit is sufficiently small, cf. Theorem \ref{thm:mainthrm}. Theorems \ref{thm:mainthrmII}  and \ref{thm:mainthrmIII}  provide similar existence results for non-Birkhoff orbits in $\mathbb{D}_2$-symmetric billiards that are not covered by Theorem \ref{thm:mainthrm}. Our theorems apply in particular to elliptical billiards, for which they provide  alternative proofs of known results. We also showed that as soon as a our theorems guarantee the existence of one non-Birkhoff periodic orbit, then they guarantee that infinitely many such  orbits exist, with arbitrarily long periods, cf. Theorems \ref{thm:1}, \ref{thm:2}
and \ref{thm:3}. We gave various numerical illustrations of our results. In Appendix \ref{app:code}, we provide a brief introduction to the Matlab code that was used to produce these illustrations. We have made this code available on GitHub.

As was highlighted in the introduction, our proofs  crucially rely only on the discrete symmetry and on the monotone variational structure of the billiard problem. Thus we expect that many of our results can be  generalized to other problems with such structure. We anticipate in particular that our methodology applies to billiard systems different from the classical planar type, such as outer billiards \cite{tabachnikov2005geometry}, symplectic billiards \cite{albers2018introducing}, magic billiards \cite{dragovic2024magic}, and possibly to polygonal billiards \cite{tabachnikov2005geometry}. We refer to \cite{tabachnikov2005geometry}, and references therein, and to a  recent survey by Schwartz in \cite{schwartz2021survey}, for an overview of some of these  generalized billiards. We also mention an extension of the classical planar billiard problem to bounded geodesically convex domains in surfaces of constant curvature, cf. \cite{dos2017periodic}. 

Extending our results to general monotone variational problems with discrete symmetries would require various modifications of the proofs given in this paper. For instance, two outer billiards or two  symplectic billiards that are related by an affine transformation have conjugate dynamics 
\cite{albers2018introducing,tabachnikov2005geometry}. For outer and symplectic billiards, the eigenvalues of the Hessian of the periodic action will  therefore not depend on the curvature of the billiard in the same way as for classical billiards. On the other hand, unlike in the classical billiard problem, the generating function of the symplectic billiard is smooth. This implies that the gradient flow cannot develop singularies, which simplifies some of the proofs in this paper. It would be interesting to investigate if our methods can be used to prove the existence of symmetric non-Birkhoff periodic orbits in any of the aforementioned generalized billiard problems.
To the best of our knowledge, non-Birkhoff periodic orbits have not yet been found in such billiards. 

\changes{We conclude by pointing out a surprising relation between our work and recent results in \cite{bialy2025}. Let us denote by $\overline{\kappa}:=\fint_{\Gamma} \kappa(s) ds$ the average curvature of a convex billiard $\Gamma$, and by $\overline L_{\rm max}:=\frac{1}{n}\sum_{i=1}^n\|Z_{i+1}-Z_i\|$  the average segment length of an $(m,n)$-periodic length-maximizing Birkhoff orbit $Z$ of  $\Gamma$. 
The results in \cite{bialy2025} imply that these quantities satisfy an ``isoperimetric inequality'' of the form $\overline \kappa \overline L_{\rm max} \geq 2\sin\left(\frac{m\pi}{n} \right)$, with equality holding if and only if $\Gamma$ is a circle. This inequality provides an intriguing ``opposite'' of the condition under which Theorem \ref{thm:1} is applicable. }
\\ \mbox{} 
\\
\noindent \textbf{Acknowledgments.} The work of C.O. was partially supported by the Engineering and Physical Sciences Research Council (EPSRC) grant [EP/W522569/1]. B.R. acknowledges the hospitality and financial support of the Sydney Mathematical Research Institute (SMRI). M.S. was partially supported by the Italian Ministry of University and Research (MUR) through the PRIN 2020 project ``Integrated Mathematical Approaches to Socio-Epidemiological Dynamics'' (No. 2020JLWP23, CUP: E15F21005420006). M. S. is a member and acknowledges the support of {\it Gruppo Nazionale di Fisica Matematica} (GNFM) of {\it Istituto Nazionale di Alta Matematica} (INdAM). \\

\noindent \textbf{Data Availability Statement} Data sharing not applicable to this article as no
datasets were generated or analyzed during the current study.\\

\noindent \textbf{Declarations}\\

\noindent \textbf{Conflict of interest} The authors have no conflict of interest to declare that is relevant to the content of this article.\\

\appendix

\section{Appendix: convexity of Lima\c{c}on-type curves}\label{app:convexity}
\noindent 

\begin{proposition}\label{prop:convexityprop}
   The Lima\c{c}on-type curve that is the image of the map $\gamma_{\alpha}: \mathbb{R}\rightarrow \mathbb{R}^2$ defined by 
   \begin{equation}\label{eq:limaconformulaappendix}
   \gamma_{\alpha}(x)=r_{\alpha}(x)(\cos (2\pi x), \sin (2 \pi x)) \ \mbox{in which}\ r_{\alpha}(x) = 1+ \alpha \cos(2\pi n x)\ \mbox{and}\ |\alpha|<1\, ,
 \end{equation}
is $\mathbb{D}_n$-symmetric. It bounds a strictly convex domain if and only if
 $$
| \alpha|\leq \alpha^*(n):=\dfrac{1}{1+n^2}\, .
$$
\end{proposition}
\begin{proof}
It is clear that $\gamma_{\alpha}$ parametrizes a simple closed curve for $|\alpha|<1$. For $\alpha=0$ this curve is a circle; otherwise its symmetry group is $\mathbb{D}_n=\langle R, S\rangle$, because $\gamma_{\alpha}$ satisfies the $\mathbb{D}_n$-equivariance conditions \eqref{equivarianceofgamma}. 

A  simple closed $C^2$-smooth planar curve, which is  parametrized counterclockwise by an immersion $\gamma:\mathbb{R}\rightarrow \mathbb{R}^2$, bounds a convex domain if and only if 
$$\det(\gamma'(x), \gamma''(x))\geq 0\ \mbox{for all}\ x\in \mathbb{R}\, .$$ 
When $\gamma$ takes the form $\gamma(x)=r(x)(\cos (2\pi x), \sin (2 \pi x))$, then this determinant is equal to 
\begin{equation}\nonumber 
\det(\gamma'(x), \gamma''(x)) = 8\pi^3 r(x)^2+4\pi (r'(x))^2-2\pi r(x)r''(x)\, .
\end{equation}
For $r(x) = r_{\alpha}(x) = 1 + \alpha \cos(2\pi n x)$ this expression reduces to  
\begin{equation} \label{eq:ineq-det}
\det(\gamma'_{\alpha}(x), \gamma''_{\alpha}(x)) = 8\pi^3
 \left( A_1(x)\alpha^2+A_2(x)\alpha+1 \right) ,
\end{equation}
in which   
$$A_1(x)=\cos^2(2\pi nx)(1-n^2)+2n^2 \ \mbox{and}\
A_2(x)=\cos(2\pi nx)(2+n^2)\, .$$ 
For convenience, assume from now on that $\alpha \geq 0$ (the proof is similar when $\alpha<0$). Note that the functions $A_1$ and $A_2$ both attain their minimum when $\cos(2\pi nx)=-1$, that is, at $x=\frac{1}{2n}+\tfrac{A}{n}\  (A\in \mathbb{Z})$,  their minimum values being $A_1\left(\frac{1}{2n}\right) = 1+n^2$ and $A_2\left(\frac{1}{2n}\right) = -(2+n^2)$.  Thus, as $\alpha\geq 0$, the expression in \eqref{eq:ineq-det} is nonnegative for all $x\in \mathbb{R}$ precisely when 
\begin{equation}
    \label{eq:quadraticinequality}
(1+n^2)\alpha^2 - (2+ n^2)\alpha + 1 \geq 0\, .
\end{equation}
The solutions to this quadratic inequality are given by $\alpha \leq \alpha^*(n) := \frac{1}{1+n^2}$ and $\alpha \geq 1$ (but recall that we excluded the latter case.)
This (together with a similar argument for $\alpha<0$) proves that the curves $\gamma_{\alpha}$ defined in \eqref{eq:limaconformulaappendix} are convex precisely when $|\alpha| \leq \alpha^*(n)$. 

We remark that the left-hand side of \eqref{eq:quadraticinequality} is strictly positive  when $ |\alpha|<\alpha^*(n)$, so $\gamma_{\alpha}$ bounds a strictly convex domain for these values of $\alpha$. For $\alpha=\pm \alpha^*(n)$ the expression in \eqref{eq:ineq-det} only vanishes at finitely many points. Hence also for $\alpha=\pm \alpha^*(n)$, the curve defined by $\gamma_{\alpha}$ bounds a strictly convex domain. 
\end{proof}

\section{Appendix: structure of the numerical code}\label{app:code}
\noindent The GitHub repository \href{https://github.com/MattiaSensi/BilliardOrbitFinder}{BilliardOrbitFinder} contains three commented examples of the Matlab code that was used to produce the figures in this manuscript. More precisely, the repository contains one code with a $\mathbb{D}_4$-symmetric billiard of Lima\c con-type (see Example \ref{ex:limacon}) and two codes with elliptical billiards. Other parametrized billiards can easily be implemented. 
The main ingredients of the code are the following:
\begin{enumerate}
    \item Provide a parametrization of the boundary of the billiard table of the form $\gamma(x)=(\gamma_1(x),\gamma_2(x))$ (for $x\in \mathbb{R}$),  as well as a parametrization of its derivative $\gamma'(x)=(\gamma'_1(x),\gamma'_2(x))$.
    \item Provide a sequence $(\hat{x}_1, \ldots, \hat x_p) \in \mathbb{R}^{p}$, corresponding to positions $\hat z_i = \gamma(\hat x_i)$ on the boundary of the billiard table. This sequence serves as initial condition for the gradient flow. Any desired symmetry of the sequence is imposed in the form of an affine relation between the $\hat x_i$ (see Lemma \ref{liftlemma}).
    \item Starting from the initial conditions $x_i(0)=\hat x_i$, numerically integrate the system of ODEs
    \begin{equation}
\label{eq:sys_of_ODE}\dot{x}_i= F_i(x) = F^-(x_{i-1},x_i)+F^+(x_i,x_{i+1}), \quad i=1,\dots,p\, ,
    \end{equation}
     as given in \eqref{eq:gradientflow}, \eqref{eq:F-formula} and \eqref{eq:F+formula}, forward in time. We define $x_0:=x_p$ and $x_{p+1}:=x_1$  throughout the integration. The symmetries imposed on the initial conditions $\hat x_i$ are preserved by imposing that the corresponding affine relations continue to hold for the $x_i(t)$. Equations \eqref{eq:sys_of_ODE} are integrated forward in time until the solution appears to stabilize at a stationary point.
\end{enumerate}

\bibliographystyle{abbrvurl} 
\bibliography{Bibliography}


\end{document}